\documentclass[12pt]{amsart}%
\usepackage{verbatim}
\usepackage{latexsym}
\usepackage{amsmath,enumitem}
\usepackage{amsthm}
\usepackage{amsfonts,array}
\usepackage{amssymb}
\usepackage{mathtools}
\usepackage[table]{xcolor}
\usepackage{xfrac}
\usepackage{hyperref}
\usepackage{graphicx}%
\usepackage{tikz,tikz-cd}
\usepackage{mathrsfs}
\usepackage{makecell}
\usetikzlibrary{decorations.markings}
\graphicspath{ {./} }

\newcommand{\C}{\mathbb{C}}

\newcommand{\R}{\mathbb{R}}
\newcommand{\Q}{\mathbb{Q}}

\newtheorem{thm}{Theorem}
\newtheorem*{thm*}{Theorem}

\newtheorem{prop}[thm]{Proposition}

\newtheorem{rem}[thm]{Remark}

\newcommand{\dd}{\;\mathrm{d}}%differential
\newcommand{\m}{\mathrm{m}}
\newcommand{\M}{\mathrm{M}}
\newcommand{\ol}[1]{\overline{#1}}

\newcommand{\Li}{\mathrm{Li}} %Polylogarithm
\newcommand{\I}{\mathrm{I}} %Hyperlogarithm
\DeclarePairedDelimiter\floor{\lfloor}{\rfloor}

\renewcommand{\pmod}[1]{{\ifmmode\text{\rm\ (mod~$#1$)}\else\discretionary{}{}{\hbox{ }}\rm(mod~$#1$)\fi}}

\definecolor{LightGreen}{rgb}{0.75,0.95,0.65}
\definecolor{LightCyan}{rgb}{0.88,1,1}
\definecolor{LightGray}{gray}{0.7}
\newcolumntype{a}{>{\columncolor{LightCyan}}c}
\newcolumntype{b}{>{\columncolor{LightGreen}}c}
\newcolumntype{d}{>{\columncolor{LightGray}}c}
\newcolumntype{C}{>{$}c<{$}}%for tables with column of alignment c in math mode
\newcolumntype{L}{>{$}l<{$}}
\newcolumntype{R}{>{$}r<{$}}

\begin{document}

\title{The Mahler measure of a family of polynomials with arbitrarily many variables}
\begin{abstract}
We present an exact formula for the Mahler measure of an infinite family of polynomials with arbitrarily many variables. The formula is obtained by manipulating the integral defining the Mahler measure using certain transformations, followed by an iterative process that reduces this computation to the evaluation of certain polylogarithm functions at sixth roots of unity. This yields values of the Riemann zeta function and the Dirichlet $L$-function associated to the character of conductor 3.
\end{abstract}

\author{Siva Sankar Nair}
\address{Siva Sankar Nair:  D\'epartement de math\'ematiques et de statistique,
                                    Universit\'e de Montr\'eal,
                                    CP 6128, succ. Centre-ville,
                                     Montreal, QC H3C 3J7, Canada}\email{siva.sankar.nair@umontreal.ca}
\maketitle
%\tableofcontents

\section{Introduction}

For a non-zero rational function in $n$ variables, $P \in \C(x_1,\dots,x_n)$,  the (logarithmic) Mahler measure of $P$ is defined to be 
\begin{equation*}
 \m(P):=\frac{1}{(2\pi i)^n}\int_{\mathbb{T}^n}\log|P(x_1,\dots, x_n)|\frac{\mathrm{d}  x_1}{x_1}\cdots \frac{\mathrm{d}  x_n}{x_n},
\end{equation*}
where the integration path is taken along the unit $n$-torus $\mathbb{T}^n=\{(x_1,\dots,x_n)\in \mathbb{C}^n : |x_1|=\cdots=|x_n|=1\}$ with respect to the Haar measure. 

While Jensen had already studied this integral for single-variable holomorphic functions in the late $19^\text{th}$ century, it was in Lehmer's work \cite{Le} that this quantity first appeared in the context of single-variable polynomial functions.  One may note that using Jensen's formula, the Mahler measure of a univariate polynomial can be expressed in terms of the absolute value of its roots that lie outside the unit circle.  The generalization to multivariate polynomials was made by Mahler \cite{Mah} when he studied the quantity $\M(P)=\exp(\m(P))$ in his work related to polynomial heights.  In the 1980s, it was Smyth \cite{B1,S1} who first observed a relation between Mahler measures and values of certain $L$-functions when he showed the following equalities:
\begin{equation}\label{eq:level2}
\m(x+y+1) = \frac{3 \sqrt{3}}{4 \pi} L(\chi_{-3},2)=L'(\chi_{-3},-1),
\end{equation} 
\begin{equation}\label{eq:level3}
\m(x+y+z+1)=\frac{7}{2\pi^2} \zeta(3)=-14\cdot\zeta'(-2),
\end{equation}
where  $L(\chi_{-3},s)$ is the Dirichlet $L$-function in the character of conductor 3 and $\zeta(s)$ is the Riemann zeta function. 

Interest in multivariate Mahler measures grew following these results, and several relations between Mahler measures and special values of $L$-functions associated to various arithmetic objects were observed.  For instance,  Deninger  \cite{Deninger} and Boyd \cite{Bo98}  conjectured the following relation involving an elliptic curve, which was later proven by Rogers and Zudilin \cite{RZ14}:
\[\m\left(x+\frac{1}{x}+y+\frac{1}{y}+1\right)=\frac{15}{4\pi^2}L(E_{15a8},2)=L'(E_{15a8},0).\]
Here $L(E_{15a8},s)$ denotes the $L$-function associated to the elliptic curve defined by the Weierstrass equation
\[
E_{15a8}:y^2+xy+y=x^3+x^2,
\]
with Cremona label $15a8$.  While much remains unknown regarding this mysterious link between Mahler measures and $L$-functions, there has been remarkable work that sheds light on this phenomenon.  In the 1990s, Deninger \cite{Deninger} related Mahler measures to certain regulator values from $K$-theory. In view of Beilinson's conjectures, these regulators are expected to appear as values of $L$-functions, as well as polylogarithm functions evaluated at algebraic numbers. In turn, certain combinations of these polylogarithms yield values of the Riemann zeta function and various other Dirichlet $L$-functions, explaining their presence in  Mahler measure computations.  For example, Smyth's first example \eqref{eq:level2} is the Mahler measure of a two-variable polynomial expressed in terms of a dilogarithm, and his second example \eqref{eq:level3} is that of a three-variable polynomial in terms of a trilogarithm. On a slightly different but related note, one may view this in context of Zagier's grand conjecture \cite{Zagier1991} -- the value $\zeta_K(n)$ of the Dedekind zeta function corresponding to a number field $K$ evaluated at an integer $n>1$ can be expressed as the determinant of a matrix whose entries are linear combinations of polylogarithms evaluated at certain elements of $K$.  Once again, this has connections to the (Borel) regulator, as seen in the proofs for the case $n=3$ by Goncharov \cite{Goncharov} and $n=4$ by Goncharov and Rudenko \cite{goncharov2022motivic}.

A particularly interesting result was given by Lalín \cite{nvar} involving rational functions with arbitrarily many variables whose Mahler measure can be computed explicitly.  These formulae arise by evaluating polylogarithms at the fourth root of unity which yield values of the Dirichlet $L$-function $L(\chi_{-4},s)$ associated to the character of conductor 4.  This is one of the few results in the literature where the exact Mahler measure of multivariable polynomials with arbitrarily many variables has been calculated (see \cite{LNR} and \cite{smyth2003explicit} for other examples). We explain the result in greater detail. Let
\begin{align*}
R_m(z_1,\dots, z_m,y):=y+\left(\frac{1-z_1}{1+z_1}\right)\cdots\left(\frac{1-z_m}{1+z_m}\right),
\end{align*}
and for $a_1, \dots a_m \in \C$, define the elementary symmetric polynomials
\[s_\ell(a_1,\dots, a_m)=\left\{
\begin{array}{ll}
1& \mbox{if }\ell=0,\\
\sum_{i_1<\cdots<i_\ell}a_{i_1}\cdots a_{i_\ell}& \mbox{if }0<\ell\leq m,\\
0& \mbox{if }\ell>m.
\end{array}
\right.\]
We have an explicit formula for the Mahler measure of the rational function $R_m$:
\begin{thm*}\label{thm:nvar} (\cite{nvar,LL}) For $n \geq 1$,
 \[\m\left(R_{2n} \right) =\sum_{h=1}^n  \frac{s_{n-h}(2^2, 4^2, \dots, (2n-2)^2)}{(2n-1)!} \left(\frac{2}{\pi}\right)^{2h}\mathcal{A}(h)
,\]
where
\[\mathcal{A}(h):=(2h)!\left(1-\frac{1}{2^{2h+1}}\right) \zeta(2h+1).\]

For $n \geq 0$,
\[ \m\left( R_{2n+1}\right)= \sum_{h=0}^n \frac{s_{n-h}(1^2,3^2, \dots, (2n-1)^2)}{(2n)!}  \left(\frac{2}{\pi}\right)^{2h+1} \mathcal{B}(h),\]
where
\[\mathcal{B}(h):=(2h+1)!L(\chi_{-4}, 2h+2).\]
\end{thm*}

The purpose of this article is to express the Mahler measure of another family of polynomials with arbitrarily many variables in terms of zeta values and certain $L$-values.  The method of obtaining these results is similar to that of Lalín \cite{nvar} applied to a more complicated family of functions. However, the polynomials we consider have algebraic coefficients rather than integral. The interesting outcome is that in this case,  the calculations involve the polylogarithm evaluated at sixth roots of unity and we get values of the $L$-function $L(\chi_{-3},s)$ associated to the character of conductor 3. This is the first example of an infinite family of polynomials with arbitrarily many variables yielding Mahler measures that involve special values of the $L$-function $L(\chi_{-3},s)$.  In particular, we consider the rational function
\[
Q_n(z_1,\dots,z_n,y)=y+\left(\frac{z_1+\omega}{z_1+1}\right)\cdots\left(\frac{z_n+\omega}{z_n+1}\right)
\]
%\[
%Q_n(z_1,\dots,z_n,y)=y+\left(\frac{\ol\omega z_1+\omega}{z_1+1}\right)\cdots\left(\frac{\ol\omega z_n+\omega}{z_n+1}\right),
%\]
where 
\[
\omega=e^{2\pi i/3}=-\frac12+\frac{\sqrt3i}2
\]
is a third root of unity.  We let \begin{equation*}
\mathcal{C}(h)=(2h)!\left(1-\frac1{3^{2h+1}}\right)\left(1-\frac1{2^{2h+1}}\right)\zeta(2h+1),
\end{equation*}
and
\begin{equation*}
\mathcal{D}(h)=(2h+1)!\left(1+\frac1{2^{2h+2}}\right)L(\chi_{-3},2h+2).
\end{equation*}
Then 
\begin{thm}\label{thm:main} For $n\ge1$,
\begin{equation*}
\m(Q_{2n})=\frac{2}{12^n}\Bigg(\sum_{h=1}^n  a_{n,h-1}\left(\frac{3}{\pi}\right)^{2h}\mathcal{C}(h)\;+\;\sum_{h=0}^{n-1}b_{n,h}\left(\frac{3}{\pi}\right)^{2h+1}\mathcal{D}(h)\Bigg),
\end{equation*}
and for $n\ge0$ we have
\begin{equation*}
\m(Q_{2n+1})=\frac{1}{{12}^n}\Bigg(\frac{1}{\sqrt3}\sum_{h=1}^n  c_{n,h-1}\left(\frac{3}{\pi}\right)^{2h}\mathcal{C}(h)\;+\;\frac{1}{\sqrt3}\sum_{h=0}^{n}d_{n,h}\left(\frac{3}{\pi}\right)^{2h+1}\mathcal{D}(h)\Bigg),
\end{equation*}
where the coefficients $a_{r,s}\,,b_{r,s}\,,c_{r,s},d_{r,s}$ are real numbers given recursively by equations \eqref{eq:aij}-\eqref{eq:dij} starting from the initial value $d_{0,0}=1$. %,a_{1,0}=2, b_{1,0}=\frac2{\sqrt3},c_{1,0}=2\sqrt3,d_{1,0}=4$ and $d_{1,1}=2$. 
\end{thm}
Using the respective functional equations, we may also express these Mahler measures as $\Q$-linear combinations of derivatives of the Riemann zeta function evaluated at negative even integers and derivatives of the $L$-function $L(\chi_{-3},s)$ evaluated at negative odd integers (see equations \eqref{eq:dereven} and \eqref{eq:derodd} in Section \ref{subsec:derivatives}). The first few examples of this family are given in Table \ref{table:arbitrary}.

We remark that using an alternate method, we can write the coefficients $a_{r,s}\,,b_{r,s}\,,c_{r,s},d_{r,s}$ in a more explicit form. This method, described in Section \ref{sec:alternat}, uses polynomials that arise as coefficients in certain power series, and whose explicit expressions may be obtained using standard convolution techniques. The first two polynomials in this family have been computed (equations \eqref{eq:am_expl} and \eqref{eq:bm_expl}).
\renewcommand{\arraystretch}{3.5}
\renewcommand{\tabcolsep}{0.5cm}
\begin{table}
\caption{First few examples of the family $Q_n$}\label{table:arbitrary}
\begin{equation*}
\begin{array}{|c|l|}
\hline
\m\left(y+\left(\frac{ z_1+\omega}{z_1+1}\right)\right)&\frac{5\sqrt3}{4 \pi}\,L(\chi_{-3},2)=\frac{5}{3}\,L'(\chi_{-3},-1)\\
\hline
\m\left(y+\left(\frac{ z_1+\omega}{z_1+1}\right)\left(\frac{ z_2+\omega}{z_2+1}\right)\right)&\begin{multlined}\textstyle\frac{91}{18\pi^2}\zeta(3)+\frac{5}{4\sqrt3 \pi}\,L(\chi_{-3},2)\\
\textstyle=-\frac{182}9\zeta'(-2)+\frac59L'(\chi_{-3},-1)\end{multlined}\\[4mm]
\hline
\m\left(y+\left(\frac{ z_1+\omega}{z_1+1}\right)\cdots\left(\frac{ z_3+\omega}{z_3+1}\right)\right)&\begin{multlined}\textstyle\frac{91}{36\pi^2}\zeta(3)+\frac{5}{4\sqrt3 \pi}\,L(\chi_{-3},2)+\frac{153\sqrt3}{16\pi^3}L(\chi_{-3},4)\\
\textstyle=-\frac{91}9\zeta'(-2)+\frac59L'(\chi_{-3},-1)-\frac{17}{18}L'(\chi_{-3},-3)\end{multlined}\\[4mm]
\hline
\m\left(y+\left(\frac{ z_1+\omega}{z_1+1}\right)\cdots\left(\frac{ z_4+\omega}{z_4+1}\right)\right)&\begin{multlined}\textstyle\frac{91}{36\pi^2}\zeta(3)+\frac{3751}{108\pi^4}\zeta(5)+\frac{35}{36\sqrt3 \pi}\,L(\chi_{-3},2)+\frac{51\sqrt3}{8\pi^3}L(\chi_{-3},4)\\
\textstyle=-\frac{91}9\zeta'(-2)+\frac{3751}{81}\zeta'(-4)+\frac{35}{81}L'(\chi_{-3},-1)-\frac{17}{27}L'(\chi_{-3},-3).\end{multlined}\\[4mm]
\hline
\end{array}
\end{equation*}
\end{table}

To evaluate the above Mahler measures, we first use contour integration to derive general formulae for certain integrals of the form 
\[
\int_0^\infty \frac{t\log^kt}{(t^2+at+a^2)(t^2+bt+b^2)}\dd t\quad\text{and}\quad\int_0^\infty\frac{t(t+a)\log^k t}{(t^3-a^3)(t^2+bt+b^2)}\dd t,
\]
for $a,b\in\R$. This is done in Section \ref{sec:generalintegrals}. We also present an alternate method to evaluate these integrals, which results in more explicit expressions rather than recursive ones. This is done in Section \ref{sec:alternat}, and summarized in Proposition \ref{prop:summary}. The multiple integral appearing in the Mahler measure computation can be evaluated using the formulae from Section \ref{sec:generalintegrals} one after the other for each variable.  To systematically capture the details of this process,  we lay out an iterative procedure in Section \ref{sec:compute}. This lets us express the Mahler measure as an integral of a univariate function, and these integrals in turn can be related to values of certain polylogarithms at sixth roots of unity via hyperlogarithms, which we introduce in Section \ref{sec:hyperlogs}. The polylogarithm values can then be expressed in terms of the Riemann zeta function and Dirichlet $L$-functions, giving us the desired result. 

\section*{Acknowledgements}
The author would like to thank his Ph.D. supervisor, Matilde Lalín, for her valuable guidance and support, both throughout the course of this project and beyond. The author also thanks the reviewer for several helpful comments that have greatly improved the readability and quality of the article. The author is also grateful for support from the Institut des sciences mathématiques, the Centre de recherches mathématiques, the Natural Sciences and Engineering Research Council of Canada, and the Fonds de recherche du Québec -- Nature et technologies.

\section{Hyperlogarithms and multiple polylogarithms}\label{sec:hyperlogs}
We first give a brief introduction to hyperlogarithms and multiple polylogarithms and relate them via iterated integrals. One may refer to \cite[Section 12]{Gon95},\cite[Section 2]{Gon-arxiv} and \cite[Section 3]{nvar} for more details.  We follow the same notation to denote an iterated integral
\[
\int_0^{a_{k+1}}\frac{\dd t}{t-a_1}\circ\cdots\circ\frac{\dd t}{t-a_k}:=\underset{P}{\int\cdots\int}\frac{\dd t_1}{t_1-a_1}\cdots\frac{\dd t_k}{t_k-a_k},
\]
where the $a_j$'s are complex numbers and the path of integration $P$ is a path $0\to t_1\to t_2\to\cdots\to t_{k}\to a_{k+1}$ joining 0 and $a_{k+1}$. The value of this integral depends on the homotopy class of $P$ in $\C\setminus\{a_1,a_2,\dots,a_m\}$. With this notation in mind, the \emph{hyperlogarithm} is defined as
\begin{multline*}
\I_{n_1,\dots,n_m}(a_1:a_2:\cdots:a_{m+1}):=\\
\int_0^{a_{m+1}}\underbrace{\frac{\dd t}{t-a_1}\circ\frac{\dd t}{t}\circ\cdots\circ\frac{\dd t}{t}}_{n_1\text{ times}}\circ\underbrace{\frac{\dd t}{t-a_2}\circ\frac{\dd t}{t}\circ\cdots\circ\frac{\dd t}{t}}_{n_2\text{ times}}\circ\cdots\circ\underbrace{\frac{\dd t}{t-a_m}\circ\frac{\dd t}{t}\circ\cdots\circ\frac{\dd t}{t}}_{n_m\text{ times}},
\end{multline*}
where $n_1,\dots,n_m$ are positive integers.  We also define the \emph{multiple polylogarithm} of length $m$ and height $w=n_1+n_2+\cdots+n_m$ as
\[
\Li_{n_1,\dots,n_m}(x_1,\dots,x_m):=\sum_{1\le k_1<k_2<\cdots<k_m}\frac{x_1^{k_1}x_2^{k_2}\cdots x_m^{k_m}}{k_1^{n_1}k_2^{n_2}\cdots k_m^{n_m}}.
\]
 The sum is absolutely convergent for $|x_j|<1$, and also for $|x_j|\le1$ if $n_m\ge2$. One can show (see \cite{Gon-arxiv}) the following identities
\begin{align}
\Li_{n_1,\dots,n_m}(x_1,\dots,x_m)&=(-1)^m\, \I_{n_1,\dots,n_m}\left(\frac{1}{x_1\cdots x_m}:\frac1{x_2\cdots x_m}:\dots:\frac1{x_m}:1\right)\label{eq:polylogint},\\
\I_{n_1,\dots,n_m}(a_1:\cdots:a_{m+1})&=(-1)^m\,\Li_{n_1,\dots,n_m}\left(\frac{a_2}{a_1},\frac{a_3}{a_2},\dots,\frac{a_{m+1}}{a_m}\right).\label{eq:hypertopoly}
\end{align}
Note that identity \eqref{eq:polylogint} enables an analytic continuation of the multiple polylogarithm using the definition of the hyperlogarithm. The value of the hyperlogarithm function depends on the homotopy class of the path joining 0 and $a_{m+1}$ in $\C\setminus\{a_1,\dots,a_m\}$. This means that these multivalued functions may not be defined uniquely. However, we will always encounter linear combinations of these functions that yield single-valued functions. For example, in Section \ref{sec:polylogsaslvals}, we will see that the Mahler measure of $Q_n$ can be expressed as sums of integrals such as 
\[
\int_0^1\frac{\log^{2h+1}t}{t^2+t+1}\dd t,
\]
which are well-defined. Note that we have
\[\frac{1}{t^2+t+1}=\frac{1}{i\sqrt3}\left(\frac1{t-\omega}-\frac1{t-\omega^2}\right),\]
where $\omega=e^{2\pi i/3}$ is a third root of unity. For $a\in\C^*$, we have 
\begin{align*}
\int_0^1 \log^{2h}t\,\frac{1}{t-a}\dd t&= (-1)^{2h}(2h)!\int_0^1 \frac{\dd t}{t-a}\circ \overbrace{\frac{\dd t}{t}\circ\cdots\circ \frac{\dd t}{t}}^\text{$2h$ times}\\
&=(2h)!\cdot\mathrm{I}_{2h+1}(a,1)=-(2h)!\,\mathrm{Li}_{2h+1}(1/a).
\end{align*}
Using these, we may write
\[
\int_0^1\frac{\log^{2h+1}t}{t^2+t+1}\dd t=(2h)!\left(\mathrm{I}_{2h+1}(\omega,1)-\mathrm{I}_{2h+1}(\omega^2,1)\right),
\]
which shows that while the individual hyperlogarithms $\mathrm{I}_{2h+1}(\omega,1)$ and $\mathrm{I}_{2h+1}(\omega^2,1)$ are undetermined, due to the multi-valued nature of the hyperlogarithm, their difference is well-defined.\\

In Section \ref{sec:polylogsaslvals}, we will use the hyperlogarithm identities to write the integral defining the Mahler measure in terms of polylogarithms, which would then give us certain $L$-values.

\section{Some general integrals}\label{sec:generalintegrals}
In this section, we will evaluate certain single-variable integrals involving arbitrary powers of the logarithm. At the end of this section, we will also discuss an alternate method of computing these integrals. In Section \ref{sec:compute} we will use these formulae to explicitly compute the Mahler measure of a certain family of polynomials. 

\subsection{The integral $f_1(k)$}\label{sec:f1}
First, we wish to evaluate the following integral which we denote by $f_1(k)$:
\[
f_1(k)=\int_0^\infty\frac{t\log^kt}{(t^2+at+a^2)(t^2+bt+b^2)}\dd t,
\]
for any integer $k\ge0$ and $a,b\in\R^+$. Here, and in the entire article, the symbol $\infty$ will always mean $+\infty$. To compute the integral $f_1(k)$, we evaluate a contour integral 
\begin{equation}\label{cont:1}
\oint_C\frac{z\log^{k+1}z}{(z^2+az+a^2)(z^2+bz+b^2)}\dd z, 
\end{equation}
with the keyhole contour $C$ as given by Figure \ref{fig:contour2}, comprising of the paths labelled  $C_\varepsilon, \gamma_1,C_R$ and $\gamma_2$. Note that this keyhole contour is chosen in order to skip the point $z=0$. We choose the branch cut of the logarithm to be the positive real axis so that the imaginary part of the logarithm function has the range $(0,2\pi)$. The path $C_\varepsilon$ is a circular path centred at $z=0$ with small radius $\varepsilon>0$, drawn such that it doesn't cross the positive real axis. Similarly, $C_R$ is another circular path centred at the origin, with large radius $R$, and drawn such that it doesn't cross the positive real axis. The integrals corresponding to paths $C_\varepsilon$ and $C_R$ will vanish as $\varepsilon\to0$ and $R\to\infty$. Indeed, we can parametrize the integration path along $C_\varepsilon$ by $z=\varepsilon e^{i\theta}$, where $\theta$ varies from $2\pi^-$ to $0^+$. Thus, taking the limit $\varepsilon\to0$, the integral in \eqref{cont:1} along $C_\varepsilon$ is given by
\begin{align}\label{cont:limit1}
\lim_{\varepsilon\to0}\int_{2\pi^-}^{0^+}\frac{\varepsilon e^{i\theta}(\log \varepsilon+i\theta)^{k+1}}{(\varepsilon^2 e^{2i\theta}+a\varepsilon e^{i\theta}+a^2)(\varepsilon^2 e^{2i\theta}+b\varepsilon e^{i\theta}+b^2)}\cdot i\varepsilon e^{i\theta}\dd\theta.
\end{align}
Let the integrand in \eqref{cont:limit1} be given by the function 
\[
h(\varepsilon,\theta)=i\varepsilon e^{i\theta}\cdot\frac{\varepsilon e^{i\theta}(\log \varepsilon+i\theta)^{k+1}}{(\varepsilon^2 e^{2i\theta}+a\varepsilon e^{i\theta}+a^2)(\varepsilon^2 e^{2i\theta}+b\varepsilon e^{i\theta}+b^2)}.
\]
Choose a positive real number $\varepsilon_0>0$ such that $\varepsilon_0<\min\{a,b\}$. This means that the roots of the equation $z^2+az+a^2=0$ and $z^2+bz+b^2=0$ lie outside $C_{\varepsilon_0}$, and thus, if $\varepsilon<\varepsilon_0$, then
\[
|\varepsilon^2 e^{2i\theta}+a\varepsilon e^{i\theta}+a^2|\ge |a-\varepsilon_0|^2>0,
\]
and 
\[
|\varepsilon^2 e^{2i\theta}+b\varepsilon e^{i\theta}+b^2|\ge |b-\varepsilon_0|^2>0.
\]
Therefore, we have
\begin{align*}
|h(\varepsilon,\theta)|&\le\frac{\varepsilon^2|\log\varepsilon+2\pi i|}{|\varepsilon^2 e^{2i\theta}+a\varepsilon e^{i\theta}+a^2|\cdot|\varepsilon^2 e^{2i\theta}+b\varepsilon e^{i\theta}+b^2|}\\
&\le\frac{\varepsilon^2\sqrt{\log^2\varepsilon+4\pi^2}}{(a-\varepsilon_0)^2(b-\varepsilon_0)^2},
\end{align*}
for $0<\varepsilon<\varepsilon_0$. %Moreover, we have 
%\[
%\lim_{\varepsilon\to0}\frac{\varepsilon^2\sqrt{\log^2\varepsilon+4\pi^2}}{(a-\varepsilon_0)^2(b-\varepsilon_0)^2}=0,
%\]
%which means that $|h(\varepsilon,\theta)|$ is bounded for small enough $\varepsilon$. Thus, using the dominated convergence theorem, we may bring the limit inside the integral in \eqref{cont:limit1}, and
Using this upper bound in \eqref{cont:limit1}, we get 
\begin{align*}
&\left|\lim_{\varepsilon\to0}\int_{2\pi^-}^{0^+}\frac{\varepsilon e^{i\theta}(\log \varepsilon+i\theta)^{k+1}}{(\varepsilon^2 e^{2i\theta}+a\varepsilon e^{i\theta}+a^2)(\varepsilon^2 e^{2i\theta}+b\varepsilon e^{i\theta}+b^2)}\cdot i\varepsilon e^{i\theta}\dd\theta\right|\\
&\qquad\qquad=\lim_{\varepsilon\to0}\left|\int_{2\pi^-}^{0^+}\frac{\varepsilon e^{i\theta}(\log \varepsilon+i\theta)^{k+1}}{(\varepsilon^2 e^{2i\theta}+a\varepsilon e^{i\theta}+a^2)(\varepsilon^2 e^{2i\theta}+b\varepsilon e^{i\theta}+b^2)}\cdot i\varepsilon e^{i\theta}\dd\theta\right|\\
&\qquad\qquad\le\lim_{\varepsilon\to0}\int_{0^+}^{2\pi^-}\left|\frac{\varepsilon e^{i\theta}(\log \varepsilon+i\theta)^{k+1}}{(\varepsilon^2 e^{2i\theta}+a\varepsilon e^{i\theta}+a^2)(\varepsilon^2 e^{2i\theta}+b\varepsilon e^{i\theta}+b^2)}\cdot i\varepsilon e^{i\theta}\right|\dd\theta\\
&\qquad\qquad\le\lim_{\varepsilon\to0}\int_{0^+}^{2\pi^-}\frac{\varepsilon^2\sqrt{\log^2\varepsilon+4\pi^2}}{(a-\varepsilon_0)^2(b-\varepsilon_0)^2}\dd\theta=\lim_{\varepsilon\to0}\frac{\varepsilon^2\sqrt{\log^2\varepsilon+4\pi^2}}{(a-\varepsilon_0)^2(b-\varepsilon_0)^2}\cdot2\pi=0.
\end{align*}

Thus, the contribution to the contour integral \eqref{cont:1} along $C_\epsilon$ is zero.\\

Similarly, the contribution along $C_R$ as $R\to\infty$ is given by the integral
\begin{equation*}\lim_{R\to\infty}\int_{0^+}^{2\pi^-}h(R,\theta)\dd\theta.
\end{equation*}
For large enough $R$, we have that $|R^2 e^{2i\theta}+aR e^{i\theta}+a^2|>(R-a)^2$ and $|R^2 e^{2i\theta}+bR e^{i\theta}+b^2|>(R-b)^2$, which means that
\begin{align*}
|h(R,\theta)|&\le\frac{R^2|\log R+2\pi i|}{|R^2 e^{2i\theta}+aR e^{i\theta}+a^2|\cdot|R^2 e^{2i\theta}+bR e^{i\theta}+b^2|}\\
&\le\frac{R^2\sqrt{\log^2R+4\pi^2}}{(R-a)^2(R-b)^2}.
\end{align*}
As $R\to\infty$, we have
\[
\frac{R^2\sqrt{\log^2R+4\pi^2}}{(R-a)^2(R-b)^2}\to0,
\]and a similar argument as before tells us that the contribution corresponding to $C_R$ is also zero.\\

What remains is the combination of $\gamma_1$ and $\gamma_2$. The path $\gamma_1$ is a straight line just above the positive real axis that connects the end of the path $C_\varepsilon$ and the start of $C_R$. The argument of the complex logarithm on this path is a small positive angle which goes to zero as $\varepsilon\to0$. Similarly, the path $\gamma_2$  is a straight line just below the real axis, joining the end of $C_R$ and the start of $C_\varepsilon$, where the argument of the complex logarithm is taken to be $2\pi $ in the limit. Thus, as $\varepsilon\to0$, the respective contributions of $\gamma_1$ and $\gamma_2$ give the sum
\begin{align}\label{int:contour}
&\int_0^\infty\frac{t\log^{k+1} t}{(t^2+at+a^2)(t^2+bt+b^2)}\dd t+\int_\infty^0\frac{t(\log t+2\pi i)^{k+1}}{(t^2+at+a^2)(t^2+bt+b^2)}\dd t.
\end{align}
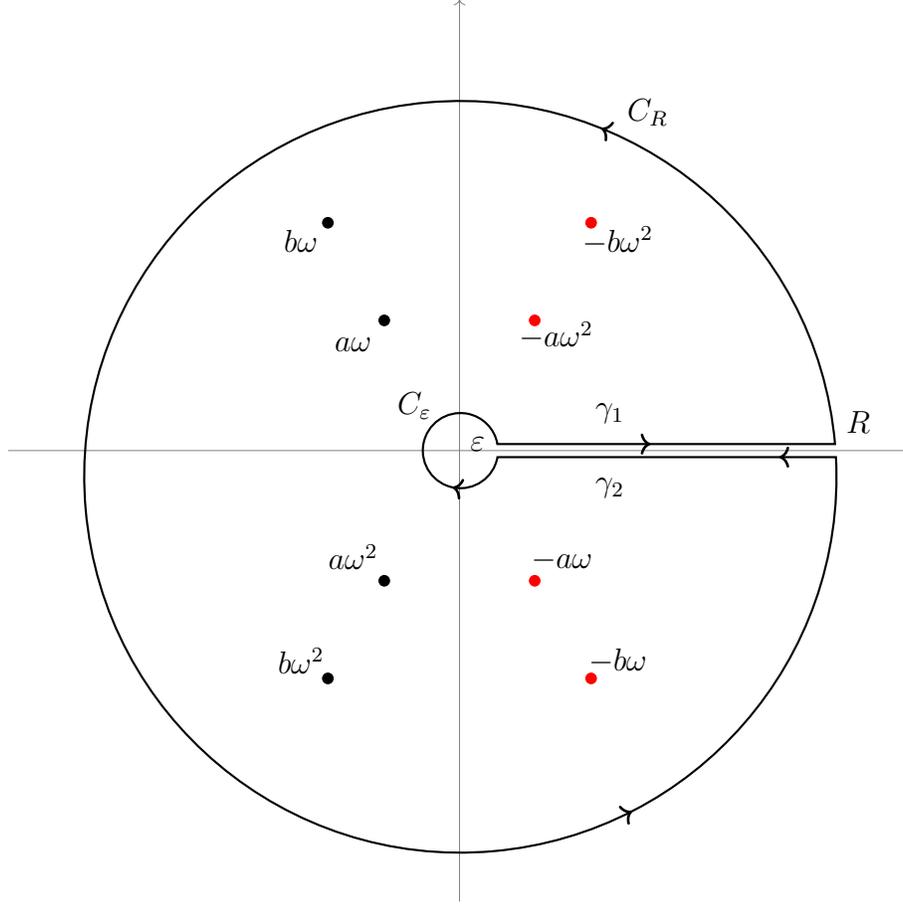
\begin{figure}[htb]
\caption[The contour for general $k$]
{The keyhole contour $C$ for $f_j(k)$ - the black dots denote the poles for $f_1(k)$, while the red dots denote the poles for $f_2(k)$.}
\centering
\begin{tikzpicture}
[decoration={markings,
mark=at position 2.05cm with {\arrow[line width=1pt]{>}},
%mark=at position 4.5cm with {\arrow[line width=1pt]{>}},
mark=at position 10cm with {\arrow[line width=1pt]{>}},
mark=at position 30cm with {\arrow[line width=1pt]{>}},
mark=at position 36.5cm with {\arrow[line width=1pt]{>}},
%mark=at position 39.5cm with {\arrow[line width=1pt]{>}},
mark=at position 41.05cm with {\arrow[line width=1pt]{>}},
mark=at position 43.5cm with {\arrow[line width=1pt]{>}}
}
]
\label{fig:contour2}
% The axes
\draw[help lines,->] (-6,0) -- (6,0) coordinate (xaxis);
\draw[help lines,->] (0,-6) -- (0,6) coordinate (yaxis);

% The path
\path[draw,line width=0.8pt,postaction=decorate] (10:0.5) node[left] {$\varepsilon$} -- ++(4.5,0) node[above right] {$R$} arc(5:363:5)-- ++(-4.5,0)  arc (-10:-350:0.5);

\filldraw[black] (120:3.5) circle (2pt);
\node at (127:3.5) {$b\omega$};
\filldraw[black] (240:3.5) circle (2pt);
\node at (233:3.5) {$b\omega^2$};
\filldraw[black] (120:2) circle (2pt);
\node at (135:2) {$a\omega$};
\filldraw[black] (240:2) circle (2pt);
\node at (225:2) {$a\omega^2$};
%\filldraw[black] (2.5,0) circle (2pt);
\node at (53:3.5) {$-b\omega^2$};
\filldraw[red] (300:3.5) circle (2pt);
\node at (307:3.5) {$-b\omega$};
\filldraw[red] (60:3.5) circle (2pt);
\node at (50:2) {$-a\omega^2$};
\filldraw[red] (300:2) circle (2pt);
\node at (313:2) {$-a\omega$};
\filldraw[red] (60:2) circle (2pt);
%\node[below right] at (2.5,0.05) {\scriptsize $1$};
\node at (-0.6,0.6) {$C_{\varepsilon}$};
\node at (2,0.5) {$\gamma_1$};
%\node at (2.5,0.9) {$C_{1}$};
%\node at (4.0,0.5) {$\gamma_2$};
\node at (2.5,4.5) {$C_{R}$};
%\node at (4.0,-0.5) {$\gamma_3$};
%\node at (2.5,-0.9) {$C_{2}$};
\node at (2,-0.5) {$\gamma_2$};
%\node at (1.4,1.9) {$\gamma_{r}$};
\end{tikzpicture}
\end{figure}
Writing $\theta=2\pi/3$ so that $2\pi=3\theta$,  and using the binomial expansion, we have the sum of the integrals in \eqref{int:contour} to be equal to
\begin{equation}\label{eq:contint}
-\int_0^\infty\frac{t\sum_{j=1}^{k+1}{{k+1}\choose j}\log^{k+1-j}t \;(3\theta)^ji^j}{(t^2+at+a^2)(t^2+bt+b^2)}\dd t.
\end{equation}
By the residue theorem, this expression will be equal to the sum of residues inside the contour of integration. To calculate the value of $f_1(k)$, we will compare the imaginary part of equation \eqref{eq:contint} with the imaginary part of the sum of residues.  Collecting the purely imaginary terms in the contour integration \eqref{eq:contint} gives
\begin{align}\label{eq:intk}
&-\sum_{j=1,\text{ odd}}^{k+1}{{k+1}\choose j}(3\theta)^j\,i^j\,\int_0^\infty\frac{t\log^{k+1-j}t}{(t^2+at+a^2)(t^2+bt+b^2)}\dd t\nonumber\\
=&-\sum_{j=1,\text{ odd}}^{k+1}{{k+1}\choose j}(3\theta)^j\,i^j\,f_1(k+1-j)\nonumber\\
=&-3i\theta(k+1)f_1(k)-\sum_{j>1,\text{ odd}}^{k+1}{{k+1}\choose j}(3\theta)^j\,i^j\,f_1(k+1-j).
\end{align}
Next, we compute the sum of residues at the poles of the integrand in $f_1(k)$ lying inside the contour.  These poles are denoted by black dots at $z=a\omega,a\omega^2,b\omega$ and $b\omega^2$ in Figure \ref{fig:contour2}. The residue calculation is given by
\begin{align}\label{eq:rescalc}
&2\pi i\Biggl[\frac{a\omega(\log a+i\theta)^{k+1}}{(a\omega-a\omega^2)(a\omega-b\omega)(a\omega-b\omega^2)}+\frac{a\omega^2(\log a+2i\theta)^{k+1}}{(a\omega^2-a\omega)(a\omega^2-b\omega)(a\omega^2-b\omega^2)}\nonumber\\
&\qquad\qquad\qquad\qquad+\frac{b\omega(\log b+i\theta)^{k+1}}{(b\omega-a\omega)(b\omega-a\omega^2)(b\omega-b\omega^2)}+\frac{b\omega^2(\log b+2i\theta)^{k+1}}{(b\omega^2-a\omega)(b\omega^2-a\omega^2)(b\omega^2-b\omega)}\Biggr]\nonumber\\
=&2\pi i\left[\frac{(a\omega^2-b\omega)\Big((\log a+i\theta)^{k+1}-(\log b+2i\theta)^{k+1}\Big)+(a\omega-b\omega^2)\Big((\log b+i\theta)^{k+1}-(\log a+2i\theta)^{k+1}\Big)}{(\omega-\omega^2)(a^3-b^3)}\right]\nonumber\\
=&2\pi \left[\frac{(a\omega^2-b\omega)\Big((\log a+i\theta)^{k+1}-(\log b+2i\theta)^{k+1}\Big)+(a\omega-b\omega^2)\Big((\log b+i\theta)^{k+1}-(\log a+2i\theta)^{k+1}\Big)}{\sqrt3(a^3-b^3)}\right].
\end{align}
We wish to isolate the purely imaginary terms in \eqref{eq:rescalc}. We have
\[
a\omega^2-b\omega=-\frac{(a-b)}2-\frac{\sqrt3(a+b)}2\;i,
\]
whose conjugate gives $a\omega-b\omega^2$. We can write
\begin{equation*}
(\log a+i\theta)^{k+1}=\sum_{j=0,\text{ even}}^{k+1}{{k+1}\choose{j}}i^{j}\theta^j\log^{k+1-j}a+\sum_{j=1,\text{ odd}}^{k+1}{{k+1}\choose{j}}i^{j}\theta^j\log^{k+1-j}a, 
\end{equation*}
where the first summation corresponds to the real part and the second summation corresponds to the imaginary part. We may write the remaining $(k+1)^{\text{th}}$ powers appearing in \eqref{eq:rescalc} in a similar manner as a sum of real and imaginary parts. Carrying out the appropriate products in \eqref{eq:rescalc}, and collecting only the purely imaginary terms together, we obtain
\begin{multline}\label{eq:resk}
\frac{-\pi}{\sqrt3(a^3-b^3)}\Biggl[\sqrt3(a+b)\sum_{j=0,\text{ even}}^{k+1}{{k+1}\choose{j}}(2^j+1)i^{j+1}\theta^j(\log^{k+1-j}a-\log^{k+1-j}b)\\
+ (a-b)\sum_{j=1,\text{ odd}}^{k+1}{{k+1}\choose{j}}(2^j-1)i^{j}\theta^j(\log^{k+1-j}a+\log^{k+1-j}b) \Biggr].
\end{multline}
Expression \eqref{eq:resk} forms the imaginary term in the sum of residues, which we equate to expression \eqref{eq:intk} to get
\begin{multline*}
-3i\theta(k+1)f_1(k)-\sum_{j>1,\text{ odd}}^{k+1}{{k+1}\choose j}(3\theta)^j\,i^j\,f_1(k+1-j)\\
=\frac{-\pi}{\sqrt3(a^3-b^3)}\Biggl[\sqrt3(a+b)\sum_{j=0,\text{ even}}^{k+1}{{k+1}\choose{j}}(2^j+1)i^{j+1}\theta^j(\log^{k+1-j}a-\log^{k+1-j}b)\\
+ (a-b)\sum_{j=1,\text{ odd}}^{k+1}{{k+1}\choose{j}}(2^j-1)i^{j}\theta^j(\log^{k+1-j}a+\log^{k+1-j}b) \Biggr].
\end{multline*}
Recall again that $3\theta=2\pi$. Dividing throughout by $3\theta i(k+1)$ and isolating for $f_1(k)$, we obtain
\begin{multline}\label{eq:f1sum}
f_1(k)=\frac1{k+1}\sum_{j>1,\text{ odd}}^{k+1}(-1)^{\frac{j+1}2}{{k+1}\choose j}(3\theta)^{j-1}f_1(k+1-j)\\
+\frac{(a+b)}{2(k+1)(a^3-b^3)}\sum_{j=0,\text{ even}}^{k+1}(-1)^{\frac j2}\theta^j{{k+1}\choose{j}}(2^j+1)(\log^{k+1-j}a-\log^{k+1-j}b)\\
+\frac{1}{2\sqrt3(k+1)(a^2+ab+b^2)}\sum_{j=1,\text{ odd}}^{k+1}(-1)^{\frac{j+1}2}\theta^j{{k+1}\choose{j}}(2^j-1)(\log^{k+1-j}a+\log^{k+1-j}b).
\end{multline}
%\subsubsection{The polynomials $R_k(x)$ and $S_k(x)$ }
To write the above expression in a more systematic way,  we will define two polynomials in a recursive manner. Observe that in the expression for $f_1(k)$ in \eqref{eq:f1sum}, there are three parts as $j$ varies from $1$ to $k+1$ -- one involving lower terms $f_1(k+1-j)$:
\[
\frac1{k+1}\sum_{j>1,\text{ odd}}^{k+1}(-1)^{\frac{j+1}2}{{k+1}\choose j}(3\theta)^{j-1}f_1(k+1-j),
\]
where $k+1-j$ is strictly smaller than $k$; the second part which corresponds to the even values of $j$:
\[
\frac{(a+b)}{2(k+1)(a^3-b^3)}\sum_{j=0,\text{ even}}^{k+1}(-1)^{\frac j2}\theta^j{{k+1}\choose{j}}(2^j+1)(\log^{k+1-j}a-\log^{k+1-j}b);
\]
and finally the third part that corresponds to the odd values of $j$:
\[
\frac{1}{2\sqrt3(k+1)(a^2+ab+b^2)}\sum_{j=1,\text{ odd}}^{k+1}(-1)^{\frac{j+1}2}\theta^j{{k+1}\choose{j}}(2^j-1)(\log^{k+1-j}a+\log^{k+1-j}b).
\]
We will define two polynomials with rational coefficients $R_k(x)$ and $S_k(x)$, one for the even values of $j$ (the second part), and one for the odd values (the third part), respectively. The first part will define the recursive property of both of these polynomials.  The idea is to write $\frac{\log a}{\theta}$ as $x$, so that we may replace
\[
\theta^j \cdot(\log^{k+1-j} a) \quad\text{by}\quad\theta^{k+1} x^{k+1-j},
\]
and similarly for $\frac{\log b}{\theta}$.  Consequently, we define
\begin{multline}\label{eq:rk}
R_k(x)=\frac1{k+1}\sum_{j>1,\text{ odd}}^{k+1}(-1)^{\frac{j+1}2}{{k+1}\choose j}3^{j-1}R_{k+1-j}(x)\\
+\frac{1}{2(k+1)}\sum_{j=0,\text{ even}}^{k+1}(-1)^{\frac j2}{{k+1}\choose{j}}(2^j+1)x^{k+1-j},
\end{multline}
and
\begin{multline}\label{eq:sk}
S_k(x)=\frac1{k+1}\sum_{j>1,\text{ odd}}^{k+1}(-1)^{\frac{j+1}2}{{k+1}\choose j}3^{j-1}S_{k+1-j}(x)\\
+\frac{1}{2(k+1)}\sum_{j=1,\text{ odd}}^{k+1}(-1)^{\frac{j+1}2}{{k+1}\choose{j}}(2^j-1)x^{k+1-j}.
\end{multline}
Usually, we need to specify initial values to completely determine a family of recursive polynomials. However, in our case we do not need to do so since there is also a non-recursive part in the definition of the polynomials which serves to define the initial values (at $k=0$). For instance, we have 
\[
R_0(x)=\frac{1}{2}(2)x=x,
\]
and
\[
S_0(x)=\frac{1}{2}(-1)(1)=-\frac12.
\]
We also have the following examples:
\renewcommand{\arraystretch}{2.5}
\renewcommand{\tabcolsep}{0.5cm}
\[\begin{array}{|c|c|}
\hline 
R_k(x) & S_k(x)\\
\hline
\displaystyle R_1(x)=\frac{x^2}2-\frac54 &\displaystyle S_1(x)=-\frac{x}2\\
\hline
\displaystyle R_2(x)=\frac{x^3}3+\frac{x}2 &\displaystyle S_2(x)=-\frac{x^2}2-\frac13\\
\hline
\displaystyle R_3(x)=\frac{x^4}4+\frac{3x^2}4-\frac{73}8 &\displaystyle S_3(x)=-\frac{x^3}2-x.\\
\hline
\end{array}\]
One may also note that the degree of $R_k(x)$ is $k+1$ while that of $S_k(x)$ is $k$.
Lastly, we write the final expression for $f_1(k)$ from \eqref{eq:f1sum},
\begin{multline}\label{eq:fk}
f_1(k)=\frac{\theta^{k+1}(a+b)}{a^3-b^3}\Bigg[R_k\left(\frac{\log a}\theta\right)-R_k\left(\frac{\log b}\theta\right)\Bigg]\\+\frac{\theta^{k+1}}{\sqrt3(a^2+ab+b^2)}\Bigg[S_k\left(\frac{\log a}\theta\right)+S_k\left(\frac{\log b}\theta\right)\Bigg].
\end{multline}

\subsection{The integral $f_2(k)$}\label{sec:f2}
Next, we evaluate an integral similar to $f_1(k)$ from Section \ref{sec:f1} but with a minor modification. We take $a,b\in\R^+$ and consider the integral
\[
f_2(k)=\int_0^\infty\frac{t\log^kt}{(t^2-at+a^2)(t^2-bt+b^2)}\dd t.
\]
Note the changes in sign in the expression in the denominator as compared to $f_1(k)$. The contour integral we evaluate in this case is
\[
\oint_C\frac{z\log^{k+1}z}{(z^2-az+a^2)(z^2-bz+b^2)}\dd z, 
\]
with contour as given by Figure \ref{fig:contour2}.  As before, the contribution along the paths $C_\varepsilon$ and $C_R$ will be zero. Combining the contributions along $\gamma_1$ and $\gamma_2$ and collecting the imaginary terms using the binomial expansion we get a term similar to \eqref{eq:intk}:
\begin{align}\label{eq:recf2}
-3i\theta(k+1)f_2(k)-\sum_{j>1,\text{ odd}}^{k+1}{{k+1}\choose j}(3\theta)^j\,i^j\,f_2(k+1-j).
\end{align}
Now we compute the residues which are obtained at $z=-a\omega,-a\omega^2,-b\omega$ and $-b\omega^2$. The only difference here is that the arguments will be $\delta=\pi/3$ and $5\delta=5\pi/3$ this time. The residue calculation is given by
\begin{align*}
&2\pi i\Biggl[\frac{-a\omega(\log a+5i\delta)^{k+1}}{(-a\omega+a\omega^2)(-a\omega+b\omega)(-a\omega+b\omega^2)}+\frac{-a\omega^2(\log a+i\delta)^{k+1}}{(-a\omega^2+a\omega)(-a\omega^2+b\omega)(-a\omega^2+b\omega^2)}\\
&\qquad\qquad\qquad\qquad+\frac{-b\omega(\log b+5i\delta)^{k+1}}{(-b\omega+a\omega)(-b\omega+a\omega^2)(-b\omega+b\omega^2)}+\frac{-b\omega^2(\log b+i\delta)^{k+1}}{(-b\omega^2+a\omega)(-b\omega^2+a\omega^2)(-b\omega^2+b\omega)}\Biggr]\\
=&2\pi i\left[\frac{(a\omega^2-b\omega)\Big((\log a+5i\delta)^{k+1}-(\log b+i\delta)^{k+1}\Big)+(a\omega-b\omega^2)\Big((\log b+5i\delta)^{k+1}-(\log a+i\delta)^{k+1}\Big)}{(\omega-\omega^2)(a^3-b^3)}\right].
\end{align*}
Collecting imaginary terms together gives
\begin{multline}\label{eq:resk2}
\frac{-\pi}{\sqrt3(a^3-b^3)}\Biggl[\sqrt3(a+b)\sum_{j=0,\text{ even}}^{k+1}{{k+1}\choose{j}}(5^j+1)i^{j+1}\delta^j(\log^{k+1-j}a-\log^{k+1-j}b)\\
+ (a-b)\sum_{j=1,\text{ odd}}^{k+1}{{k+1}\choose{j}}(5^j-1)i^{j}\delta^j(\log^{k+1-j}a+\log^{k+1-j}b) \Biggr].
\end{multline}
Equating \eqref{eq:recf2} and \eqref{eq:resk2} and rearranging, we obtain
\begin{multline}\label{eq:f2sum}
f_2(k)=\frac1{k+1}\sum_{j>1,\text{ odd}}^{k+1}(-1)^{\frac{j+1}2}{{k+1}\choose j}(3\theta)^{j-1}f_2(k+1-j)\\
+\frac{(a+b)}{2(k+1)(a^3-b^3)}\sum_{j=0,\text{ even}}^{k+1}(-1)^{\frac j2}\delta^j{{k+1}\choose{j}}(5^j+1)(\log^{k+1-j}a-\log^{k+1-j}b)\\
+\frac{1}{2\sqrt3(k+1)(a^2+ab+b^2)}\sum_{j=1,\text{ odd}}^{k+1}(-1)^{\frac{j-1}2}\delta^j{{k+1}\choose{j}}(5^j-1)(\log^{k+1-j}a+\log^{k+1-j}b).
\end{multline}

%\subsubsection{The polynomials $P_k(x)$ and $Q_k(x)$}
Once again, we use two polynomials -- $P_k(x)$ corresponding to the even indices and $Q_k(x)$ to the odd indices in \eqref{eq:f2sum}, both having a recursive part corresponding to the recursive part in \eqref{eq:f2sum}.  Note that $3\theta=6\delta=2\pi$. Once again, we would like to replace $\frac{\log a}{\delta}$ by $x$ so that
\[
\delta^j \cdot(\log^{k+1-j} a)=\delta^{k+1} x^{k+1-j},
\]
and similarly for $\frac{\log b}{\delta}$.
We define 
\begin{multline}\label{eq:pk}
P_k(x)=\frac1{k+1}\sum_{j>1,\text{ odd}}^{k+1}(-1)^{\frac{j+1}2}{{k+1}\choose j}6^{j-1}P_{k+1-j}(x)\\
+\frac{1}{2(k+1)}\sum_{j=0,\text{ even}}^{k+1}(-1)^{\frac j2}{{k+1}\choose{j}}(5^j+1)x^{k+1-j},
\end{multline}
and
\begin{multline}\label{eq:qk}
Q_k(x)=\frac1{k+1}\sum_{j>1,\text{ odd}}^{k+1}(-1)^{\frac{j+1}2}{{k+1}\choose j}6^{j-1}Q_{k+1-j}(x)\\
+\frac{1}{2(k+1)}\sum_{j=1,\text{ odd}}^{k+1}(-1)^{\frac{j-1}2}{{k+1}\choose{j}}(5^j-1)x^{k+1-j}.
\end{multline}
We list the first few examples:
\renewcommand{\arraystretch}{2.5}
\renewcommand{\tabcolsep}{0.5cm}
\[\begin{array}{|c|c|}
\hline 
P_k(x) & Q_k(x)\\
\hline
\displaystyle P_0(x)=x &\displaystyle Q_0(x)=2\\
\hline
\displaystyle P_1(x)=\frac{x^2}2-\frac{13}2 &\displaystyle Q_1(x)=2x\\
\hline
\displaystyle P_2(x)=\frac{x^3}3-x &\displaystyle Q_2(x)=2x^2+\frac{10}3\\
\hline
\displaystyle P_3(x)=\frac{x^4}4-\frac{3x^2}2-\frac{623}4 &\displaystyle Q_3(x)=2x^3+10x.\\
\hline
\end{array}\]

Thus, using the above definitions in \eqref{eq:f2sum}, we obtain
\begin{multline}\label{eq:f2k}
f_2(k)=\frac{\delta^{k+1}(a+b)}{a^3-b^3}\Bigg[P_k\left(\frac{\log a}\delta\right)-P_k\left(\frac{\log b}\delta\right)\Bigg]\\+\frac{\delta^{k+1}}{\sqrt3(a^2+ab+b^2)}\Bigg[Q_k\left(\frac{\log a}\delta\right)+Q_k\left(\frac{\log b}\delta\right)\Bigg].
\end{multline}

\subsection{The integral $g_1(k)$}
\label{sec:g1}
Now we look at evaluating 
\[
g_1(k) =\int_0^\infty \frac{t(t+a)\log^k t}{(t^3-a^3)(t^2+bt+b^2)}\dd t,
\]
with $a,b\in\R^+$.  As before, we consider the integral
\begin{equation}\label{cont:3}
\oint_{C'} \frac{z(z+a)\log^{k+1} z}{(z^3-a^3)(z^2+bz+b^2)}\dd z,
\end{equation}
with a modified keyhole contour $C'$ as shown in Figure ~\ref{fig:contour3}. Note that in this case, the keyhole contour skips the points $z=0$ and $z=a$, and consists of the paths labelled $C_\varepsilon,\gamma_1,C_1,\gamma_2,C_R,\gamma_3,C_2$ and $\gamma_4$.\\
\begin{figure}[htb]
\caption{The contour $C'$ for $g_j(k)$ - the black dots denote the poles for $g_1(k)$, while the red dots denote the poles for $g_2(k)$.  Note that $z=a$ is outside the contour.}
\centering
\begin{tikzpicture}
[decoration={markings,
mark=at position 1.05cm with {\arrow[line width=1pt]{>}},
mark=at position 4.5cm with {\arrow[line width=1pt]{>}},
mark=at position 10cm with {\arrow[line width=1pt]{>}},
mark=at position 30cm with {\arrow[line width=1pt]{>}},
mark=at position 37.75cm with {\arrow[line width=1pt]{>}},
mark=at position 39.5cm with {\arrow[line width=1pt]{>}},
mark=at position 41.05cm with {\arrow[line width=1pt]{>}},
mark=at position 43.5cm with {\arrow[line width=1pt]{>}}
}
]
\label{fig:contour3}
% The axes
\draw[help lines,->] (-6,0) -- (6,0) coordinate (xaxis);
\draw[help lines,->] (0,-6) -- (0,6) coordinate (yaxis);

% The path
\path[draw,line width=0.8pt,postaction=decorate] (10:0.5) node[left] {$\varepsilon$} -- ++(1.5,0) node[right] {$r$} arc (180:0:0.5)--+(2,0) node[above right] {$R$} arc(5:363:5)-- ++(-2,0) arc(0:-180:0.5) --+(-1.5,0) arc (-10:-350:0.5);
\filldraw[red] (180:2.5) circle (2pt);

\filldraw[black] (120:3.5) circle (2pt);
\node[below right] at (-2.5,-0.05) { $-a$};

\filldraw[black] (120:3.5) circle (2pt);
\node at (127:3.5) {$b\omega$};
\filldraw[black] (240:3.5) circle (2pt);
\node at (233:3.5) {$b\omega^2$};
\filldraw[black] (120:2) circle (2pt);
\node at (135:2) {$a\omega$};
\filldraw[black] (240:2) circle (2pt);
\node at (225:2) {$a\omega^2$};
\filldraw[green] (2.5,0) circle (2pt);
\node[below right] at (2.5,0.05) { $a$};
\node at (53:3.5) {$-b\omega^2$};
\filldraw[red] (300:3.5) circle (2pt);
\node at (307:3.5) {$-b\omega$};
\filldraw[red] (60:3.5) circle (2pt);
\node at (50:2) {$-a\omega^2$};
\filldraw[red] (300:2) circle (2pt);
\node at (313:2) {$-a\omega$};
\filldraw[red] (60:2) circle (2pt);

\node at (-0.6,0.6) {$C_{\varepsilon}$};
\node at (1.2,0.5) {$\gamma_1$};
\node at (2.5,0.9) {$C_{1}$};
\node at (4.0,0.5) {$\gamma_2$};
\node at (2.5,4.5) {$C_{R}$};
\node at (4.0,-0.5) {$\gamma_3$};
\node at (2.5,-0.9) {$C_{2}$};
\node at (1.2,-0.5) {$\gamma_4$};
\end{tikzpicture}
\end{figure}
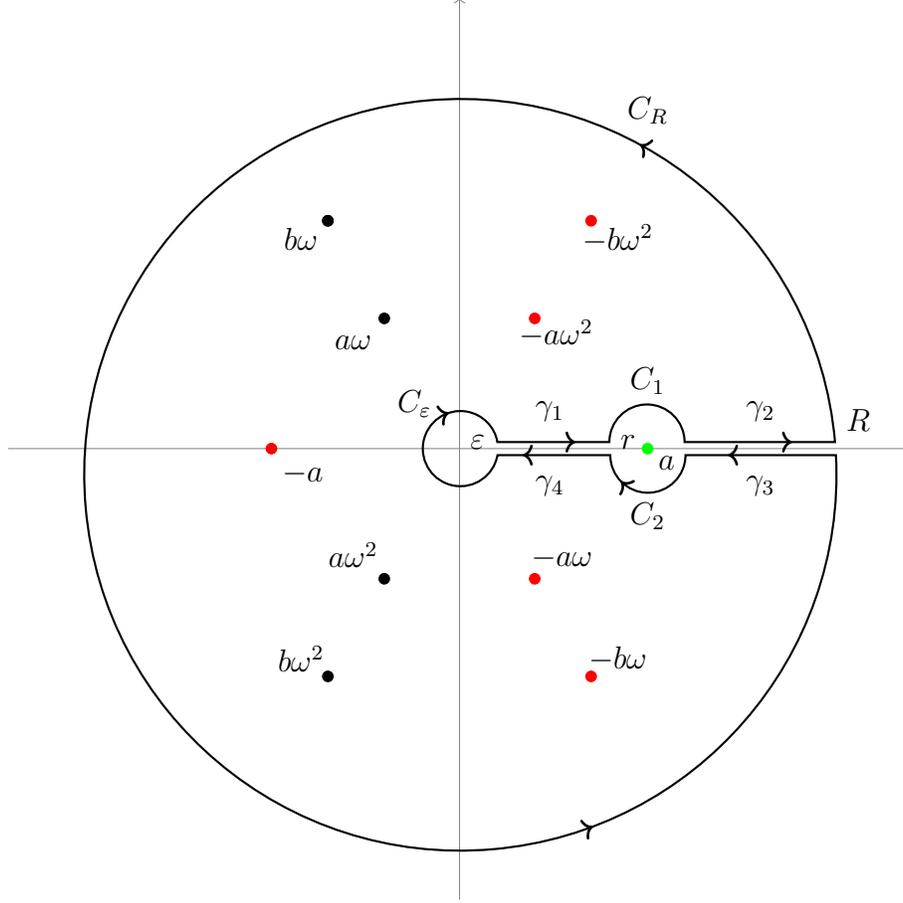
Note that $C_1$ is a circle of radius $r>0$ with centre at $z=a$, and the path of integration along $C_1$ can be parametrized by $z=a+re^{i\theta}$, where $\theta$ varies from $\pi^-$ to $0^+$.  Thus,  taking the limit $r\to0^+$, the integral \eqref{cont:3} along $C_1$ is given by
\begin{align}
&\lim_{r\to0^+}\int_{\pi^-}^{0^+}\frac{(a+re^{i\theta})(2a+re^{i\theta})\log^{k+1}(a+re^{i\theta})}{\big((a+re^{i\theta})^3-a^3\big)\big((a+re^{i\theta})^2+b(a+re^{i\theta})+b^2\big)}\cdot ire^{i\theta}\dd\theta\nonumber\\
=&\lim_{r\to0^+}\int_{\pi^-}^{0^+}\frac{(a+re^{i\theta})(2a+re^{i\theta})\log^{k+1}(a+re^{i\theta})}{\big(r^3e^{3i\theta}+3ar^2e^{2i\theta}+3a^2re^{i\theta}\big)\big(r^2e^{2i\theta}+(2a+b)re^{i\theta}+a^2+ab+b^2\big)}\cdot ire^{i\theta}\dd\theta\nonumber\\
\label{cont:limit}=&\lim_{r\to0^+}\int_{\pi^-}^{0^+}\frac{(a+re^{i\theta})(2a+re^{i\theta})\log^{k+1}(a+re^{i\theta})}{\big(r^2e^{2i\theta}+3are^{i\theta}+3a^2\big)\big(r^2e^{2i\theta}+(2a+b)re^{i\theta}+a^2+ab+b^2\big)}\cdot i\dd\theta
\end{align}
Let us denote by $h(r,\theta)$, the integrand 
\[
h(r,\theta)=i\cdot\frac{(a+re^{i\theta})(2a+re^{i\theta})\log^{k+1}(a+re^{i\theta})}{\big(r^2e^{2i\theta}+3are^{i\theta}+3a^2\big)\big(r^2e^{2i\theta}+(2a+b)re^{i\theta}+a^2+ab+b^2\big)}
\]
Note that the functions $(a+re^{i\theta})$, $(2a+re^{i\theta})$ and $\log^{k+1}(a+re^{i\theta})$ are continuous for $0<r<1$ and $\theta\in(0,\pi)$. This means we can find upper bounds for these functions in the bounded region  $0<r<1$ and $\theta\in(0,\pi)$. For instance, using the triangle inequality, we have $|a+re^{i\theta}|\le|a|+|1|=a+1$ and $|2a+re^{i\theta}|<2a+1$. We can also write 
\[
|\log{(a+re^{i\theta})}|\le\big|\log|a+re^{i\theta}|+i\pi\big|\le\sqrt{\pi^2+\log^2(a+1)}.
\]
Next, we find lower bounds for the terms in the denominator of $h(r,\theta)$. Let $y=re^{i\theta}$. Note that the quadratic terms
\[
y^2+3ay+3a^2,
\]
and
\[
y^2+(2a+b)y+a^2+ab+b^2,
\]
have roots of absolute value $\sqrt3a$ and $\sqrt{a^2+ab+b^2}$ respectively. Choose a positive real number $r_0>0$ such that 
\[
r_0<\min\{\sqrt3a,\sqrt{a^2+ab+b^2}\}.
\]
Therefore, if $0<r<r_0$, and $0<\theta<\pi$, then for $y=re^{i\theta}$, we have
\[
|y^2+3ay+3a^2|\ge|\sqrt3a-r_0|^2>0,
\]
and
\[
|y^2+(2a+b)y+a^2+ab+b^2|\ge|\sqrt{a^2+ab+b^2}-r_0|^2>0,
\]
giving us the required lower bounds. Consequently, we can write
\[
|h(r,\theta)|\le\frac{(2a+1)(a+1)\sqrt{\pi^2+\log^2(a+1)}}{(\sqrt3a-r_0)^2(\sqrt{a^2+ab+b^2}-r_0)^2},
\]
for $0<r<\min\{1,r_0\}$ and $0<\theta<\pi$. We have bounded the integrand in \eqref{cont:limit} by a constant function, independent of $r$ and $\theta$. Thus, using the dominated convergence theorem, we may bring the limit inside the integral in \eqref{cont:limit} to obtain
\begin{align}
&\lim_{r\to0^+}\int_{\pi^-}^{0^+}\frac{(a+re^{i\theta})(2a+re^{i\theta})\log^{k+1}(a+re^{i\theta})}{\big(r^2e^{2i\theta}+3are^{i\theta}+3a^2\big)\big(r^2e^{2i\theta}+(2a+b)re^{i\theta}+a^2+ab+b^2\big)}\cdot i\dd\theta\nonumber\\
&\qquad\qquad=\int_\pi^0\lim_{r\to0^+}\frac{(a+re^{i\theta})(2a+re^{i\theta})\log^{k+1}(a+re^{i\theta})}{\big(r^2e^{2i\theta}+3are^{i\theta}+3a^2\big)\big(r^2e^{2i\theta}+(2a+b)re^{i\theta}+a^2+ab+b^2\big)}\cdot i\dd\theta\nonumber\\
&\qquad\qquad=\int_\pi^0\frac{(a)(2a)\log^{k+1}(a)}{(3a^2)(a^2+ab+b^2)}\cdot i\dd\theta\nonumber\\
\label{cont:C1}&\qquad\qquad=-2\pi i\frac{\log^{k+1}a}{3(a^2+ab+b^2)}.
\end{align} 
%\[
%|\big(r^2e^{2i\theta}+3are^{i\theta}+3a^2\big)|\ge|-r^2-3ar+3a^2|\ge|3a^2-3a-1|.
%\]
%\[
%-2\pi i\frac{\log^{k+1}a}{3(a^2+ab+b^2)},
%\]
A similar argument will show that the contribution of integral \eqref{cont:3} along $C_2$ gives
\begin{equation}\label{cont:C2}
-2\pi i\frac{(\log a+2\pi i)^{k+1}}{3(a^2+ab+b^2)}.
\end{equation}
The contribution towards the imaginary terms of the sum of \eqref{cont:C1} and \eqref{cont:C2} is
\[
\frac{-2\pi i}{3(a^2+ab+b^2)}\left[ \log^{k+1}a+\sum_{j=0,\text{ even}}^{k+1}(-1)^{\frac j2}{{k+1}\choose{j}}(2\pi)^j\log^{k+1-j}a \right].
\]
The integrals along $C_R$ and $C_\varepsilon$ will be zero and the purely imaginary terms corresponding to the remaining paths $\gamma_1,\gamma_2,\gamma_3,\gamma_4,$ gives an expression similar to \eqref{eq:intk}, but with $g_1(k)$:
\begin{equation}\label{eq:recg1}
-3i\theta(k+1)g_1(k)-\sum_{j>1,\text{ odd}}^{k+1}{{k+1}\choose j}(3\theta)^j\,i^j\,g_1(k+1-j).
\end{equation}
The next step is calculating the residues which are obtained at $a\omega,a\omega^2,b\omega$ and $b\omega^2$ (denoted by black dots in Figure \ref{fig:contour3}.  The residue is given by
\begin{multline*}
\frac{2\pi i}{3(a^3-b^3)}\Biggl[(a\omega^2-b\omega)(\log a+i\theta)^{k+1}-(a\omega-b\omega^2)(\log a+2i\theta)^{k+1}\\
+ (b\omega+a)(\omega^2-1)(\log b+i\theta)^{k+1}+(b\omega^2+a)(\omega-1)(\log b+2i\theta)^{k+1}\Biggr].
\end{multline*}
Collecting the imaginary terms together, we obtain from above
\begin{multline}\label{eq:resg1}
\frac{2\pi i}{6(a^3-b^3)}\Biggl[ (a-b)\sum_{j=0,\text{ even}}^{k+1}(-1)^{\frac j2}\theta^j{{k+1}\choose{j}}(2^j+1)\Bigl(\log^{k+1-j}a-3\log^{k+1-j}b\Bigr)\\
+\sqrt3(a+b)\sum_{j=1,\text{ odd}}^{k+1}(-1)^{\frac{j-1}2}\theta^j{{k+1}\choose{j}}(2^j-1)\Bigl(\log^{k+1-j}a-\log^{k+1-j}b\Bigr)\Biggr].
\end{multline}
Finally, equating \eqref{eq:recg1} and \eqref{eq:resg1} and rearranging, we have
\begin{multline}\label{eq:g2sum}
g_1(k)=\frac1{k+1}\sum_{j>1,\text{ odd}}^{k+1}(-1)^{\frac{j+1}2}{{k+1}\choose j}(3\theta)^{j-1}g_1(k+1-j)\\
-\frac{1}{6(k+1)(a^2+ab+b^2)}\left[\sum_{j=0,\text{ even}}^{k+1}(-1)^{\frac j2}\theta^j{{k+1}\choose{j}}(2^j+1)(\log^{k+1-j}a-3\log^{k+1-j}b)\right]\\
-\frac{a+b}{2\sqrt3(k+1)(a^3-b^3)}\left[\sum_{j=1,\text{ odd}}^{k+1}(-1)^{\frac{j-1}2}\theta^j{{k+1}\choose{j}}(2^j-1)(\log^{k+1-j}a-\log^{k+1-j}b)\right]\\
-\frac{1}{3(k+1)(a^2+ab+b^2)}\left[ \log^{k+1}a+\sum_{j=0,\text{ even}}^{k+1}(-1)^{\frac j2}{{k+1}\choose{j}}(2\pi)^j\log^{k+1-j}a \right].
\end{multline}
Note that this time in \eqref{eq:g2sum}, we have an additional sum along with the usual three parts. To incorporate this extra sum, we need to define another family of polynomials with rational coefficients:
\begin{multline*}%\label{eq:yk}
Y_k(x)=\frac1{k+1}\sum_{j>1,\text{ odd}}^{k+1}(-1)^{\frac{j+1}2}{{k+1}\choose j}3^{j-1}Y_{k+1-j}(x)\\
-\frac{1}{(k+1)}\left[x^{k+1}+\sum_{j=0,\text{ even}}^{k+1}(-1)^{\frac{j}2}{{k+1}\choose{j}}3^jx^{k+1-j}\right].
\end{multline*}
Here are the first few examples of the polynomial $Y_k(x)$:
\renewcommand{\arraystretch}{2.5}
\renewcommand{\tabcolsep}{0.5cm}
\[\begin{array}{|c|}
\hline 
Y_k(x)\\
\hline
\displaystyle Y_0(x)=-2x \\
\hline
\displaystyle Y_1(x)=-x^2+\frac{9}2\\
\hline
\displaystyle Y_2(x)=-\frac{2x^3}3+3x \\
\hline
\displaystyle Y_3(x)=-\frac{x^4}2+\frac{9x^2}2+\frac{81}4.\\
\hline
\end{array}\]
This means, using the definitions of the polynomials $R_k$ in \eqref{eq:rk} and $S_k$ in \eqref{eq:sk}, we can write
\begin{multline}\label{eq:gk}
g_1(k)=\frac{\theta^{k+1}}{3(a^2+ab+b^2)}\left[-R_k\left(\frac{\log a}{\theta}\right)+3R_k\left(\frac{\log b}{\theta}\right)+Y_k\left(\frac{\log a}{\theta}\right)\right]\\+\frac{\theta^{k+1}(a+b)}{\sqrt3(a^3-b^3)}\left[S_k\left(\frac{\log a}{\theta}\right)-S_k\left(\frac{\log b}{\theta}\right)\right].\\
\end{multline}
In our calculations in Section \ref{sec:compute}, $a$ will invariably be equal to 1, so that all the $\log a$ terms above vanish. \\
\subsection{The integral $g_2(k)$}
\label{sec:g2}
The final integral we will consider is
\[
g_2(k)=\int_0^\infty \frac{t(t-a)\log^k t}{(t^3+a^3)(t^2-bt+b^2)}\dd t,
\]
again with $a,b\in\R^+$, and the corresponding contour integral
\[
\oint_C \frac{z(z-a)\log^{k+1} z}{(z^3+a^3)(z^2-bz+b^2)} \dd z.
\]
In this case the integral along $C_1$ and $C_2$ is zero, however there is an additional residue contribution from the pole at $z=-a$.  The contour integral 
evaluates, as before, to
\begin{equation}\label{eq:recg2}
-3i\theta(k+1)g_2(k)-\sum_{j>1,\text{ odd}}^{k+1}{{k+1}\choose j}(3\theta)^j\,i^j\,g_2(k+1-j),
\end{equation}
and the residues at $z=-a\omega,-a\omega^2,-b\omega,-b\omega^2$ and $-a$ (denoted by red dots in Figure \ref{fig:contour3}) are given by
\begin{multline*}
\frac{2\pi i}{3(a^3-b^3)}\Biggl[-(\log a+5i\delta)^{k+1}(a\omega^2-b\omega)-(\log a+i\delta)^{k+1}(a\omega-b\omega^2)\\
+(\log b+5i\delta)^{k+1}(\omega^2-1)(b\omega+a)+(\log b+i\delta)^{k+1}(\omega-1)(b\omega^2+a)\Biggr]\\
+\frac{4\pi i(\log a+\pi i)^{k+1}}{3(a^2+ab+b^2)}.
\end{multline*}
Again,  we collect the purely imaginary terms in the above expression
\begin{multline}\label{eq:resg2}
\frac{2\pi i}{6(a^3-b^3)}\Biggl[(a-b)\sum_{j=0,\text{ even}}^{k+1}{{k+1}\choose{j}}i^{j}\delta^j(5^j+1)(\log^{k+1-j}a-3\log^{k+1-j}b)\\
-\sqrt3(a+b)\sum_{j=1,\text{ odd}}^{k+1}{{k+1}\choose{j}}i^{j-1}\delta^j(5^j-1)(\log^{k+1-j}a-\log^{k+1-j}b) \Biggr]\\
+\frac{4\pi i}{3(a^2+ab+b^2)}\sum_{j=0,\text{ even}}^{k+1}{{k+1}\choose{j}}i^{j}\pi^j\log^{k+1-j}a.
\end{multline}
Therefore, by equating \eqref{eq:recg2} and \eqref{eq:resg2}, we get
\begin{multline}\label{eq:g2term}
g_2(k)=\frac1{k+1}\sum_{j>1,\text{ odd}}^{k+1}(-1)^{\frac{j+1}2}{{k+1}\choose j}(3\theta)^{j-1}g_2(k+1-j)\\
-\frac{1}{6(k+1)(a^2+ab+b^2)}\left[\sum_{j=0,\text{ even}}^{k+1}(-1)^{\frac j2}\delta^j{{k+1}\choose{j}}(5^j+1)(\log^{k+1-j}a-3\log^{k+1-j}b)\right]\\
+\frac{a+b}{2\sqrt3(k+1)(a^3-b^3)}\left[\sum_{j=1,\text{ odd}}^{k+1}(-1)^{\frac{j-1}2}\delta^j{{k+1}\choose{j}}(5^j-1)(\log^{k+1-j}a-\log^{k+1-j}b)\right]\\
-\frac{2}{3(k+1)(a^2+ab+b^2)}\left[\sum_{j=0,\text{ even}}^{k+1}(-1)^{\frac j2}{{k+1}\choose{j}}\pi^j\log^{k+1-j}a \right].
\end{multline}
Once again, owing to the presence of an additional sum in \eqref{eq:g2term},we define another family of polynomials with rational coefficients, defined by the recursion:
\begin{multline*}%\label{eq:zk}
Z_k(x)=\frac1{k+1}\sum_{j>1,\text{ odd}}^{k+1}(-1)^{\frac{j+1}2}{{k+1}\choose j}6^{j-1}Z_{k+1-j}(x)\\
-\frac{2}{(k+1)}\sum_{j=0,\text{ even}}^{k+1}(-1)^{\frac j2}{{k+1}\choose{j}}3^jx^{k+1-j}.
\end{multline*}
Here are the first few examples of the polynomial $Z_k(x)$:
\renewcommand{\arraystretch}{2.5}
\renewcommand{\tabcolsep}{0.5cm}
\[\begin{array}{|c|}
\hline 
Z_k(x)\\
\hline
\displaystyle Z_0(x)=-2x \\
\hline
\displaystyle Z_1(x)=-x^2+9\\
\hline
\displaystyle Z_2(x)=-\frac{2x^3}3-6x \\
\hline
\displaystyle Z_3(x)=-\frac{x^4}2-9x^2+\frac{567}4.\\
\hline
\end{array}\]
Finally, using this and the polynomials $P_k$ and $Q_k$ defined in \eqref{eq:pk} and \eqref{eq:qk} respectively, we have
\begin{multline}\label{eq:g2k}
g_2(k)=\frac{\delta^{k+1}}{3(a^2+ab+b^2)}\left[-P_k\left(\frac{\log a}{\delta}\right)+3P_k\left(\frac{\log b}{\delta}\right)+Z_k\left(\frac{\log a}{\delta}\right)\right]\\+\frac{\delta^{k+1}(a+b)}{\sqrt3(a^3-b^3)}\left[Q_k\left(\frac{\log a}{\delta}\right)-Q_k\left(\frac{\log b}{\delta}\right)\right].
\end{multline}
\subsection{Dependence on the sign of $a$ and $b$}\label{sec:signindep}
We end this section with a discussion on the sums of the integrals $f_1(k)$ and $f_2(k)$ as well as $g_1(k)$ and $g_2(k)$. We also investigate how these integrals depend on the sign of the parameters $a$ and $b$.\\

First, we note that the results for integrals $f_1(k)$ and $f_2(k)$ from equations \eqref{eq:fk} and \eqref{eq:f2k} respectively, only work when both the parameters $a$ and $b$ are positive. Consider the sum $f_1(k)+f_2(k)$ given by
\begin{align}\label{eq:sumf}
f_1(k)+f_2(k)&=\int_0^\infty\left(\frac{t\log^kt}{(t^2+at+a^2)(t^2+bt+b^2)}+\frac{t\log^kt}{(t^2-at+a^2)(t^2-bt+b^2)}\right)\dd t\\
&=\int_0^\infty\frac{2t\log^kt\cdot(t^4+(a^2+b^2+ab)t^2+a^2b^2)}{(t^2+at+a^2)(t^2-at+a^2)(t^2+bt+b^2)(t^2-bt+b^2)}\dd t.\nonumber
\end{align}
We already know this integral evaluates to the sum of the equations \eqref{eq:fk} and \eqref{eq:f2k} (corresponding to $f_1(k)$ and $f_2(k)$  respectively). However, we may also use the same procedure of computing a contour integral of the form
\[
\oint_C\frac{2z\log^{k+1}z(z^4+(a^2+b^2+ab)z^2+a^2b^2)}{(z^2+az+a^2)(z^2-az+a^2)(z^2+bz+b^2)(z^2-bz+b^2)}\dd z,
\]
with an appropriate keyhole contour $C$, as done in Sections \ref{sec:f1} and \ref{sec:f2} to compute the integrals $f_1(k)$ and $f_2(k)$. In this case, the residues are obtained at the poles of the integrand which will be located at $z=\pm a\omega,\pm a\omega^2,\pm b\omega,\pm b\omega^2$. We remark that the location of these poles does not change on changing the sign of $a$ or $b$ from positive to negative.  Therefore, the evaluation of the residues and the entire contour integration remains the same even if $a$ or $b$ is negative. This means that the expression for $f_1(k)+f_2(k)$ is the same if $a$ or $b$ is negative, with the understanding that $\log a$ and $\log b$ would be replaced by $\log|a|$ and $\log|b|$ respectively. A similar argument shows that the same is true for the expression for 
\begin{align}\label{eq:sumg}
g_1(k)+g_2(k)=\int_0^\infty\left( \frac{t(t+a)\log^k t}{(t^3-a^3)(t^2+bt+b^2)}+\frac{t(t-a)\log^k t}{(t^3+a^3)(t^2-bt+b^2)}\right)\dd t,
\end{align}
obtained by adding equations \eqref{eq:gk} and \eqref{eq:g2k} (corresponding to $g_1(k)$ and $g_2(k)$  respectively).\\

Thus, despite the results for $f_1(k),f_2(k),g_1(k)$ and $g_2(k)$ being individually valid only for $a,b>0$, the following expression for the sum $f_1(k)+f_2(k)$,
\begin{align}\label{eq:sumffinal}
f_1(k)+f_2(k)=&\frac{\theta^{k+1}(a+b)}{a^3-b^3}\Bigg[R_k\left(\frac{\log|a|}\theta\right)-R_k\left(\frac{\log|b|}\theta\right)\Bigg]\nonumber\\
&+\frac{\theta^{k+1}}{\sqrt3(a^2+ab+b^2)}\Bigg[S_k\left(\frac{\log|a|}\theta\right)+S_k\left(\frac{\log|b|}\theta\right)\Bigg]\nonumber\\
&+\frac{\delta^{k+1}(a+b)}{a^3-b^3}\Bigg[P_k\left(\frac{\log|a|}\delta\right)-P_k\left(\frac{\log|b|}\delta\right)\Bigg]\nonumber\\
&+\frac{\delta^{k+1}}{\sqrt3(a^2+ab+b^2)}\Bigg[Q_k\left(\frac{\log|a|}\delta\right)+Q_k\left(\frac{\log|b|}\delta\right)\Bigg],
\end{align}
is valid for $a,b\in\R^\times$, and so is the expression for the sum $g_1(k)+g_2(k)$,
\begin{align}\label{eq:sumgfinal}
g_1(k)+g_2(k)=&\frac{\theta^{k+1}}{3(a^2+ab+b^2)}\left[-R_k\left(\frac{\log|a|}{\theta}\right)+3R_k\left(\frac{\log |b|}{\theta}\right)+Y_k\left(\frac{\log |a|}{\theta}\right)\right]\nonumber\\
&+\frac{\theta^{k+1}(a+b)}{\sqrt3(a^3-b^3)}\left[S_k\left(\frac{\log |a|}{\theta}\right)-S_k\left(\frac{\log |b|}{\theta}\right)\right]\nonumber\\
&+\frac{\delta^{k+1}}{3(a^2+ab+b^2)}\left[-P_k\left(\frac{\log |a|}{\delta}\right)+3P_k\left(\frac{\log |b|}{\delta}\right)+Z_k\left(\frac{\log |a|}{\delta}\right)\right]\nonumber\\
&+\frac{\delta^{k+1}(a+b)}{\sqrt3(a^3-b^3)}\left[Q_k\left(\frac{\log |a|}{\delta}\right)-Q_k\left(\frac{\log |b|}{\delta}\right)\right],
\end{align}
where, as before, $\theta=2\pi/3$ and $\delta=\pi/3$. 
\subsection{An alternate method}\label{sec:alternat}
In this section, we provide an alternate method to calculate the integrals $f_1(k),f_2(k),g_1(k)$ and $g_2(k)$ evaluated in the previous sections. The advantage of the following method is that it sheds light on the explicit form of the polynomials $R_k,S_k,P_k,Q_k,Y_k$ and $Z_k$ and we obtain a non-recursive expression. This section serves as an analogue of \cite[Lemma 5.2]{LL} to Lalín's result in \cite{nvar}.
\subsubsection{The integral $f_1(k)$}\label{sec:altf1}
Consider the integral
\[
I_1(\beta)=\int_0^\infty \frac{x^\beta}{(x^2+ax+a^2)(x^2+bx+b^2)}\dd x,
\]
where $a,b>0$ are real numbers, and let $k$ be any non-negative integer. The value of the $k^{\text{th}}$-derivative $I_1^{(k)}(\beta)$,with respect to $\beta$, gives the value of the integral
\[
f_1(k)=\int_0^\infty\frac{t\log^kt}{(t^2+at+a^2)(t^2+bt+b^2)}\dd t,
\]
under the limit $\beta\to1$.\\

Integral $I_1(\beta)$ can be evaluated using residue theorems (see \cite[Chapter 4, Section 5.3.4]{ahlfors} or \cite[Remark 2, p.129]{keckic}). The residues are calculated at $x=\omega a,\ol\omega a,\omega b, \ol \omega b$, where $\omega=e^{2\pi i/3}$. For instance, the residue at $x=\omega a$ is given by
\[
\frac{(\omega a)^{\beta-1}}{(a-b)(\omega-\ol\omega)(\omega a-\ol\omega b)}.
\] 
Then, using residue theorems, $I_1(\beta)$ is given by the sum of residues multiplied by a factor of $\frac{2\pi i}{1-e^{2\pi i \beta}}$, which gives
\begin{equation}\label{residue}
I_1(\beta)=\frac{2\pi i}{1-e^{2\pi i \beta}}\cdot\frac{1}{(a-b)(\sqrt3i)}\cdot\left[\frac{(a^\beta-b^\beta)\omega^{\beta-2}(1-\omega^{\beta-2})-ab(a^{\beta-2}-b^{\beta-2})\omega^\beta(1-\omega^\beta)}{a^2+ab+b^2}\right].
\end{equation}
Using the factorization of $(1-x^3)$, we have
\[
1-e^{2\pi i \beta}=(1-\omega^\beta)(1-\omega^{\beta-2})(1-\omega^{\beta+2}),
\]
and using this to simplify \eqref{residue}, we get
\[
I_1(\beta)=\frac{2\pi}{\sqrt3(a^3-b^3)}\left[ \frac{(a^\beta-b^\beta)\omega^{\beta-2}}{(1-\omega^\beta)(1-\omega^{\beta+2})}-\frac{ab(a^{\beta-2}-b^{\beta-2})\omega^\beta}{(1-\omega^{\beta-2})(1-\omega^{\beta+2})}  \right].
\]
We may rewrite this as follows:
\begin{align*}
\frac{(a^\beta-b^\beta)\omega^{\beta-2}}{(1-\omega^\beta)(1-\omega^{\beta+2})}&=\frac{a^\beta-b^\beta}{(\omega^{-\beta/2}-\omega^{\beta/2})(\omega^{-\beta/2-1}-\omega^{\beta/2+1})}\\
&=\frac{a^\beta-b^\beta}{-4\sin(\theta{\frac\beta2})\sin(\theta({\frac\beta2+1}))}\\
&=\frac{b^\beta-a^\beta}{4}\csc\left(\theta{\frac\beta2}\right)\csc\left(\theta\frac\beta2+\theta\right)\\
&=\frac{a^\beta-b^\beta}{4}\csc\left(\frac\theta2(\beta-1)+\frac\theta2\right)\csc\left(\frac\theta2(\beta-1)\right),
\end{align*}
where $\theta=2\pi/3$, and we have used $\csc\left(\theta\frac\beta2+\theta\right)=-\csc\left(\frac\theta2(\beta-1)\right)$. Similarly,
\begin{align*}
\frac{ab(a^{\beta-2}-b^{\beta-2})\omega^\beta}{(1-\omega^{\beta-2})(1-\omega^{\beta+2})}&=\frac{ab(a^{\beta-2}-b^{\beta-2})}{(\omega^{-\beta/2+1}-\omega^{\beta/2-1})(\omega^{-\beta/2-1}-\omega^{\beta/2+1})}\\
&=\frac{ab(a^{\beta-2}-b^{\beta-2})}{-4\sin(\theta(\frac\beta2-1))\sin(\theta(\frac\beta2+1))}\\
&=\frac{ba^{\beta-1}-ab^{\beta-1}}{-4}\csc\left(\theta\frac\beta2-\theta\right)\csc\left(\theta\frac\beta2+\theta\right)\\
&=\frac{ba^{\beta-1}-ab^{\beta-1}}{4}\csc\left(\frac\theta2(\beta-1)-\frac\theta2\right)\csc\left(\frac\theta2(\beta-1)\right).
\end{align*}
Thus,
\begin{multline}\label{eq:I1}
I_1(\beta)=\frac{\pi}{2\sqrt3(a^3-b^3)}\left[(a^\beta-b^\beta)\csc\left(\frac\theta2(\beta-1)+\frac\theta2\right)\csc\left(\frac\theta2(\beta-1)\right)\right.\\
\left.-(ba^{\beta-1}-ab^{\beta-1})\csc\left(\frac\theta2(\beta-1)-\frac\theta2\right)\csc\left(\frac\theta2(\beta-1)\right)\right].
\end{multline}
We wish to write the power series expansion of $I_1(\beta)$ in a neighbourhood of $\beta=1$. Note that the function $\csc\left(\frac\theta2(\beta-1)\right)$ has a simple pole at $\beta=1$, while all the other functions are holomorphic in a neighbourhood of $\beta=1$. Also observe that if we write the Laurent series expansion around $\beta=1$,
\begin{equation}\label{expansion:1}
\csc\left(\frac\theta2(\beta-1)-\frac\theta2\right)\csc\left(\frac\theta2(\beta-1)\right)=\sum_{j=-1}^\infty v_j\theta^j(\beta-1)^j,
\end{equation}
with $v_j\in\R$, then we have
\[
\csc\left(\frac\theta2(\beta-1)+\frac\theta2\right)\csc\left(\frac\theta2(\beta-1)\right)=\sum_{j=-1}^\infty (-1)^jv_j\theta^j(\beta-1)^j.
\]
Now, the power series expansion of $a^{\beta-1}$ is given by
\[
a^{\beta-1}=\sum_{j=0}^\infty\frac{\log^j a}{j!}(\beta-1)^j.
\]
This means
\[
a^{\beta-1}\cdot\csc\left(\frac\theta2(\beta-1)-\frac\theta2\right)\csc\left(\frac\theta2(\beta-1)\right)=\sum_{k=-1}^\infty \left(\sum_{j=0}^{k+1}\frac{\log^j a}{j!}\,v_{k-j}\theta^{k-j}\right) (\beta-1)^k.
\]
Using this idea, we may write 
\begin{equation*}
I_1(\beta)=\frac{\pi}{2\sqrt3(a^3-b^3)}\sum_{k=-1}^\infty w_k(\beta-1)^k,
\end{equation*}
where
\[
w_k=\sum_{j=0}^{k+1}\frac{(-1)^{k-j}(a\log^ja-b\log^jb)-b\log^ja+a\log^jb}{j!}v_{k-j}\theta^{k-j}.
\]
We note that $w_{-1}=0$, which means $I_1(\beta)$ is a holomorphic function around $\beta=1$. For $k\ge0$, $w_k$ represents the $k^{\text{th}}$-derivative of $I_1(\beta)$ at $\beta=1$, which tells us that
\[
f_1(k)=I_1^{(k)}(1)=\int_0^\infty\frac{t\log^kt}{(t^2+at+a^2)(t^2+bt+b^2)}\dd t=\frac{\pi(k!)}{2\sqrt3(a^3-b^3)}w_k.
\]
We can rewrite $w_k$ as follows:
\begin{align*}
w_k &=\sum_{j=0}^{k+1}\frac{(-1)^{k-j}(a\log^ja-b\log^jb)-b\log^ja+a\log^jb}{j!}v_{k-j}\theta^{k-j}\\
       &=\theta^k\sum_{j=0}^{k+1}\left[(a-(-1)^{k-j}b)\left(\frac{\log b}{\theta}\right)^j-(b-(-1)^{k-j}a)\left(\frac{\log a}{\theta}\right)^j\right]\frac{v_{k-j}}{j!}.
\end{align*}
Suppose $k=2n$ is even. Then, splitting the cases where $j$ is even or odd, we obtain
\begin{multline*}
w_{2n}=\theta^{2n}(a-b)\sum_{j=0, \text{ even}}^{2n+1}\left[\left(\frac{\log b}{\theta}\right)^j+\left(\frac{\log a}{\theta}\right)^j\right]\frac{v_{2n-j}}{j!}\\
+\theta^{2n}(a+b)\sum_{j=1, \text{ odd}}^{2n+1}\left[\left(\frac{\log b}{\theta}\right)^j-\left(\frac{\log a}{\theta}\right)^j\right]\frac{v_{2n-j}}{j!}.
\end{multline*}
Recall that the coefficients $v_j$ are obtained via the power series expansion \eqref{expansion:1}, and moreover, for any real number $x$, the hyperbolic functions $\sinh(x(\beta-1))$ and $\cosh(x(\beta-1))$ are given by
\[
\sinh(x(\beta-1))=\sum_{j=1,\text{ odd}}^\infty\frac{x^j}{j!}(\beta-1)^j\qquad\text{and}\qquad\cosh(x(\beta-1))=\sum_{j=0,\text{ even}}^\infty\frac{x^j}{j!}(\beta-1)^j,
\]
around $\beta=1$.
This means that the sums
\[
 \sum_{j=1, \text{ odd}}^{2n+1}v_{2n-j}\cdot\frac{x^j}{j!}\qquad\text{and}\qquad\sum_{j=0, \text{ even}}^{2n+1}v_{2n-j}\cdot \frac{x^j}{j!}.
\]
represent the $(2n)^{\text{th}}$-coefficients of 
\[
\csc\left(\frac T2-\frac\theta2\right)\csc\left(\frac T2\right)\sinh(xT)\qquad\text{and}\qquad\csc\left(\frac T2-\frac\theta2\right)\csc\left(\frac T2\right)\cosh(xT),
\]
respectively. We write
\begin{equation}\label{eq:defnAm}
\csc\left(\frac T2-\frac\theta2\right)\csc\left(\frac T2\right)\sinh(xT)=\sum_{m=-1}^\infty A_m(x)\frac{T^m}{m!},
\end{equation}
and\begin{equation}\label{eq:defnBm}
\csc\left(\frac T2-\frac\theta2\right)\csc\left(\frac T2\right)\cosh(xT)=\sum_{m=-1}^\infty B_m(x)\frac{T^m}{m!},
\end{equation}

where $A_m(x)$ and $B_m(x)$ are polynomials in $x$ with real coefficients. Then,
\begin{align}
f_1(2n)&=\frac{\pi(2n)!}{2\sqrt3(a^3-b^3)}w_{2n}\nonumber\\
&=\label{eq:f1even}\frac{\theta^{2n+1}(a+b)}{a^3-b^3}\cdot\frac{\sqrt3}4\left[-A_{2n}\left(\frac{\log a}\theta\right)+A_{2n}\left(\frac{\log b}\theta\right)\right]\\
&\qquad\qquad+\frac{\theta^{2n+1}}{\sqrt3(a^2+ab+b^2)}\cdot\frac{3}4\left[B_{2n}\left(\frac{\log b}\theta\right)+B_{2n}\left(\frac{\log a}\theta\right)\right]\nonumber
\end{align}
Similarly, if $k=2n+1$ is odd, then 
\begin{multline*}
w_{2n+1}=\theta^{2n+1}(a+b)\sum_{j=0, \text{ even}}^{2n+2}\left[\left(\frac{\log b}{\theta}\right)^j-\left(\frac{\log a}{\theta}\right)^j\right]\frac{v_{2n+1-j}}{j!}\\
+\theta^{2n+1}(a-b)\sum_{j=1, \text{ odd}}^{2n+2}\left[\left(\frac{\log b}{\theta}\right)^j+\left(\frac{\log a}{\theta}\right)^j\right]\frac{v_{2n+1-j}}{j!},
\end{multline*}
and
\begin{multline}\label{eq:f1odd}
f_1(2n+1)=\frac{\theta^{2n+2}(a+b)}{a^3-b^3}\cdot\frac{\sqrt3}4\left[-B_{2n+1}\left(\frac{\log a}\theta\right)+B_{2n+1}\left(\frac{\log b}\theta\right)\right]\\
+\frac{\theta^{2n+2}}{\sqrt3(a^2+ab+b^2)}\cdot\frac{3}4\left[A_{2n+1}\left(\frac{\log b}\theta\right)+A_{2n+1}\left(\frac{\log a}\theta\right)\right].
\end{multline}

Recall that the polynomials $A_m(x)$ and $B_m(x)$ are obtained as coefficients of the functions $\csc\left(\frac T2-\frac\theta2\right)\csc\left(\frac T2\right)\sinh(xT)$ and $\csc\left(\frac T2-\frac\theta2\right)\csc\left(\frac T2\right)\cosh(xT)$, and can be computed explicitly using standard convolution techniques. We have the following Taylor series in the variable $T$ around the point $T=0$:
\begin{equation}\label{eq:sinh}
\sinh (xT)=\sum_{j=0}^\infty \frac{x^{2j+1}}{(2j+1)!}\;T^{2j+1},
\end{equation}
and
\begin{equation}\label{eq:csc}
\csc\left(\frac T2\right)=\sum_{k=0}^\infty(-1)^{k-1}2\left(1-\frac1{2^{2k-1}}\right)\frac{\mathscr{B}_{2k}}{(2k)!}\;T^{2k-1},
\end{equation}
both of which consist of odd powers of $T$. Here $\mathscr{B}_j$ denotes the $j^{\text{th}}$-Bernoulli number, defined as coefficients of the Taylor series
\begin{equation*}\frac{x}{e^{x}-1}=\sum_{j=0}^\infty \frac{\mathscr{B}_j x^j}{j!}.\end{equation*}

For $\csc\left(\frac T2-\frac\theta2\right)$, we may write the Taylor series as
\begin{equation}\label{eq:shiftedcsc}
\csc\left(\frac T2-\frac\theta2\right)=\sum_{n=0}^\infty \frac{\mu_n}{n!}\; T^n, 
\end{equation}
where $\mu_n$ is the $n^\text{th}$ derivative evaluated at zero. We use Faà di Bruno's formula to calculate the $n^\text{th}$ derivative. Recall that if $f(x)$ and $g(x)$ are functions for which all necessary derivatives defined, then the $n^\text{th}$ derivative of $f\circ g(x)=f(g(x))$ is given by
\begin{equation*}
\frac{\mathrm{d^n}}{\mathrm{dx^n}}\; f(g(x))=\sum_{\substack{\text{partitions}\\\text{of $n$}}}\frac{n!}{k_1!k_2!\cdots k_n!} f^{(k)}(g(x))\left(\frac{g'(x)}{1!}\right)^{k_1}\left(\frac{g''(x)}{2!}\right)^{k_2}\cdots\left(\frac{g^{(n)}(x)}{n!}\right)^{k_n},
\end{equation*}
where the sum runs over the $n$-tuples of non-negative integers $(k_1,k_2,\dots,k_n)$ satisfying $k_1+2k_2+3k_3+\cdots+nk_n=n$ with $k:=k_1+k_2+\cdots+k_n$. In other words, the sum is over all partitions of $n$ with $k_j$ denoting the number of blocks of length $j$. Additionally, we define the following sums:
\[
\epsilon:=\sum_{\substack{j=1\\\text{$j$ odd}}}^nk_j, \quad\text{and}\quad\xi:=\sum_{\substack{j=1\\\text{$j\equiv0,-1\pmod4$}}}^nk_j.
\]
Taking 
\[
f(x)=x^{-1}\quad\text{ and} \quad g(x)=\sin\left(\frac T2-\frac\theta2\right),
\]
we obtain 
\[
f(g(T))=\left[\sin\left(\frac T2-\frac\theta2\right)\right]^{-1}=\csc\left(\frac T2-\frac\theta2\right)=\sum_{n=0}^\infty \frac{\mu_n}{n!}\; T^n.
\]
Using Faà di Bruno's formula we can write the $n^\text{th}$ derivative $\mu_n$  as
\begin{align}
\mu_{n}&=\left.\frac{\mathrm{d^{n}}}{\mathrm{dT^{n}}}\; \left[\sin\left(\frac T2-\frac\theta2\right)\right]^{-1}\right|_{T=0}\nonumber\\
&=\sum_{\substack{\text{partitions}\\\text{of $n$} }}\frac{n!}{k_1!k_2!\cdots k_n!} \left[\frac{(-1)^kk!}{\sin^{k+1}\left(-\frac\theta2\right)}\right]\left(\frac{\frac12\cos\left(-\frac\theta2\right)}{1!}\right)^{k_1}\left(\frac{\frac{-1}4\sin\left(-\frac\theta2\right)}{2!}\right)^{k_2}\cdots\left(\frac{\frac{\mathrm{d^{n}}}{\mathrm{dT^{n}}}\left.\sin\left(\frac T2-\frac\theta2\right)\right|_{T=0}}{n!}\right)^{k_{n}}\nonumber\\
&=\sum_{\substack{\text{partitions}\\\text{of $n$} }}\frac{n!}{k_1!k_2!\cdots k_n!} \left[\frac{(-1)^kk!}{\left(-\frac{\sqrt3}2\right)^{k+1}}\right]\,\left(\frac{\frac12\cdot\frac12}{1!}\right)^{k_1}\left(\frac{\frac1{2^2}\cdot\frac{\sqrt3}2}{2!}\right)^{k_2}\left(\frac{-\frac1{2^3}\cdot\frac{1}2}{3!}\right)^{k_3}\left(\frac{-\frac1{2^4}\cdot\frac{\sqrt3}2}{4!}\right)^{k_4}\cdots\nonumber\\
&=\sum_{\substack{\text{partitions}\\\text{of $n$} }}\frac{n!}{k_1!k_2!\cdots k_n!} \left[\frac{(-1)^kk!}{\left(-\frac{\sqrt3}2\right)^{k+1}}\right]\,\frac{(-1)^{k_3+k_4+k_7+k_8+\cdots}}{2^{k_1+2k_2+\cdots+nk_{n}}}\left(\frac{\frac12}{1!}\right)^{k_1}\left(\frac{\frac{\sqrt3}2}{2!}\right)^{k_2}\left(\frac{\frac12}{3!}\right)^{k_3}\left(\frac{\frac{\sqrt3}2}{4!}\right)^{k_4}\cdots\nonumber\\
&=\sum_{\substack{\text{partitions}\\\text{of $n$} }}\frac{n!}{k_1!k_2!\cdots k_n!} \left[\frac{(-1)2^{k+1}k!}{(\sqrt3)^{k+1}}\right]\,\frac{(-1)^{\xi}}{2^{n}}\cdot\frac1{2^k}\cdot\frac{(\sqrt3)^{k_2+k_4+k_6+\cdots}}{(1!)^{k_1}(2!)^{k_2}\cdots(n!)^{k_{n}}}\nonumber\\
&\label{eq:defnmun}=-\frac{1}{2^{n-1}\sqrt3}\sum_{\substack{\text{partitions}\\\text{of $n$} }}\frac{n!\;k!}{k_1!k_2!\cdots k_n!} \left(\frac{(-1)^{\xi}}{(\sqrt3)^{\epsilon}}\right)\cdot\frac{1}{(1!)^{k_1}(2!)^{k_2}\cdots(n!)^{k_{n}}}.
\end{align}
%\begin{align*}
%\mu_{n}&=\left.\frac{\mathrm{d^{2i}}}{\mathrm{dT^{2i}}}\; \left[\sin\left(\frac T2-\frac\theta2\right)\right]^{-1}\right|_{T=0}\\
%&=\sum_{\substack{\text{partitions}\\\text{of $2i$} }}\frac{(2i)!}{k_1!k_2!\cdots k_n!} \left[\frac{(-1)^kk!}{\sin^{k+1}\left(-\frac\theta2\right)}\right]\left(\frac{\frac12\cos\left(-\frac\theta2\right)}{1!}\right)^{k_1}\left(\frac{\frac{-1}4\sin\left(-\frac\theta2\right)}{2!}\right)^{k_2}\cdots\left(\frac{\frac{(-1)^i}{2^{2i}}\sin\left(-\frac\theta2\right)}{(2i)!}\right)^{k_{2i}}\\
%&=\sum_{\substack{\text{partitions}\\\text{of $2i$} }}\frac{(2i)!}{k_1!k_2!\cdots k_n!} \left[\frac{(-1)^kk!}{\left(-\frac{\sqrt3}2\right)^{k+1}}\right]\,\frac{(-1)^{k_3+k_4+k_7+k_8+\cdots}}{2^{k_1+2k_2+\cdots+(2i)k_{2i}}}\left(\frac{1/2}{1!}\right)^{k_1}\left(\frac{\sqrt3/2}{2!}\right)^{k_2}\cdots\left(\frac{\sqrt3/2}{(2i)!}\right)^{k_{2i}}\\
%&=\sum_{\substack{\text{partitions}\\\text{of $2i$} }}\frac{(2i)!}{k_1!k_2!\cdots k_n!} \left[\frac{(-1)2^{k+1}k!}{(\sqrt3)^{k+1}}\right]\,\frac{(-1)^{k_3+k_4+k_7+k_8+\cdots}}{2^{2i}}\cdot\frac1{2^k}\cdot\frac{(\sqrt3)^{k_2+k_4+\cdots+k_{2i}}}{(1!)^{k_1}(2!)^{k_2}\cdots((2i)!)^{k_{2i}}}\\
%&=-\frac{1}{2^{2i-1}\sqrt3}\sum_{\substack{\text{partitions}\\\text{of $2i$} }}\frac{(2i)!\;k!}{k_1!k_2!\cdots k_n!} \left(\frac{(-1)^{k_3+k_4+k_7+k_8+\cdots}}{(\sqrt3)^{k_1+k_3+\cdots+k_{2i-1}}}\right)\cdot\frac{1}{(1!)^{k_1}(2!)^{k_2}\cdots((2i)!)^{k_{2i}}}\\
%\end{align*}
We can now proceed to calculating the coefficient $A_m(x)$ of $T^m$ in equation \eqref{eq:defnAm} using a convolution of the coefficients from equations \eqref{eq:sinh}, \eqref{eq:csc} and \eqref{eq:shiftedcsc}. Note that the series for the functions $\sinh(xT)$ and $\csc\left(\frac T2\right)$ contribute coefficients associated to odd powers of $T$ to the convolution. The convolution gives
\begin{align}\label{eq:am_expl}
A_{m}(x)=m!\cdot\sum_{j=0}^{\floor*{\frac m2}}\left(\sum_{k=0}^{\floor*{\frac{m+1}2}} (-1)^{k-1}2\left(1-\frac1{2^{2k-1}}\right)\frac{\mathscr{B}_{2k}}{(2k)!} \cdot\frac{\mu_{m-2j-2k}}{(m-2j-2k)!} \right)\frac1{(2j+1)!}\;x^{2j+1},
\end{align}
where $\mu_i$ is given by equation \eqref{eq:defnmun}. Similarly, using the series expansion
\[
\cosh(xT)=\sum_{j=0}^\infty\frac{x^{2j}}{(2j)!}T^{2j},
\]
in equation \eqref{eq:defnBm}, we obtain
\begin{equation}\label{eq:bm_expl}
B_m(x)=m!\cdot\sum_{j=0}^{\floor*{\frac{m+1}2}}\left(\sum_{k=0}^{\floor*{\frac{m+1}2}} (-1)^{k-1}2\left(1-\frac1{2^{2k-1}}\right)\frac{\mathscr{B}_{2k}}{(2k)!} \cdot\frac{\mu_{m-2j-2k+1}}{(m-2j-2k+1)!} \right)\frac1{(2j)!}\;x^{2j}.
\end{equation}
Considering the expressions in equations \eqref{eq:defnmun}, \eqref{eq:am_expl} and \eqref{eq:bm_expl}, we can observe that the coefficients of polynomials $A_m(x)$ and $B_m(x)$ involve rational numbers multiplied by integer powers of $\sqrt3$. This shows that all coefficients of these polynomials lie in $\Q\cup\sqrt3\cdot\Q$. We list the first examples of these polynomials:
\renewcommand{\arraystretch}{2.5}
\renewcommand{\tabcolsep}{0.5cm}
\[\begin{array}{|c|c|}
\hline 
A_m(x) & B_k(x)\\
\hline
\displaystyle A_0(x)=-\frac{4x}{\sqrt3} &\displaystyle B_0(x)=-\frac{2}3\\
\hline
\displaystyle A_1(x)=-\frac{2x}3 &\displaystyle B_1(x)=-\frac{2x^2}{\sqrt3}-\frac1{\sqrt3}\\
\hline
\displaystyle A_2(x)=-\frac{4x^3}{3\sqrt3}-\frac{2x}{\sqrt3} &\displaystyle B_2(x)=-\frac{2x^2}3-\frac49.\\
\hline
\displaystyle A_3(x)=-\frac{2x^3}3-\frac{4x}3 &\displaystyle B_3(x)=-\frac{x^4}{\sqrt3}-\sqrt3x^2-\frac{13}{10\sqrt3}.\\
\hline
\end{array}\]
 We can compare the expressions for $f_1(k)$ in equations \eqref{eq:f1even} and \eqref{eq:f1odd} with equation \eqref{eq:fk} from Section \ref{sec:f1}. This lets us relate the polynomials $R_m(x)$ and $S_m(x)$ to $A_k(x)$ and $B_k(x)$. Indeed, we have 
\[
R_{2n}(x)\approx-\frac{\sqrt3}4A_{2n}(x)\qquad\text{and}\qquad R_{2n+1}(x)\approx-\frac{\sqrt3}4B_{2n+1}(x),
\]
\[
S_{2n}(x)=\frac{3}4B_{2n}(x)\qquad\text{and}\qquad S_{2n+1}(x)=\frac{3}4A_{2n+1}(x).
\]
Here the ``$\approx$'' symbol denotes the fact that the two polynomials are equal except, possibly, their constant terms. The constant terms may differ due to the combinatorics involved in writing certain trigonometric functions. However, note that in equation \eqref{eq:fk}, we have the term $R_m\left(\frac{\log a}\theta\right)-R_m\left(\frac{\log b}\theta\right)$, and the constant terms cancel each other. Thus, this difference in constant terms does not affect the final computation and we get an explicit non-recursive description of the polynomials $R_m(x)$ and $S_m(x)$ via the definitions of $A_m(x)$ and $B_m(x)$. 

\subsubsection{The integral $f_2(k)$}
Now consider the integral
\begin{equation*}
I_2(\beta)=\int_0^\infty \frac{x^\beta}{(x^2-ax+a^2)(x^2-bx+b^2)}\dd x,
\end{equation*}
with $a,b>0$. Following the same process as for $I_1(\beta)$, we can compute
\[
I_2(\beta)=\frac{2\pi}{\sqrt3(a^3-b^3)}\left[\frac{(a^\beta-b^\beta)\sin(\theta(\beta-2))-(ba^{\beta-1}-ab^{\beta-1})\sin(\theta\beta)}{\sin\pi\beta}\right].
\]
Then,
\[
f_2(k)=\int_0^\infty\frac{t\log^kt}{(t^2-at+a^2)(t^2-bt+b^2)}\dd t=I_2^{(k)}(1).
\]
We write
\begin{equation}\label{eq:defnCnD}
\begin{aligned}
\sin(2T-\theta)\csc(3T)\sinh(xT)&=\sum_{m=-1}^\infty C_m(x)\frac{T^m}{m!},\\
\sin(2T-\theta)\csc(3T)\cosh(xT)&=\sum_{m=-1}^\infty D_m(x)\frac{T^m}{m!}.
\end{aligned}
\end{equation}
Then, for $\delta=\theta/2=\pi/3$, we get
\begin{multline}\label{eq:f2even}
f_2(2n)=\frac{\delta^{2n+1}(a+b)}{a^3-b^3}\cdot2\sqrt3\left[-C_{2n}\left(\frac{\log a}{\delta}\right)+C_{2n}\left(\frac{\log b}{\delta}\right)\right]\\
+\frac{\delta^{2n+1}}{\sqrt{3}(a^2+ab+b^2)}\cdot6\left[-D_{2n}\left(\frac{\log a}{\delta}\right)-D_{2n}\left(\frac{\log b}{\delta}\right)\right],
\end{multline}
and
\begin{multline}\label{eq:f2odd}
f_2(2n+1)=\frac{\delta^{2n+2}(a+b)}{a^3-b^3}\cdot2\sqrt3\left[-D_{2n+1}\left(\frac{\log a}{\delta}\right)+D_{2n+1}\left(\frac{\log b}{\delta}\right)\right]\\
+\frac{\delta^{2n+2}}{\sqrt{3}(a^2+ab+b^2)}\cdot6\left[C_{2n+1}\left(\frac{\log a}{\delta}\right)+C_{2n+1}\left(\frac{\log b}{\delta}\right)\right].
\end{multline}
The polynomials $C_m(x)$ and $D_m(x)$ may be compared to the polynomials $P_k(x)$ and $Q_k(x)$ from equation \eqref{eq:f2k} in Section \ref{sec:f2} (keeping in mind that the constant terms may not match, but will ultimately cancel out in the computation). Moreover, using convolution techniques, one can obtain explicit expressions for polynomials $C_m(x)$ and $D_m(x)$, as done for $A_m(x)$ and $B_m(x)$ in equations \eqref{eq:am_expl} and  \eqref{eq:bm_expl}. Observe once again that the coefficients of $C_m(x)$ and $D_m(x)$ will involve rational numbers multiplied by integer powers of $\sqrt3$, since the computations contain powers of $\sin(-2\pi/3)=-\sqrt3/2$ and $\cos(-2\pi/3)=-1/2$.

\subsubsection{The integral $g_1(k)$}
Next, we have the integral
\begin{equation*}
I_3(\beta)=\int_0^\infty \frac{x^\beta(x+a)}{(x^3-a^3)(x^2+bx+b^2)}\dd x,
\end{equation*}
and $a,b>0$. Note that the integrand is not defined at the point $x=a$ which lies in the interval of integration. When we write $I_3(\beta)$ we will always mean the Cauchy principal value of this integral. Using certain modified residue theorems (see \cite[Chapter 5]{keckic} or \cite[Corollary 3]{lowell}), we may compute $I_3(\beta)$ as
\begin{multline*}
I_3(\beta)=\frac{2\pi}{3(a^3-b^3)}\csc(\pi(\beta-1))\Bigg[a^{\beta-1}\left(a\cos\left(\frac{\theta}2(\beta+1)\right)+b\cos\left(\frac{\theta}2\beta\right)\right)\\+\sqrt3b^{\beta-1}\left(a\sin\left(\frac{\theta}2\beta\right)-b\sin\left(\frac{\theta}2(\beta+1)\right)\right)\Bigg]
-\frac{2\pi a^{\beta-1}}{3(a^2+ab+b^2)}\cot(\pi\beta).
\end{multline*}
We let
\begin{equation}\label{eq:defnEnF}
\begin{aligned}
\csc\left(\frac{3T}2\right)\cos\left(\frac T2+\theta\right)\sinh(xT)&=\sum_{m=-1}^\infty E_m(x)\frac{T^m}{m!},\\
\csc\left(\frac{3T}2\right)\cos\left(\frac T2+\theta\right)\cosh(xT)&=\sum_{m=-1}^\infty F_m(x)\frac{T^m}{m!},\\
\cot\left(\frac{3T}2\right)e^{xT}&=\sum_{m=-1}^\infty G_m(x)\frac{T^m}{m!},\\
\csc\left(\frac{3T}2\right)\sin\left(\frac T2+\frac\theta2\right)\sinh(xT)&=\sum_{m=-1}^\infty K_m(x)\frac{T^m}{m!},\\
\csc\left(\frac{3T}2\right)\sin\left(\frac T2+\frac\theta2\right)\cosh(xT)&=\sum_{m=-1}^\infty L_m(x)\frac{T^m}{m!}.
\end{aligned}
\end{equation}
Then,
\begin{multline}\label{eq:g1even}
g_1(2n)=\frac{\theta^{2n+1}}{3(a^2+ab+b^2)}\cdot3\left[E_{2n}\left(\frac{\log a}{\theta}\right)+\sqrt3K_{2n}\left(\frac{\log b}{\theta}\right)-G_{2n}\left(\frac{\log a}{\theta}\right)\right]\\
+\frac{\theta^{2n+1}(a+b)}{\sqrt3(a^3-b^3)}\cdot\sqrt3\left[F_{2n}\left(\frac{\log a}{\theta}\right)+\sqrt3L_{2n}\left(\frac{\log b}{\theta}\right) \right],
\end{multline}
and
\begin{multline}\label{eq:g1odd}
g_1(2n+1)=\frac{\theta^{2n+2}}{3(a^2+ab+b^2)}\cdot3\left[F_{2n+1}\left(\frac{\log a}{\theta}\right)+\sqrt3L_{2n+1}\left(\frac{\log b}{\theta}\right)-G_{2n+1}\left(\frac{\log a}{\theta}\right)\right]\\
+\frac{\theta^{2n+1}(a+b)}{\sqrt3(a^3-b^3)}\cdot\sqrt3\left[E_{2n+1}\left(\frac{\log a}{\theta}\right)+\sqrt3K_{2n+1}\left(\frac{\log b}{\theta}\right) \right].
\end{multline}
The polynomial $G_k(x)$ may be compared to the polynomial $Y_k(x)$ in equation \eqref{eq:gk} from Section \ref{sec:g1}. Using similar convolution techniques as done for $A_m(x)$ and $B_m(x)$, explicit expressions can be obtained for the polynomials $E_m(x),F_m(x),G_m(x),K_m(x)$ and $L_m(x)$. Again, the coefficients involved will consist of rationals multiplied by $\sqrt3$, since the trigonometric functions evaluate yield $\sqrt3$. 
\subsubsection{The integral $g_2(k)$}
Finally, we consider  the integral
\begin{equation*}
I_4(\beta)=\int_0^\infty \frac{x^\beta(x-a)}{(x^3+a^3)(x^2-bx+b^2)}\dd x,
\end{equation*}
with $a,b>0$. Again, residue theorems will give 
\begin{multline*}
I_4(\beta)=\frac{2\pi}{3(a^3-b^3)}\csc(\pi(\beta-1))\left[\sqrt3b^{\beta-1}(b\sin\theta(\beta+1)+a\sin\theta\beta)\right.\\
\left.-a^{\beta-1}(b\cos\theta\beta-a\cos\theta(\beta+1))-a^{\beta-1}(a-b)\right].
\end{multline*}
This time, we let
\begin{equation}\label{eq:defnUnV}
\begin{aligned}
\csc\left(3T\right)\cos\left(2T-2\delta\right)\sinh(xT)&=\sum_{m=-1}^\infty U_m(x)\frac{T^m}{m!},\\
\csc\left(3T\right)\cos\left(2T-2\delta\right)\cosh(xT)&=\sum_{m=-1}^\infty V_m(x)\frac{T^m}{m!},\\
\csc(3T)e^{xT}&=\sum_{m=-1}^\infty W_m(x)\frac{T^m}{m!},\\
\csc\left(3T\right)\sin\left(2T-2\delta\right)\sinh(xT)&=\sum_{m=-1}^\infty N_m(x)\frac{T^m}{m!},\\
\csc\left(3T\right)\sin\left(2T-2\delta\right)\cosh(xT)&=\sum_{m=-1}^\infty O_m(x)\frac{T^m}{m!}.\\
\end{aligned}
\end{equation}
Then,
\begin{multline}\label{eq:g2even}
g_2(2n)=\frac{\delta^{2n+1}}{3(a^2+ab+b^2)}\cdot6\left[-\sqrt3N_{2n}\left(\frac{\log b}{\delta}\right)+U_{2n}\left(\frac{\log a}{\delta}\right)-W_{2n}\left(\frac{\log a}{\delta}\right)\right]\\
+\frac{\delta^{2n+1}(a+b)}{\sqrt3(a^3-b^3)}\cdot2\sqrt3\left[\sqrt3O_{2n}\left(\frac{\log b}{\delta}\right)+V_{2n}\left(\frac{\log a}{\delta}\right)\right],
\end{multline}
and
\begin{multline}\label{eq:g2odd}
g_2(2n+1)=\frac{\delta^{2n+2}}{3(a^2+ab+b^2)}\cdot6\left[-\sqrt3O_{2n+1}\left(\frac{\log b}{\delta}\right)+V_{2n+1}\left(\frac{\log a}{\delta}\right)-W_{2n+1}\left(\frac{\log a}{\delta}\right)\right]\\
+\frac{\delta^{2n+2}(a+b)}{\sqrt3(a^3-b^3)}\cdot2\sqrt3\left[\sqrt3N_{2n+1}\left(\frac{\log b}{\delta}\right)+U_{2n+1}\left(\frac{\log a}{\delta}\right)\right].
\end{multline}
Once again, we can compare the polynomial $W_k(x)$ to $Z_k(x)$ in equation \eqref{eq:g2k} from Section \ref{sec:g2}. Similar to equations \eqref{eq:am_expl} and \eqref{eq:bm_expl}, explicit expressions for polynomials $U_m(x),V_m(x),W_m(x),N_m(x)$ and $O_m(x)$ can be obtained involving coefficients that are rational multiples of integer powers of $\sqrt3$.
\subsubsection{Summary}
We summarize the results of the above sub-sections in the proposition below:
\begin{prop}\label{prop:summary}
For a non-negative integer $m$ and real numbers $a,b>0$, let $f_1(m),f_2(m),g_1(m)$ and $g_2(m)$ denote the integrals
\[
f_1(m)=\int_0^\infty\frac{t\log^mt}{(t^2+at+a^2)(t^2+bt+b^2)}\dd t\quad\text{and}\quad f_2(m)=\int_0^\infty\frac{t\log^mt}{(t^2-at-a^2)(t^2+bt+b^2)}\dd t,
\]
\[
g_1(m)=\int_0^\infty \frac{t(t+a)\log^m t}{(t^3-a^3)(t^2+bt+b^2)}\dd t \qquad\text{and}\qquad g_2(m)=\int_0^\infty \frac{t(t-a)\log^m t}{(t^3+a^3)(t^2-bt+b^2)}\dd t.
\]
Then, there are two ways to evaluate the above integrals. The first is with the help of a family of polynomials defined recursively -- $R_m(x)$, $S_m(x)$, $P_m(x)$, $Q_m(x)$, $Y_m(x)$ and $Z_m(x)$. The second is using a family of polynomials that arise as coefficients of certain power series -- $A_m(x)$, $B_m(x)$, $C_m(x)$, $D_m(x)$, $E_m(x)$, $F_m(x)$, $G_m(x)$, $K_m(x)$, $L_m(x)$, $N_m(x)$ and $O_m(x)$. Explicit expressions for the latter family of polynomials may be obtained using convolution techniques, such as equations \eqref{eq:am_expl} and \eqref{eq:bm_expl} for $A_m(x)$ and $B_m(x)$ respectively. Thus, the polynomials in the first family can be expressed in terms of polynomials in the second family. Moreover, all these polynomials have coefficients that are rational multiples of integer powers of $\sqrt3$, i.e., they lie in $\Q\cup\sqrt3\cdot\Q$. The evaluations are tabulated below
\renewcommand{\arraystretch}{2.5}
\renewcommand{\tabcolsep}{0.5cm}
\[\begin{array}{|c|c|c|c|}
\hline 
\text{Integral} & \makecell{\text{Recursive}\\\text{expression}} & \makecell{\text{Alternate}\\ \text{expression}} & \text{Polynomials involved} \\
\hline
\displaystyle f_1(m) &\makecell{\eqref{eq:fk}, \text{using }\\R_m(x),S_m(x)} &\makecell{\eqref{eq:f1even},\eqref{eq:f1odd}} &\makecell{A_m(x)-\eqref{eq:am_expl},\\
														   B_m(x)-\eqref{eq:bm_expl},\\
														   \text{using } \eqref{eq:defnAm},\eqref{eq:defnBm}}\\
\hline														   
\displaystyle f_2(m) &\makecell{\eqref{eq:f2k},\text{ using}\\P_m(x),Q_m(x)} &\makecell{ \eqref{eq:f2even},\eqref{eq:f2odd}} &\makecell{C_m(x),D_m(x),\\
														   \text{using } \eqref{eq:defnCnD}}\\

\hline
\displaystyle g_1(m) & \makecell{\eqref{eq:gk},\text{using}\\R_m(x),S_m(x),Y_m(x)} &\makecell{\eqref{eq:g1even},\eqref{eq:g1odd}} &\makecell{E_m(x),F_m(x),G_m(x),\\
												K_m(x),L_m(x),\\
														   \text{using } \eqref{eq:defnEnF}}\\
\hline
\displaystyle g_2(m) & \makecell{\eqref{eq:g2k},\text{using}\\P_m(x),Q_m(x),Z_m(x)} &\makecell{ \eqref{eq:g2even},\eqref{eq:g2odd}} &\makecell{U_m(x),V_m(x),W_m(x),\\
												N_m(x),O_m(x),\\
														   \text{using } \eqref{eq:defnUnV}}\\
\hline
\end{array}\]
\end{prop}
\begin{rem}
Note that an explicit expression for polynomials in the second family results in an alternate method to compute the coefficients  $a_{r,s}\,,b_{r,s}\,,c_{r,s},d_{r,s}$ appearing in Theorem \ref{thm:main}. In Section \ref{sec:altf1}, we obtain such explicit expressions \eqref{eq:am_expl} and \eqref{eq:bm_expl} for polynomials $A_m(x)$ and $B_m(x)$ respectively, using standard convolution techniques. These techniques can be applied to the remaining polynomials as well. 
\end{rem}

\section{Computing the Mahler measure}
\label{sec:compute}
In this section, we proceed to evaluate the Mahler measure of the polynomial $Q_n(z_1,\dots,z_n,y)$. In particular, we take
\[
Q_n(z_1,\dots,z_n,y)=y+\left(\frac{z_1+\omega^{m-2}}{z_1+1}\right)\cdots\left(\frac{z_n+\omega^{m-2}}{z_n+1}\right),
\]

where 
\[
\omega=e^{2\pi i/m}=\cos\frac{2\pi}m+i\sin\frac{2\pi}m
\]
is an $m^\text{th}$-root of unity. We will later set $m=3$, in which case $\omega$ is a third root of unity and the polynomials $Q_n$ are of the form
\[
Q_n(z_1,\dots,z_n,y)=y+\left(\frac{z_1+\omega}{z_1+1}\right)\cdots\left(\frac{z_n+\omega}{z_n+1}\right).
\]
 
Now, consider the rational function 
\[
G_n(z_1,\dots,z_n,y)=y+\left(\frac{\ol\omega z_1+\omega}{z_1+1}\right)\cdots\left(\frac{\ol\omega z_n+\omega}{z_n+1}\right),
\]
where $\ol\omega$ is the complex conjugate of $\omega$, which in this case is $\omega^{m-1}$. Making the transformation $z_j\to1/z_j$ for each variable $z_j$ in $G_n$, followed by the transformation $y\to\omega^ny$, shows that the Mahler measure of $G_n$ is exactly the same as that of $Q_n$, since these transformations do not affect the Mahler measure. Thus, we have
\[
\m(Q_n)=\m(G_n).
\]
We shall evaluate the Mahler measure of $Q_n$ by evaluating the Mahler measure of $G_n$.\\

Note that using Jensen's formula (see \cite[Proposition 1.4]{BrunaultZudilin}), the Mahler measure of the polynomial $P_\gamma(y)=\gamma+y$ is given by $\log^+|\gamma|:=\log\max\{1,|\gamma|\}$. Thus, we may write the integral defining the Mahler measure of $G_n$ as
\[
\m(Q_n)=\m(G_n)=\frac{1}{(2\pi)^n}\int_{-\pi}^{\pi}\cdots \int_{-\pi}^{\pi}\m\left(P_{\left(\frac{\ol\omega z_1+\omega}{z_1+1}\right)\cdots\left(\frac{\ol\omega z_n+\omega}{z_n+1}\right)}\right)\dd\theta_1\cdots\dd \theta_n,
\]
where $z_j=e^{i\theta_j}.$ Let
\begin{equation}\label{eq:subst}
\tan\frac{\theta_j}2=\frac{x_j-\cos\frac{2\pi}m}{\sin\frac{2\pi}m},
\end{equation}
so that 
\[
x_j=\frac{\ol\omega z_j+\omega}{z_j+1}.
\]
Differentiating equation \eqref{eq:subst} gives
\begin{align*}
\left(1+\tan^2\frac{\theta_j}2\right)\frac{\dd\theta_j}2&=\frac{\dd x_j}{\sin\frac{2\pi}m},\\
\Longrightarrow\dd {\theta_j}&=\frac{2\sin\frac{2\pi}m}{x_j^2-2\cos\frac{2\pi}mx_j+1}\dd x_j\\
&\qquad=\frac{2\sin\frac{2\pi}m}{(x_j-\omega)(x_j-\ol\omega)}\dd x_j.
\end{align*}
Thus, the Mahler measure of $Q_n$ is given by
\begin{equation*}%\label{eq:nvar}
\m(Q_n)=\left(\frac{\sin\frac{2\pi}m}{\pi}\right)^n\int_{-\infty}^\infty\cdots \int_{-\infty}^\infty \m(P_{x_1\cdots x_n})\frac{\dd x_1}{(x_1-\omega)(x_1-\ol\omega)}\cdots \frac{\dd x_n}{(x_n-\omega)(x_n-\ol\omega)}.
\end{equation*}
In this discussion, we take $\omega$ to be the third root of unity, $e^{2\pi i/3}$. Taking $m=3$,  the above integral can be written as
\begin{equation*}
\m(Q_n)=\left(\frac{\sqrt3}{2\pi}\right)^n\int_{-\infty}^\infty\cdots \int_{-\infty}^\infty \m(P_{x_1\cdots x_n})\frac{\dd x_1}{(x_1^2+x_1+1)}\cdots \frac{\dd x_n}{(x_n^2+x_n+1)},
\end{equation*}
and we make the transformation ${y_1}=x_1,\;{y_2}=x_1x_2,\;\dots,{y_n}=x_1\cdots x_n$ to obtain
\begin{equation}\label{eq:mahlerint}
\m(Q_n)=\left(\frac{\sqrt3}{2\pi}\right)^n\int_{y_n}^*\cdots \int_{y_1}^* \m(P_{y_n})\frac{y_1\dd y_1}{(y_1^2+y_1+1)}\cdot \frac{y_2\dd y_2}{(y_2^2+y_2y_1+y_1^2)}\cdots \frac{\dd y_n}{(y_n^2+y_ny_{n-1}+y_{n-1}^2)}.
\end{equation}
The limits of integration for the variables $y_j$ have not been described above. We will discuss them in more detail in Section~\ref{sec:iteration}.  Throughout this article,  we will denote these limits using the symbol $\int_{y_j}^*$.\\
\subsection{The iteration step}
\label{sec:iteration}
Using the explicit formulae for the integrals considered in Section \ref{sec:generalintegrals}, we lay out an iterative procedure to calculate the Mahler measure $\m(Q_n)$. We rewrite equation \eqref{eq:mahlerint} as
\begin{equation}\label{eq:transfint}
\frac{\m(Q_n)}{\left({\sqrt3}/ {2\pi}\right)^n}=\int_{y_n}^*\cdots \int_{y_1}^* \m(P_{y_n})\frac{y_1\dd y_1}{(y_1^2+y_1+1)}\cdot \frac{y_2\dd y_2}{(y_2^2+y_2y_1+y_1^2)}\cdots \frac{\dd y_n}{(y_n^2+y_ny_{n-1}+y_{n-1}^2)},
\end{equation}
where $\m(P_x)=\log^+|x|$. We first wish to discuss the limits for each variable $y_k$ in integral \eqref{eq:transfint}. For $k\ge2$, since $y_k=x_ky_{k-1}$, we cannot specify the exact limits of $y_k$ without knowing the limits of the preceding variable $y_{k-1}$. However, we know $y_1=x_1$ varies from $-\infty$ to $\infty$, which will determine the limits of the  variable $y_2$, and in turn of $y_3,y_4$ and so on. Starting with $y_1$, we evaluate each integral with respect to the variable $y_k$, using the formulae obtained in Section \ref{sec:generalintegrals} at each step. We begin by writing the integral in \eqref{eq:transfint} as 
\begin{multline}\label{eq:basey}
\int_{y_n}^* \cdots \int_{y_2}^*\left(\int_{y_1=-\infty}^\infty \frac{y_1\dd y_1}{(y_1^2+y_1+1)(y_2^2+y_2y_1+y_1^2)}\right)\cdot \frac{y_2\dd y_2}{(y_2^2+y_2y_3+y_3^2)}\cdots\\ \cdots \frac{y_{n-1}\dd y_{n-1}}{(y_n^2+y_ny_{n-1}+y_{n-1}^2)}\cdot \m(P_{y_n})\dd y_n,
\end{multline}
and evaluating the first inner integral with respect to $y_1$. To do this, we split the integral $\int_{y_1=-\infty}^{\infty}$ into two integrals  $\int_{y_1=-\infty}^{0}$ and  $\int_{y_1=0}^{\infty}$, and use the results from Section \ref{sec:generalintegrals}. We will also need to know whether $y_2$ is positive or negative to be able to use the results from Section \ref{sec:generalintegrals}, that is, we need to determine the limits of $y_2$ in integral \eqref{eq:basey}. Recall that 
\[
y_{2}=x_{2}\cdot y_1,
\]
and both $x_{2}$ and $y_1$ vary from $-\infty$ to $\infty$. To determine the limits of $y_{2}$, we fix a value of $y_1$ in the given range and see how $y_{2}=x_{2}y_1$ varies as $x_{2}$ varies. Thus, the interval
\[
\int_{y_2}^*\int_{y_1=-\infty}^\infty
\]
in integral \eqref{eq:basey} can be split into the following four intervals:
\begin{enumerate}[label=(\Roman*)]
\item When $x_{2}\ge0$ and $y_{1}\ge0$: If $y_1=c\ge0$, then as $x_{2}$ varies from 0 to $\infty$, $y_2$ must also vary from $0$ to $\infty$. The limits in this case are
\begin{equation}\label{eq:inte1}
\int_{y_2=0}^\infty\; \left( \int_{y_1=0}^\infty\frac{y_1\dd y_1}{(y_1^2+y_1+1)(y_2^2+y_2y_1+y_1^2)}\right)\cdot \frac{y_2\dd y_2}{(y_2^2+y_2y_3+y_3^2)}.
\end{equation}

\item When $x_2\le0$ and $y_{1}\le0$: We fix $y_1=c\le0$, then as $x_2$ is negative and varies from $-\infty$ to $0$, $y_2=cx_2$ is positive and must vary from $\infty$ to $0$. Here, the limits are
\begin{align}
&\int_{y_2=\infty}^0\; \left( \int_{y_1=-\infty}^0\frac{y_1\dd y_1}{(y_1^2+y_1+1)(y_2^2+y_2y_1+y_1^2)}\right)\cdot \frac{y_2\dd y_2}{(y_2^2+y_2y_3+y_3^2)}\nonumber\\
\label{eq:inte2}=&\int_{y_2=0}^\infty\; \left( \int_{y_1=0}^\infty\frac{y_1\dd y_1}{(y_1^2-y_1+1)(y_2^2-y_2y_1+y_1^2)}\right)\cdot \frac{y_2\dd y_2}{(y_2^2+y_2y_3+y_3^2)},
\end{align}

\item When $x_2\le0$ and $y_{1}\ge0$. Here $y_1=c\ge0$, and $x_2$ varies from $-\infty$ to 0. This means that $y_2$ must  vary from $-\infty$ to 0 as well and the limits in this case are
\begin{equation}\label{eq:inte3}
\int_{y_2=-\infty}^0\; \left( \int_{y_1=0}^\infty\frac{y_1\dd y_1}{(y_1^2+y_1+1)(y_2^2+y_2y_1+y_1^2)}\right)\cdot \frac{y_2\dd y_2}{(y_2^2+y_2y_3+y_3^2)}.
\end{equation}
Note that compared to \eqref{eq:inte1}, in this case $y_2$ is negative. 
\item When $x_2\ge0$ and $y_{1}\le0$. Here, if $y_1=c\le0$ then as $x_2$ varies from 0 to $\infty$, we have $y_2$ to be negative varying from $0$ to $-\infty$. The limits are
\begin{align}
&\int_{y_2=0}^{-\infty}\;\left( \int_{y_1=-\infty}^0\frac{y_1\dd y_1}{(y_1^2+y_1+1)(y_2^2+y_2y_1+y_1^2)}\right)\cdot \frac{y_2\dd y_2}{(y_2^2+y_2y_3+y_3^2)}\nonumber\\
\label{eq:inte4}=&\int_{y_2=-\infty}^{0}\;\left( \int_{y_1=0}^\infty\frac{y_1\dd y_1}{(y_1^2-y_1+1)(y_2^2-y_2y_1+y_1^2)}\right)\cdot \frac{y_2\dd y_2}{(y_2^2+y_2y_3+y_3^2)},
\end{align}
again, after the transformation $(y_1\to-y_1)$ in the last line, and $y_2$ is negative, as opposed to \eqref{eq:inte2}.
\end{enumerate}
To sum up, the double integral
\begin{equation}\label{int:y1y2}
 \int_{y_2}^*\left(\int_{y_1=-\infty}^\infty \frac{y_1\dd y_1}{(y_1^2+y_1+1)(y_2^2+y_2y_1+y_1^2)}\right)\cdot \frac{y_2\dd y_2}{(y_2^2+y_2y_3+y_3^2)},
\end{equation}
appearing in integral \eqref{eq:basey} can be written as the sum of integrals \eqref{eq:inte1}, \eqref{eq:inte2}, \eqref{eq:inte3} and \eqref{eq:inte4}:
\begin{align}
\label{int:i1}\int_{y_2=0}^\infty\; \left( \int_{y_1=0}^\infty\frac{y_1\dd y_1}{(y_1^2+y_1+1)(y_2^2+y_2y_1+y_1^2)}\right)\cdot \frac{y_2\dd y_2}{(y_2^2+y_2y_3+y_3^2)}\\
\label{int:i2}+ \int_{y_2=0}^\infty\; \left( \int_{y_1=0}^\infty\frac{y_1\dd y_1}{(y_1^2-y_1+1)(y_2^2-y_2y_1+y_1^2)}\right)\cdot \frac{y_2\dd y_2}{(y_2^2+y_2y_3+y_3^2)}\\
\label{int:i3}+\int_{y_2=-\infty}^0 \left( \int_{y_1=0}^\infty\frac{y_1\dd y_1}{(y_1^2+y_1+1)(y_2^2+y_2y_1+y_1^2)}\right)\cdot \frac{y_2\dd y_2}{(y_2^2+y_2y_3+y_3^2)}
 \\
\label{int:i4}+ \int_{y_2=-\infty}^{0}\;\left( \int_{y_1=0}^\infty\frac{y_1\dd y_1}{(y_1^2-y_1+1)(y_2^2-y_2y_1+y_1^2)}\right)\cdot \frac{y_2\dd y_2}{(y_2^2+y_2y_3+y_3^2)}.
\end{align}
We remark that after choosing the parameters $a=1$ and $b=y_2$, one may use the expression for $f_1(0)$ from equation \eqref{eq:fk} in Section \ref{sec:f1} to evaluate the inner integral in line \eqref{int:i1} with respect to the variable $y_1$, and that for $f_2(0)$ from equation \eqref{eq:f2k} in Section \ref{sec:f2} to evaluate the inner integral in \eqref{int:i2}. However, the same cannot be used for the integrals in lines \eqref{int:i3} and \eqref{int:i4}, despite having identical integrands as lines \eqref{int:i1} and \eqref{int:i2}. This is so because $b=y_2$ is negative in this case, and equations  \eqref{eq:fk} and  \eqref{eq:f2k} are valid only when $a,b>0$. We can nonetheless combine integrals \eqref{int:i1} and \eqref{int:i2} together, since they have identical limits, and also combine integrals \eqref{int:i3} and \eqref{int:i4} together, to get
\begin{align}
\label{int:y2pos}&\int_{y_2=0}^\infty\; \int_{y_1=0}^\infty \left(\frac{y_1}{(y_1^2+y_1+1)(y_2^2+y_2y_1+y_1^2)}+ \frac{y_1}{(y_1^2-y_1+1)(y_2^2-y_2y_1+y_1^2)}\right)\dd y_1\cdot \frac{y_2\dd y_2}{(y_2^2+y_2y_3+y_3^2)}\\
&\label{int:y2neg}+\int_{y_2=-\infty}^0   \int_{y_1=0}^\infty\left(\frac{y_1}{(y_1^2+y_1+1)(y_2^2+y_2y_1+y_1^2)}+ \frac{y_1}{(y_1^2-y_1+1)(y_2^2-y_2y_1+y_1^2)}\right)\dd y_1\cdot \frac{y_2\dd y_2}{(y_2^2+y_2y_3+y_3^2)}.
\end{align}
The integrands in lines \eqref{int:y2pos} and \eqref{int:y2neg} are identical and the only difference between these integrals is that $y_2$ is positive in the first and negative in the second. However, note that the inner integrals are both of the form
\[
f_1(0)+f_2(0)=\int_0^\infty\left(\frac{t}{(t^2+at+a^2)(t^2+bt+b^2)}+\frac{t}{(t^2-at+a^2)(t^2-bt+b^2)}\right)\dd t,
\]
with $a=1$ and $b=y_2$ and $t=y_1$. This is exactly equation \eqref{eq:sumf} from the discussion in Section \ref{sec:signindep}, and we know that it evaluates to equation \eqref{eq:sumffinal} and that it is valid for $a,b\in\R^\times$. This means that the evaluation of the inner integrals with respect to the variable $y_1$ in lines \eqref{int:y2pos} and \eqref{int:y2neg} leads to the same expression given by equation \eqref{eq:sumffinal}, with $a=1,b=y_2$ and $k=0$. Thus, we may further combine integrals \eqref{int:y2pos} and \eqref{int:y2neg} together, with the variable $y_2$ varying from $-\infty$ to $\infty$, to get
\begin{equation}\label{int:intermediate}
\int_{y_2=-\infty}^\infty\Big(f_1(0)+f_2(0)\Big)\cdot \frac{y_2\dd y_2}{(y_2^2+y_2y_3+y_3^2)}.
\end{equation}
Recall that from the definition of the polynomials $R_k(x)$ and $S_k(x)$ given by equations \eqref{eq:rk} and \eqref{eq:sk}, we have $R_0(x)=x$ and $S_0(x)=-1/2$, and from the definitions of polynomials $P_k(x)$ and $Q_k(x)$ given by equations \eqref{eq:pk} and \eqref{eq:qk}, we obtain $P_0(x)=x$ and $Q_0(x)=2$. Plugging in these values in equation \eqref{eq:sumffinal} gives $f_1(0)+f_2(0)$ and replacing this in integral \eqref{int:intermediate}, we may finally write the double integral \eqref{int:y1y2} as
\begin{align}\label{eq:step2}
&\int_ {y_2=-\infty}^\infty\left(\frac{2(y_2+1)(\log{|y_2|})}{y_2^3-1}+\frac{2\pi}{3\sqrt3(y_2^2+y_2+1)}\right)\frac{y_2\dd y_2}{(y_2^2+y_2y_3+y_3^2)}\nonumber\\
=&2\int_ {y_2=-\infty}^\infty\frac{y_2(y_2+1)(\log{|y_2|})}{(y_2^3-1)(y_2^2+y_2y_3+y_3^2)}\dd y_2+\frac{2\pi}{3\sqrt3}\int_{y_2=-\infty}^\infty\frac{y_2}{(y_2^2+y_2+1)(y_2^2+y_2y_3+y_3^2)}\dd y_2.
\end{align}
In this way, we have reduced one variable (namely $y_1$) from the $n$-fold integral in \eqref{eq:basey}. The next step is to repeat the above procedure be evaluating the integrals in \eqref{eq:step2} while keeping track of the limits of the variable $y_3$. In other words, we evaluate 
\begin{align}
\label{eq:step2_1} &\int_{y_3}^*2\left(\int_{y_2=-\infty}^\infty \frac{y_2(y_2+1)(\log{|y_2|})}{(y_2^3-1)(y_2^2+y_2y_3+y_3^2)}\dd y_2\right)\cdot \frac{y_3\dd y_3}{(y_3^2+y_4y_3+y_4^2)}\\
\label{eq:step2_2} &\qquad\qquad\qquad+ \int_{y_3}^*\frac{2\pi}{3\sqrt3}\left(\int_{y_2=-\infty}^\infty\frac{y_2}{(y_2^2+y_2+1)(y_2^2+y_2y_3+y_3^2)}\dd y_2\right)\cdot \frac{y_3\dd y_3}{(y_3^2+y_4y_3+y_4^2)}.
\end{align}
Recall that $y_3=x_3y_2$, and once again, we split the integrals in lines \eqref{eq:step2_1} and \eqref{eq:step2_2} into four integrals each based on the signs of $y_2$ and $y_3$, and perform the integrations with respect to the variable $y_2$. Note that each of these integrals has the form of the integrals discussed in Section \ref{sec:generalintegrals}. For example, the integrals that come from line \eqref{eq:step2_1} correspond to $g_1(1)$ and $g_2(1)$, as the case may be, from Sections \ref{sec:g1} and \ref{sec:g2} respectively. Similar to the previous case, the integrals in line \eqref{eq:step2_2} correspond to  $f_1(0)$ and $f_2(0)$, from Sections \ref{sec:f1} and \ref{sec:f2} respectively. Using the same arguments as done in the previous step, we may combine  appropriate integrals with varying limits of $y_3$ into one single integral where $y_3$ varies from $-\infty$ to $\infty$, from where the same process can be continued. Thus, at the $j^\mathrm{th}$-step, we have an inner-most integral with respect to the variable $y_j$, with limits from $-\infty$ to $\infty$, and sums of integrands of the form 
\begin{equation}\label{form:1}
\frac{y_{j}\log^m|y_j|}{(y_{j}^2+y_{j}+1)(y_{j}^2+y_{j+1}y_{j}+y_{j+1}^2)},
\end{equation}
or
\begin{equation}\label{form:2}
\frac{y_{j}(y_{j}+1)\log^m |y_{j}|}{(y_{j}^3-1)(y_{j}^2+y_{j}y_{j+1}+y_{j+1}^2)},
\end{equation}
for various values of non-negative integers $m\ge0$. Taking $y_{j+1}=x_{j+1}y_j$, we split each integral according to the signs of $y_{j}$ and  $y_{j+1}$. This leads to four intervals for each integral. Now, in the two intervals where $y_{j+1}$ is positive, that is, the region where it takes values between $0$ and $\infty$, the sum of the corresponding integrals will evaluate to equation \eqref{eq:sumffinal}
\[
f_1(m)+f_2(m),
\]
if the integrand is of the form \eqref{form:1}, or to equation \eqref{eq:sumgfinal}
\[
g_1(m)+g_2(m),
\]
if it is of the form \eqref{form:2}, with the parameters $a=1,b=y_{j+1}$ and $k=m$. In light of the discussions in Section \ref{sec:signindep}, the sum of the remaining two intervals, where $y_{j+1}$ is negative, will also have the same form. Thus, these four integrals can be combined together to a single integral with respect to the variable $y_{j+1}$ of the form
\[
\int_{y_{j+1}=-\infty}^\infty\Big(f_1(m)+f_2(m)\Big)\cdot \frac{y_{j+1}\dd y_{j+1}}{(y_{j+1}^2+y_{j+1}y_{j+2}+y_{j+2}^2)},
\]
or of the form
\[
\int_{y_{j+1}=-\infty}^\infty\Big(g_1(m)+g_2(m)\Big)\cdot \frac{y_{j+1}\dd y_{j+1}}{(y_{j+1}^2+y_{j+1}y_{j+2}+y_{j+2}^2)},
\]
as the case may be. As can be seen from the results for $f_1(m),f_2(m),g_1(m)$ and $g_2(m)$ from Section \ref{sec:generalintegrals}, this leads to a sum of integrals with respect to the variable $y_{j+1}$, all of which again have the form \eqref{form:1} or \eqref{form:2}. We may thus apply the above procedure iteratively up to the last variable $y_n$ and the limits of the integral will simply be $\int_{y_n=-\infty}^\infty$. More precisely, if we denote by $F(k)$ the following integral over $k$ variables,
\begin{equation}\label{eq:genfk}
F(k):=\int_{y_k}^*\cdots \int_{y_1}^* \m(P_{y_k})\frac{y_1\dd y_1}{(y_1^2+y_1+1)}\cdot \frac{y_2\dd y_2}{(y_2^2+y_2y_1+y_1^2)}\cdots \frac{\dd y_k}{(y_k^2+y_ky_{k-1}+y_{k-1}^2)},
\end{equation}
then, after the final iteration, we can write for $n\ge0$ and $k=2n+2$, 
\begin{multline}\label{eq:evencase}
F(2n+2)=\sum_{h=1}^{n+1} a_{n+1,h-1}\;\left(\frac\pi 3\right)^{2n+2-2h}\int_{-\infty}^\infty\m(P_{y_{k}})\frac{y_{k}+1}{y_{k}^3-1}\log^{2h-1}|y_{k}|\dd {y_k}\\
+\sum_{h=0}^{n} b_{n+1,h}\;\left(\frac\pi 3\right)^{2n+1-2h}\int_{-\infty}^\infty\m(P_{y_{k}})\frac{\log^{2h}|y_{k}|}{y_{k}^2+y_{k}+1} \dd {y_k},
\end{multline}
and for $n\ge0$ and $k=2n+1$, we write
\begin{multline}\label{eq:oddcase}
F(2n+1)=\sum_{h=1}^n c_{n,h-1}\;\left(\frac\pi 3\right)^{2n-2h+1}\int_{-\infty}^\infty\m(P_{y_{k}})\frac{y_{k}+1}{y_{k}^3-1}\log^{2h-1}|y_{k}|\dd {y_k}\\
+\sum_{h=0}^{n} d_{n,h}\;\left(\frac\pi 3\right)^{2n-2h}\int_{-\infty}^\infty\m(P_{y_{k}})\frac{\log^{2h}|y_{k}|}{y_{k}^2+y_{k}+1}\dd {y_k},
\end{multline}
where $a_{r,s}\,,b_{r,s}\,,c_{r,s},d_{r,s}$ are real numbers which will be defined recursively in Section ~\ref{sec:reccoeffs}. \\

\subsection{Expressing the integrals in terms of $L$-functions}\label{sec:polylogsaslvals}
Here we write the Mahler measure integrals obtained in equations \eqref{eq:evencase} and \eqref{eq:oddcase} in terms of polylogarithm values which can in turn be written as special values of the Riemann zeta function and the Dirichlet $L$-function $L(\chi_{-3},s)$. \\

Note that $\m(P_{y_k})=\log^+|y_k|$, and we can write
\[
\int_{-\infty}^\infty\m(P_{y_{k}})\frac{y_{k}+1}{y_{k}^3-1}\log^{2h-1}|y_{k}|\dd {y_k}=\int_0^1\log^{2h}t\,\frac{1+t}{1-t^3}\dd t +\int_0^1\log^{2h}t\,\frac{1-t}{1+t^3}\dd t,
\]
and
\[
\int_{-\infty}^\infty\m(P_{y_{k}})\frac{\log^{2h}|y_{k}|}{y_{k}^2+y_{k}+1}\dd {y_k}=-\int_0^1\frac{\log^{2h+1}t}{t^2+t+1}\dd t-\int_0^1\frac{\log^{2h+1}t}{t^2-t+1}\dd t.
\]
Using the following expansions
\begin{align*}
\frac{1+t}{1-t^3}&=\frac13\left(\frac{2}{1-t}+\frac1{t-\omega}+\frac1{t-\omega^2}\right),\\
\frac{1-t}{1+t^3}&=\frac13\left(\frac{2}{t+1}-\frac1{t+\omega}-\frac1{t+\omega^2}\right),\\
\frac{1}{t^2+t+1}&=\frac{1}{i\sqrt3}\left(\frac1{t-\omega}-\frac1{t-\omega^2}\right),\\
\frac{1}{t^2-t+1}&=\frac{1}{i\sqrt3}\left(\frac1{t+\omega^2}-\frac1{t+\omega}\right),
\end{align*}
where $\omega=e^{2\pi i/3}$ is a third root of unity, we may write the Mahler measure in terms of hyperlogarithms as follows.
We can write
\begin{multline*}
\int_{-\infty}^\infty\m(P_{y_{k}})\frac{y_{k}+1}{y_{k}^3-1}\log^{2h-1}|y_{k}|\dd y_k=\frac13\int_0^1\log^{2h}t\left(-\frac{2}{t-1}+\frac1{t-\omega}+\frac1{t-\omega^2}\right)\dd t \\
+\frac13\int_0^1\log^{2h}t \left(\frac{2}{t+1}-\frac1{t+\omega}-\frac1{t+\omega^2}\right)\dd t.
\end{multline*}
Using the definition of hyperlogarithms as given in Section \ref{sec:hyperlogs} and identity \eqref{eq:hypertopoly}, for $a\in\C^*$, we have 
\begin{align*}
\int_0^1 \log^{2h}t\,\frac{1}{t-a}\dd t&= (-1)^{2h}(2h)!\int_0^1 \frac{\dd t}{t-a}\circ \overbrace{\frac{\dd t}{t}\circ\cdots\circ \frac{\dd t}{t}}^\text{$2h$ times}\\
&=(2h)!\cdot\mathrm{I}_{2h+1}(a,1)=-(2h)!\,\mathrm{Li}_{2h+1}(1/a).
\end{align*}
This gives, for $h\ge1$,
\begin{multline*}
\int_{-\infty}^\infty\m(P_{y_{k}})\frac{y_{k}+1}{y_{k}^3-1}\log^{2h-1}|y_{k}|\dd {y_k}=-\frac{(2h)!}{3}\Big(-2\Li_{2h+1}(1)+\Li_{2h+1}(\omega^2)+\Li_{2h+1}(\omega)\\
+2\Li_{2h+1}(-1)-\Li_{2h+1}(-\omega^2)-\Li_{2h+1}(-\omega) \Big).
\end{multline*}
We have the following identities:
\begin{align*}
\Li_{2h+1}(1)&=\zeta(2h+1),\\
\Li_{2h+1}(-1)&=-\zeta(2h+1)\left(1-\frac1{4^{h}}\right),\\
\Li_{2h+1}(\omega)+\Li_{2h+1}(\omega^2)&=-\zeta(2h+1)\left(1-\frac1{9^h}\right),\\
\Li_{2h+1}(-\omega)+\Li_{2h+1}(-\omega^2)&=\zeta(2h+1)\left(1-\frac1{9^h}\right)\left(1-\frac1{4^{h}}\right).\\
\end{align*}
We will show the third identity, since the rest can be proved in a similar manner. Since $2h+1\ge3$, the sum \[\Li_{2h+1}(z)=\sum_{n=1}^\infty\frac{z^n}{n^{2h+1}}\] converges absolutely for $|z|\le1$ and we are free to change the order of the terms. We write
\begin{align*}
\Li_{2h+1}(\omega)+\Li_{2h+1}(\omega^2)&=\sum_{j=1}^\infty\frac{\omega^j+\omega^{2j}}{j^{2h+1}}\\
 &=\sum_{j=1}^\infty\frac{2\cos(2\pi j/3)}{j^{2h+1}}\\
 &=-\frac1{1^{2h+1}}-\frac{1}{2^{2h+1}}+\frac2{3^{2h+1}}-\frac1{4^{2h+1}}-\cdots\\
 &=-\left(\frac1{1^{2h+1}}+\frac{1}{2^{2h+1}}+\frac1{3^{2h+1}}+\cdots\right)+3\left(\frac1{3^{2h+1}}+\frac{1}{6^{2h+1}}+\frac1{9^{2h+1}}+\cdots\right)\\
 &=-\zeta(2h+1)+\frac1{3^{2h}}\cdot\zeta(2h+1)\\
 &=-\zeta(2h+1)\left(1-\frac1{9^h}\right),
\end{align*}
as desired.\\

Using these identities, we obtain
\begin{equation*}
\int_{-\infty}^\infty\m(P_{y_{k}})\frac{y_{k}+1}{y_{k}^3-1}\log^{2h-1}|y_{k}|\dd {y_k}=2(2h)!\zeta(2h+1)\left(1-\frac1{3^{2h+1}}\right)\left(1-\frac1{2^{2h+1}}\right).
\end{equation*}
Similarly, we have for $h\ge0$,
\begin{align*}
\int_{-\infty}^\infty\m(P_{y_{k}})\frac{\log^{2h}|y_{k}|}{y_{k}^2+y_k+1}\dd {y_k}&=-\frac{1}{i\sqrt3}\int_0^1\log^{2h+1}t\left( \frac1{t-\omega}-\frac1{t-\omega^2}+\frac1{t+\omega^2}-\frac1{t+\omega}\right)\dd t\\
&=-\frac{(2h+1)!}{i\sqrt3}\Big(\Li_{2h+2}(\omega^2)-\Li_{2h+2}(\omega)+\Li_{2h+2}(-\omega)-\Li_{2h+2}(-\omega^2)\Big), 
\end{align*}
where we use
\begin{align*}
\int_0^1 \log^{2h+1}t\,\frac{1}{t-a}\dd t&= (-1)^{2h+1}(2h+1)!\int_0^1 \frac{\dd t}{t-a}\circ \overbrace{\frac{\dd t}{t}\circ\cdots\circ\frac{\dd t}t}^\text{$2h+1$ times}\\
&=-(2h+1)!\,\mathrm{I}_{2h+2}(a,1)=(2h+1)!\,\mathrm{Li}_{2h+2}(1/a).
\end{align*}
In this case, we will use the following identities
\begin{align*}
\Li_{2h+2}(\omega^2)-\Li_{2h+2}(\omega)&=-\sqrt3iL(\chi_{-3},2h+2),\\
\Li_{2h+2}(-\omega)-\Li_{2h+2}(-\omega^2)&=-\sqrt3i\left(1+\frac1{2^{2h+1}}\right)L(\chi_{-3},2h+2).\\
\end{align*}
This means
\begin{equation}\label{eq:basecased}
\int_{-\infty}^\infty\m(P_{y_{k}})\frac{\log^{2h}|y_{k}|}{y_{k}^2+y_k+1}\dd {y_k}=2(2h+1)!L(\chi_{-3},2h+2)\left(1+\frac1{2^{2h+2}}\right).
\end{equation}
Therefore, plugging the above relations into equations \eqref{eq:evencase} and \eqref{eq:oddcase} we obtain for the even case with $n\ge1$
\begin{multline}\label{eq:neweven}
F(2n)=\sum_{h=1}^n a_{n,h-1}\;\left(\frac\pi 3\right)^{2n-2h}2(2h)!\zeta(2h+1)\left(1-\frac1{3^{2h+1}}\right)\left(1-\frac1{2^{2h+1}}\right)\\
+\sum_{h=0}^{n-1} b_{n,h}\;\left(\frac\pi 3\right)^{2n-2h-1}2(2h+1)!L(\chi_{-3},2h+2)\left(1+\frac1{2^{2h+2}}\right),
\end{multline}
and for the odd case with $n\ge0$
\begin{multline}\label{eq:newodd}
F(2n+1)=\sum_{h=1}^n c_{n,h-1}\;\left(\frac\pi 3\right)^{2n-2h+1}2(2h)!\zeta(2h+1)\left(1-\frac1{3^{2h+1}}\right)\left(1-\frac1{2^{2h+1}}\right)\\
+\sum_{h=0}^{n} d_{n,h}\;\left(\frac\pi 3\right)^{2n-2h}2(2h+1)!L(\chi_{-3},2h+2)\left(1+\frac1{2^{2h+2}}\right).
\end{multline}

\subsection{The recursive coefficients}
\label{sec:reccoeffs}
What remains is to find a relation for the coefficients $a_{r,s}\,,b_{r,s}\,,c_{r,s},d_{r,s}$ appearing in the formulae above.  The idea is the following: we begin with the $(2n+2)$-fold multiple integral $F(2n+2)$ in \eqref{eq:genfk}.  As seen in equation \eqref{eq:evencase}, after the final iteration one obtains sums of single integrals with coefficients involving $a_{r,s}\,,b_{r,s}$. We may go back one step to the penultimate iteration,  where we will have a sum of double integrals, this time with coefficients involving the $c_{r,s},d_{r,s}$. Using the formulae for integrals obtained in the previous section, we can then compare coefficients and write $a_{r,s}$ and $b_{r,s}$ in terms of $c_{r,s}$ and $d_{r,s}$.  In turn, we do the same with $F(2n+1)$ to get expressions for $c_{r,s}$ and $d_{r,s}$ in terms of $a_{r,s}$ and $b_{r,s}$. 

We need to define some notation for the polynomials $R_k,S_k,P_k,Q_k,Y_k$ and $Z_k$.  Note that we  either have  only odd powers of $x$ or only even powers of $x$ occuring in each of these polynomials. Moreover, the degree of $R_k,P_k,Y_k$ and $Z_k$ is $k+1$ while it is $k$ for $S_k$ and $Q_k$.  For convenience, we define some notation in the following table: we write the polynomials which have only odd powers of $x$ in the left column and those with even powers of $x$ in the right as follows

\renewcommand{\arraystretch}{2.5}
\renewcommand{\tabcolsep}{0.5cm}
\[\begin{array}{|c|a|}
\hline 
\text{Only odd powers} & \text{Only even powers}\\
\hline
\displaystyle R_{2m}(x)=\sum_{j=1}^{m+1} r_{2m,j}\;x^{2j-1} & \displaystyle R_{2m-1}(x)=\sum_{j=0}^{m} r_{2m-1,j}\;x^{2j}\\
\hline
\displaystyle P_{2m}(x)=\sum_{j=1}^{m+1} p_{2m,j}\;x^{2j-1} & \displaystyle P_{2m-1}(x)=\sum_{j=0}^{m} p_{2m-1,j}\;x^{2j}\\
\hline
\displaystyle Y_{2m}(x)=\sum_{j=1}^{m+1} y_{2m,j}\;x^{2j-1} &\displaystyle Y_{2m-1}(x)=\sum_{j=0}^{m} y_{2m-1,j}\;x^{2j}\\
\hline
\displaystyle Z_{2m}(x)=\sum_{j=1}^{m+1} z_{2m,j}\;x^{2j-1}&\displaystyle Z_{2m-1}(x)=\sum_{j=0}^{m} z_{2m-1,j}\;x^{2j}\\
\hline
\displaystyle S_{2m-1}(x)=\sum_{j=1}^{m} s_{2m-1,j}\;x^{2j-1}&\displaystyle S_{2m}(x)=\sum_{j=0}^{m} s_{2m,j}\;x^{2j}\\
\hline 
\displaystyle Q_{2m-1}(x)=\sum_{j=1}^{m} q_{2m-1,j}\;x^{2j-1}&\displaystyle Q_{2m}(x)=\sum_{j=0}^{m} q_{2m,j}\;x^{2j}\\
\hline
\end{array}\]
and where all coefficients are rational numbers.\\

Now, for $n\ge0$ and $k=2n+2$, we have from \eqref{eq:evencase}
\begin{align}\label{eq:finit}
F(2n+2)&=\int_{y_k}\cdots \int_{y_1}\; \m(P_{y_{k}})\frac{y_1\dd y_1}{(y_1^2+y_1+1)}\cdots \frac{\dd y_{k}}{(y_{k}^2+y_{k}y_{k-1}+y_{k-1}^2)}\nonumber\\
\begin{split}
&=\sum_{h=1}^{n+1} a_{n+1,h-1}\;\left(\frac\pi 3\right)^{2n+2-2h}\int_{-\infty}^\infty\m(P_{y_{k}})\frac{y_{k}+1}{y_{k}^3-1}\log^{2h-1}|y_{k}|\dd y_k\\
&\qquad+\sum_{h=0}^{n} b_{n+1,h}\;\left(\frac\pi 3\right)^{2n-2h+1}\int_{-\infty}^\infty\m(P_{y_{k}})\frac{\log^{2h}|y_{k}|}{y_{k}^2+y_{k}+1}\dd y_k.
\end{split}
\end{align}
This is obtained by iteratively using the formulae from Section \ref{sec:generalintegrals} along with the procedure explained in \ref{sec:iteration} upto the last variable $y_{2n}=y_k$.  At the penultimate step, the integral can be written as
\begin{align}\label{eq:mainterm}
F(2n+2)&=\int_{y_k}^*\int_{y_{k-1}}^*\cdots \int_{y_1}^*\; \m(P_{y_{k}})\frac{y_1\dd y_1}{(y_1^2+y_1+1)}\cdots \frac{\dd y_{k}}{(y_{k}^2+y_{k}y_{k-1}+y_{k-1}^2)}\nonumber\\
&=\int_{y_k}^*\m(P_{y_{k}})\Bigg(\int_{y_{k-1}}^*\cdots \int_{y_1}^*\; \frac{y_1\dd y_1}{(y_1^2+y_1+1)}\cdots \frac{y_{k-1}\dd y_{k-1}}{(y_{k}^2+y_{k}y_{k-1}+y_{k-1}^2)}\Bigg)\dd y_{k}\nonumber\\
\begin{split}
&=\int_{y_k}^*\m(P_{y_{k}})\Bigg(\sum_{l=1}^n c_{n,l-1}\;\left(\frac\pi 3\right)^{2n-2l+1}\int_{y_{k-1}}^*\frac{y_{k-1}(y_{k-1}+1)\log^{2l-1}|y_{k-1}|}{(y_{k-1}^3-1)(y_{k}^2+y_{k}y_{k-1}+y_{k-1}^2)}\dd y_{k-1}\\
&\qquad\qquad\qquad+\sum_{i=0}^{n} d_{n,l}\;\left(\frac\pi 3\right)^{2n-2l}\int_{y_{k-1}}^*\frac{y_{k-1}\log^{2l}|y_{k-1}|}{(y_{k-1}^2+y_{k-1}+1)(y_{k}^2+y_{k}y_{k-1}+y_{k-1}^2)} \dd y_{k-1}\Bigg)\dd y_k.
\end{split}
\end{align}
Following the discussions in Section \ref{sec:iteration}, the limits of the inner-most integral are always from $-\infty$ to $\infty$. We have
\[
\int_{y_{k-1}}^*\frac{y_{k-1}(y_{k-1}+1)\log^{2l-1}|y_{k-1}|}{(y_{k-1}^3-1)(y_{k}^2+y_{k}y_{k-1}+y_{k-1}^2)}\dd y_{k-1}=\int_{y_{k-1}=-\infty}^\infty\frac{y_{k-1}(y_{k-1}+1)\log^{2l-1}|y_{k-1}|}{(y_{k-1}^3-1)(y_{k}^2+y_{k}y_{k-1}+y_{k-1}^2)}\dd y_{k-1},
\]
which will be given by the sum of \eqref{eq:gk} and \eqref{eq:g2k} with $a=1$ and $b=y_k$, or in other words, it will be given by equation \eqref{eq:sumgfinal} 
$$g_1(2l-1)+g_2(2l-1),$$ from Section \ref{sec:signindep}. We get

\begin{multline*}
\left. \begin{array}{rr}
\frac{\theta^{2l}}{3(y_k^2+y_k+1)}\left[-R_{2l-1}(0)+3R_{2l-1}\left(\frac{\log |y_k|}{\theta}\right)+Y_{2l-1}(0)\right]\\
+\frac{\theta^{2l}(y_k+1)}{\sqrt3(y_k^3-1)}\left[S_{2h-1}\left(\frac{\log |y_{k}|}{\theta}\right)-S_{2h-1}(0)\right]
\end{array}\qquad\right\}g_1(2l-1)\\
\left.\begin{array}{rr}
+\frac{\delta^{2l}}{3(y_k^2+y_k+1)}\left[-P_{2l-1}(0)+3P_{2l-1}\left(\frac{\log |y_{k}|}{\delta}\right)+Z_{2l-1}(0)\right]\\
+\frac{\delta^{2l}(y_{k}+1)}{\sqrt3(y_{k}^3-1)}\left[Q_{2l-1}\left(\frac{\log |y_{k}|}{\delta}\right)-Q_{2l-1}(0)\right],
\end{array}\qquad\right\} g_2(2l-1)
\end{multline*}
where the first two lines correspond to $g_1(2l-1)$ and the next two lines to $g_2(2l-1)$.  Using the notation for the polynomials $R_k,S_k,P_k,Q_k,Y_k$ and $Z_k$ defined above (and also noting that $S_{2h-1}(x)$ and $Q_{2h-1}(x)$ have no constant term),  we can write this as
\begin{multline*}
\frac{\theta^{2l}}{3(y_k^2+y_k+1)}\left[-r_{2l-1,0}+3\sum_{j=0}^lr_{2l-1,j}\left(\frac{\log |y_k|}{\theta}\right)^{2j}+y_{2l-1,0}\right]\\
+\frac{\theta^{2l}(y_k+1)}{\sqrt3(y_k^3-1)}\left[\sum_{j=1}^ls_{2l-1,j}\left(\frac{\log |y_k|}{\theta}\right)^{2j-1}\right]\\
+\frac{\delta^{2l}}{3(y_k^2+y_k+1)}\left[-p_{2l-1,0}+3\sum_{j=0}^lp_{2l-1,j}\left(\frac{\log |y_k|}{\delta}\right)^{2j}+z_{2l-1,0}\right]\\
+\frac{\delta^{2l}(y_{k}+1)}{\sqrt3(y_{k}^3-1)}\left[\sum_{j=1}^lq_{2l-1,j}\left(\frac{\log |y_k|}{\delta}\right)^{2j-1}\right].
\end{multline*}
Collecting the terms with $\frac{1}{(y_k^2+y_k+1)}$ together and those with $\frac{(y_k+1)}{(y_k^3-1)}$ and replacing $\theta=2\pi/3$ by $2\delta$ we have
\begin{multline}\label{eq:term1}
\frac{\delta^{2l}}{3(y_k^2+y_k+1)}\left[2^{2l}(-r_{2l-1,0}+y_{2l-1,0})-p_{2l-1,0}+z_{2l-1,0}+3\sum_{j=0}^l(2^{2l-2j}r_{2l-1,j}+p_{2l-1,j})\left(\frac{\log |y_k|}{\delta}\right)^{2j}\right]\\
+\frac{\delta^{2l}(y_{k}+1)}{\sqrt3(y_{k}^3-1)}\left[\sum_{j=1}^l(2^{2l-2j+1}s_{2l-1,j}+q_{2l-1,j})\left(\frac{\log |y_k|}{\delta}\right)^{2j-1}\right].
\end{multline}
Next, the second inner integral in \eqref{eq:mainterm},
\[
\int_{y_{k-1}}^*\frac{y_{k-1}\log^{2l}|y_{k-1}|}{(y_{k-1}^2+y_{k-1}+1)(y_{k}^2+y_{k}y_{k-1}+y_{k-1}^2)} \dd y_{k-1}=\int_{y_{k-1}=-\infty}^\infty\frac{y_{k-1}\log^{2l}|y_{k-1}|}{(y_{k-1}^2+y_{k-1}+1)(y_{k}^2+y_{k}y_{k-1}+y_{k-1}^2)} \dd y_{k-1}
\]
is given this time by equation \eqref{eq:sumffinal}$$f_1(2l)+f_2(2l),$$again with $a=1$ and $b=y_k$, which leads to
\begin{multline*}
\qquad\left. \begin{array}{rr}
\frac{\theta^{2l+1}(y_k+1)}{(y_k^3-1)}\left[R_{2l}\left(\frac{\log |y_{k}|}{\theta}\right)-R_{2l}(0)\right]\\
+\frac{\theta^{2l+1}}{\sqrt3(y_k^2+y_k+1)}\left[S_{2l}\left(\frac{\log |y_k|}{\theta}\right)+S_{2l}(0)\right]
\end{array}\qquad\qquad\right\}f_1(2l)\\
\left.\begin{array}{lr}
+\frac{\delta^{2l+1}(y_k+1)}{(y_k^3-1)}\left[P_{2l}\left(\frac{\log |y_{k}|}{\delta}\right)-P_{2l}(0)\right]\\
+\frac{\delta^{2l+1}}{\sqrt3(y_k^2+y_k+1)}\left[Q_{lh}\left(\frac{\log |y_k|}{\delta}\right)+Q_{2l}(0)\right].
\end{array}\qquad\qquad\right\} f_2(2l)
\end{multline*}
Rewriting the polynomials and collecting terms together we obtain
\begin{multline}\label{eq:term2}
\frac{\delta^{2l+1}(y_k+1)}{(y_k^3-1)}\left[\sum_{j=1}^{l+1}(2^{2l-2j+2}r_{2l,j}+p_{2l,j})\left(\frac{\log |y_{k}|}{\delta}\right)^{2j-1}\right]\\
+\frac{\delta^{2l+1}}{\sqrt3(y_k^2+y_k+1)}\left[2^{2l+1}s_{2l,0}+q_{2l,0}+\sum_{j=0}^h(2^{2l-2j+1}s_{2l,j}+q_{2l,j})\left(\frac{\log |y_k|}{\delta}\right)^{2j}\right].
\end{multline}
Finally, we plug the expressions \eqref{eq:term1} and \eqref{eq:term2} back into equation \eqref{eq:mainterm} and compare coefficients with the initial expression for $F(2n+2)$ given by \eqref{eq:finit}. Comparing coefficients of $\frac{(y_k+1)}{(y_k^3-1)}\log^{2h-1}|y_k|$ on both sides gives us
\begin{multline*}
a_{n+1,h-1}\delta^{2n+2-2h}=\sum_{l=h}^nc_{n,l-1}\delta^{2n-2l+1}\cdot\left(\frac{\delta^{2l-2h+1}}{\sqrt3}(2^{2l-2h+1}s_{2l-1,h}+q_{2l-1,h})\right)\\
+\sum_{l=h-1}^nd_{n,l}\delta^{2n-2l}\cdot\left(\delta^{2l-2h+2}(2^{2l-2h+2}r_{2l,h}+p_{2l,h})\right),
\end{multline*}
which means, for $n\ge0$,
\begin{equation}\label{eq:aij}
a_{n+1,h-1}=\frac{1}{\sqrt3}\sum_{l=h}^nc_{n,l-1}(2^{2l-2h+1}s_{2l-1,h}+q_{2l-1,h})+\sum_{l=h-1}^nd_{n,l}(2^{2l-2h+2}r_{2l,h}+p_{2l,h}).
\end{equation}
Similarly, comparing coefficients of $\frac{\log^{2h}|y_{k}|}{y_{k}^2+y_{k}+1}$, we obtain an expression for $b_{n+1,h}$. Since equations \eqref{eq:term1} and \eqref{eq:term2} have some additional constant terms corresponding to the case when $h=0$, we write them separately. When $h\ge1$, we have
\begin{multline*}
b_{n+1,h}\delta^{2n-2h+1}=\sum_{l=h}^nc_{n,l-1}\delta^{2n-2l+1}\cdot\left(\frac{\delta^{2l-2h}}{3}\cdot3(2^{2l-2h}r_{2l-1,h}+p_{2l-1,h})\right)\\
+\sum_{l=h-1}^nd_{n,l}\delta^{2n-2l}\cdot\left(\frac{\delta^{2l-2h+1}}{\sqrt3}\cdot(2^{2l-2h+1}s_{2l,h}+q_{2l,h}) \right),
\end{multline*}
giving us
\begin{equation}\label{eq:bij}
b_{n+1,h}=\sum_{l=h}^nc_{n,l-1}(2^{2l-2h}r_{2l-1,h}+p_{2l-1,h})+\frac1{\sqrt3}\sum_{l=h-1}^nd_{n,l}(2^{2l-2h+1}s_{2l,h}+q_{2l,h}).
\end{equation}
When $h=0$, we have
\begin{equation}\label{eq:bi0}
b_{n+1,0}=\frac13\sum_{l=1}^nc_{n,l-1}\Big(2^{2l}(2r_{2l-1,0}+y_{2l-1,0})+2p_{2l-1,0}+z_{2l-1,0}\Big)+\frac2{\sqrt3}\sum_{l=0}^nd_{n,l}(2^{2l+1}s_{2l,0}+q_{2l,0}).
\end{equation}
We can use the same procedure to obtain expressions for $c_{n,h-1}$ and $d_{n,h}$ by starting with the expression for $F(2n+1)$ for $n\ge1$ and expanding the integral from the penultimate iteration. This gives us the following for $h\ge1$ and $n\ge1$,
\begin{align}\label{eq:cij}
c_{n,h-1}=\frac1{\sqrt3}\sum_{l=h}^na_{n,l-1}(2^{2l-2h+1}s_{2l-1,h}+q_{2l-1,h})+\sum_{l=h-1}^{n-1}b_{n,l}(2^{2l-2h+2}r_{2l,h}+p_{2l,h}),
\end{align}
and 
\begin{align}
\label{eq:di0}d_{n,0}&=\frac13\sum_{l=1}^na_{n,l-1}\Big(2^{2l}(2r_{2l-1,0}+y_{2l-1,0})+2p_{2l-1,0}+z_{2l-1,0}\Big)+\frac2{\sqrt3}\sum_{l=0}^{n-1}b_{n,l}(2^{2l+1}s_{2l,0}+q_{2l,0}),\\
\label{eq:dij}d_{n,h}&=\sum_{l=h}^na_{n,l-1}(2^{2l-2h}r_{2l-1,h}+p_{2l-1,h})+\frac1{\sqrt3}\sum_{l=h}^{n-1}b_{n,l}(2^{2l-2h+1}s_{2l,h}+q_{2l,h}).
\end{align}
Finally, to complete the recursive formula, we must calculate initial values. To do this, we evaluate the Mahler measure of the first polynomial in the family as follows:
\begin{align*}
\m(Q_1)=\m\left(y+\left(\frac{\ol\omega z_1+\omega}{z_1+1}\right)\right)&=\frac{\sqrt3}{2\pi}\cdot F(1)\\
&=\frac{\sqrt3}{2\pi}\int_{-\infty}^\infty \m(P_{y_{1}})\frac{\dd {y_1}}{y_{1}^2+y_1+1}\\
&=\frac{\sqrt3}{2\pi}\cdot 2\cdot L(\chi_{-3},2)\left(1+\frac1{2^{2}}\right)=\frac{5\sqrt3}{4 \pi}\,L(\chi_{-3},2),
\end{align*}
using equation \eqref{eq:basecased}. This means that we can set $d_{0,0}=1$. From this base value, we can obtain all subsequent values of $a_{r,s}\,,b_{r,s}\,,c_{r,s},d_{r,s}$ using equations \eqref{eq:aij}-\eqref{eq:dij} above, with appropriate choices of $n$ and $h$. For example, with $n=0,h=1$ in \eqref{eq:aij} and $n=0$ in \eqref{eq:bi0}, we have $a_{1,0}=2$ and $b_{1,0}=\frac2{\sqrt3}$ respectively, giving us the Mahler measure of $Q_2$:
\[
\m(Q_2)=\frac{91}{18\pi^2}\zeta(3)+\frac{5}{4\sqrt3 \pi}\,L(\chi_{-3},2).
\] 
One can also confirm that these values of $a_{1,0}$ and $b_{1,0}$ agree with the results of Section \ref{sec:generalintegrals}. See Table \ref{table:arbitrary} for more examples of the Mahler measure evaluations.
\subsection{The final Mahler measure}
We are now ready to combine all the above details and complete the proof of our result.
\begin{proof}[Proof of Theorem \ref{thm:main}]
Observe that from the definition of $F(k)$ given in \eqref{eq:genfk}, the Mahler measure in \eqref{eq:mahlerint} is given by
\[
\m(Q_n)=\left(\frac{\sqrt3}{2\pi}\right)^n\cdot F(n).
\]
Thus,  plugging in equations \eqref{eq:neweven} and \eqref{eq:newodd} into the above and simplifying, we obtain for $n\ge1$
\begin{multline*}
\m(Q_{2n})=\frac{2}{12^n}\Bigg(\sum_{h=1}^n a_{n,h-1}\;9^h(2h)!\left(1-\frac1{3^{2h+1}}\right)\left(1-\frac1{2^{2h+1}}\right)\frac{\zeta(2h+1)}{\pi^{2h}}\\
+3\sum_{h=0}^{n-1} b_{n,h}\;9^h(2h+1)!\left(1+\frac1{2^{2h+2}}\right)\frac{L(\chi_{-3},2h+2)}{\pi^{2h+1}}\Bigg),
\end{multline*}
and for $n\ge0$,
\begin{multline*}
\m(Q_{2n+1})=\frac{1}{{12}^n\sqrt3}\Bigg(\sum_{h=1}^n c_{n,h-1}\;9^h(2h)!\left(1-\frac1{3^{2h+1}}\right)\left(1-\frac1{2^{2h+1}}\right)\frac{\zeta(2h+1)}{\pi^{2h}}\\
+3\sum_{h=0}^{n} d_{n,h}\;9^h(2h+1)!\left(1+\frac1{2^{2h+2}}\right)\frac{L(\chi_{-3},2h+2)}{\pi^{2h+1}}\Bigg),
\end{multline*}
where the coefficients $a_{r,s}\,,b_{r,s}\,,c_{r,s},d_{r,s}\in\Q(\sqrt3)$ are given recursively by \eqref{eq:aij}-\eqref{eq:dij} starting from the initial value $d_{0,0}=1$. We will see in Section \ref{subsec:derivatives}, that $a_{r,s},d_{r,s}\in\Q$ while $b_{r,s},c_{r,s}\in\sqrt3\cdot\Q$. Rearranging the powers, we can write the Mahler measures compactly as done in Theorem \ref{thm:main} as follows:
\begin{equation}\label{eq:maineven}
\m(Q_{2n})=\frac{2}{12^n}\Bigg(\sum_{h=1}^n  a_{n,h-1}\left(\frac{3}{\pi}\right)^{2h}\mathcal{C}(h)\;+\;\sum_{h=0}^{n-1}b_{n,h}\left(\frac{3}{\pi}\right)^{2h+1}\mathcal{D}(h)\Bigg),
\end{equation}
and for $n\ge0$ we have
\begin{equation}\label{eq:mainodd}
\m(Q_{2n+1})=\frac{1}{{12}^n}\Bigg(\frac1{\sqrt3}\sum_{h=1}^n  c_{n,h-1}\left(\frac{3}{\pi}\right)^{2h}\mathcal{C}(h)\;+\;\frac1{\sqrt3}\sum_{h=0}^{n}d_{n,h}\left(\frac{3}{\pi}\right)^{2h+1}\mathcal{D}(h)\Bigg),
\end{equation}
where
\begin{equation*}
\mathcal{C}(h)=(2h)!\left(1-\frac1{3^{2h+1}}\right)\left(1-\frac1{2^{2h+1}}\right)\zeta(2h+1),
\end{equation*}
and
\begin{equation*}
\mathcal{D}(h)=(2h+1)!\left(1+\frac1{2^{2h+2}}\right)L(\chi_{-3},2h+2).
\end{equation*}
\end{proof}
\subsection{More on the coefficients $a_{r,s}\,,b_{r,s}\,,c_{r,s}$ and $d_{r,s}$}\label{subsec:derivatives}
In this section, we show that the coefficients $a_{r,s}$ and $d_{r,s}$ are always rational, whereas the coefficients $b_{r,s}$ and $c_{r,s}$ are always a rational number multiplied by $\sqrt3$. This lets us express the Mahler measure $\m(Q_n)$ as a sum of rational multiples of derivatives of the Riemann zeta function and the $L$-function $L(\chi_{-3},s)$ evaluated at certain negative integers.
\begin{prop}\label{clm:rational}
Let $n\ge0$ be an integer. Then the coefficients $a_{r,s}\,,b_{r,s}\,,c_{r,s}$ and $d_{r,s}$ appearing in the recursive expressions \eqref{eq:aij}-\eqref{eq:dij} satisfy
\[
a_{n,j}, d_{n,k}\in\Q,
\]
and
\[
b_{n,j},c_{n,j}\in\sqrt3\cdot\Q,
\]
for all integers $0\le j\le n-1$ and $0\le k\le n$. 
\end{prop}
\begin{rem}
Following the discussions of Section \ref{sec:alternat} and Proposition \ref{prop:summary}, we can obtain an alternate expression for the coefficients $a_{r,s}\,,b_{r,s}\,,c_{r,s}$ and $d_{r,s}$ using polynomials that can be calculated explicitly (for instance, the polynomials $A_m(x)$ and $B_m(x)$ in equations \eqref{eq:am_expl} and  \eqref{eq:bm_expl} respectively). Since these polynomials have coefficients that are rational multiples of integer powers of $\sqrt3$, we can immediately deduce that $a_{r,s}\,,b_{r,s}\,,c_{r,s}$ and $d_{r,s}$ lie in $\Q\cup\sqrt3\cdot\Q$. Proposition \ref{clm:rational} makes this observation more precise. 
\end{rem}
\begin{proof}[Proof of Proposition]
We begin with the numbers $a_{r,s}$ and $b_{r,s}$ and proceed by induction. The base case is when $a_{1,0}=2$, which is clearly rational, and $b_{1,0}=\frac{2}{\sqrt3}=\frac23\sqrt3$, a rational multiple of $\sqrt3$, confirming the hypothesis for these cases. Now assume that for all $1\le n\le k$ and $0\le j\le n-1$, we have
\[
a_{n,j}\in\Q\quad\text{and}\quad b_{n,j}\in\sqrt3\cdot\Q.
\]
We consider $a_{n+1,j}$ and $b_{n+1,j}$ for $0\le j\le n$.  It can be observed that in each of the expressions \eqref{eq:aij}-\eqref{eq:dij}, the numbers $a_{r,s}\,,b_{r,s}\,,c_{r,s}$ and $d_{r,s}$ that appear inside the summations are multiplied by certain expressions involving rational coefficients of the polynomials $R_k,S_k,P_k,Q_k,Y_k$ and $Z_k$. For example, we have the terms
\[
(2^{2l-2h+1}s_{2l-1,h}+q_{2l-1,h})\quad\text{and}\quad(2^{2l-2h+2}r_{2l,h}+p_{2l,h}),
\]
in the two summations in equation \eqref{eq:aij}. These expressions are always rational, and for ease of notation, we will denote them by the symbol $\Delta\in\Q$ everywhere. Note that we may replace the expressions for the coefficients $c_{r,s}$ (from \eqref{eq:cij}), and for $d_{r,s}$ (from \eqref{eq:di0} and \eqref{eq:dij}) back into the expression for $a_{r,s}$ in \eqref{eq:aij}. This lets us write $a_{n+1,j}$ as a linear combination of $a_{n,k}$ and $b_{n,k}$ for $0\le k\le n-1$ as follows,
\begin{multline}
a_{n+1,j}=\frac{1}{\sqrt3}\sum_{l=j+1}^n(\Delta)\cdot\left(\frac{1}{\sqrt3}\sum_{i=l}^na_{n,i-1}(\Delta)+\sum_{i=l-1}^{n-1}b_{n,i}(\Delta)\right)\\
+\sum_{l=j}^n(\Delta)\cdot\left(\sum_{i=l}^na_{n,i-1}(\Delta)+\frac1{\sqrt3}\sum_{i=l}^{n-1}b_{n,i}(\Delta)\right),
\end{multline}
where, as discussed above, we combine all obviously rational expressions and replace them by the symbol $\Delta$. Distributing the products, this can be rewritten as
\begin{multline}
a_{n+1,j}=\frac{1}{3}\sum_{l=j+1}^n(\Delta)\cdot\left(\sum_{i=l}^na_{n,i-1}(\Delta)\right)+\frac{1}{\sqrt3}\sum_{l=j+1}^n(\Delta)\cdot\left(\sum_{i=l-1}^{n-1}b_{n,i}(\Delta)\right)\\
+\sum_{l=j}^n(\Delta)\cdot\left(\sum_{i=l}^na_{n,i-1}(\Delta)\right)+\frac1{\sqrt3}\sum_{l=j}^n(\Delta)\cdot\left(\sum_{i=l}^{n-1}b_{n,i}(\Delta)\right).
\end{multline}
Note that all sums involving the numbers $a_{n,i-1}$ have rational coefficients, whereas all sums involving $b_{n,i}$ have rational multiples of $\sqrt3$. Using the induction hypothesis, this means that $a_{n+1,j}\in\Q$, confirming the statement for $a_{r,s}$. A similar argument will show that $b_{n+1,j}\in\sqrt3\cdot\Q$, confirming the statement for $b_{r,s}$ and completing the induction procedure. The same process can be applied to the numbers $c_{r,s}$ and $d_{r,s}$ to conclude the proof of the proposition.
\end{proof}
%Following this claim, we can note that the coefficients of $\zeta(2h+1)$ appearing in the expressions for the Mahler measures in equations \eqref{eq:maineven} and \eqref{eq:mainodd}, are of the form 
%\[
%r\cdot\frac{\zeta(2h+1)}{\pi^{2h}},
%\]
%where $r$ is a rational number. Moreover, 

Following this proposition, we can rewrite the expressions for the Mahler measures in equations \eqref{eq:maineven} and \eqref{eq:mainodd}. First, using the functional equation of the Riemann zeta function, we can express its derivative at negative even integers in terms of its value at positive odd integers as follows
\begin{equation}\label{eq:zeta'}
\zeta'(-2h)=\frac{(-1)^h(2h)!}{2^{2h+1}}\cdot\frac{\zeta(2h+1)}{\pi^{2h}}.
\end{equation}
%Thus, the terms involving $\zeta(2h+1)$ in equations \eqref{eq:maineven} and \eqref{eq:mainodd} can be rewritten as rational multiples of $\zeta'(-2h)$.\\
%Similarly, the coefficients of $L(\chi_{-3},2h+2)$ appearing in these expressions are of the form
%\[
%\frac{s\sqrt3}{\pi^{2h+1}}L(\chi_{-3},2h+2),
%\]
%for a rational number $s$. 
Similarly, we can express derivatives of Dirichlet $L$-functions at negative odd integers in terms of its values at positive even integers. Indeed, the functional equation of the Dirichlet $L$-function $L(\chi,s)$ associated to an odd, primitive, non-principal Dirichlet character $\chi$ of conductor $q$ is given by (see \cite[Corollary 10.9]{montgomery})
\[
L(\chi,1-s)=2\left(\frac{q}{2\pi}\right)^s\,\frac{\tau(\chi)}{iq}\,\Gamma(s)\sin\left(\frac{\pi s}2\right)L(\ol\chi,s),
\]
where $\Gamma(s)$ is the Gamma function and $\tau(\chi)$ is the Gauss sum of $\chi$. Note that all functions appearing on the right-hand side are holomorphic on $\mathrm{Re}(s)>0$, and the function $\sin\left(\frac{\pi s}2\right)$ has zeros at even integers. Thus, taking derivatives on both sides of the functional equation, and evaluating at $s=2n$ for a positive integer $n$, we obtain
\[
-L'(\chi,1-2n)=\pi\left(\frac{q}{2\pi}\right)^{2n}\,\frac{\tau(\chi)}{iq}\,\Gamma(2n)\cos\left(\frac{\pi s}2\right)L(\ol\chi,2n).
\]
Taking $\chi=\chi_{-3}$ so that $\ol\chi=\chi,\;q=3$ and $\tau(\chi)=\sqrt3i$, we finally get
\begin{equation}\label{eq:L'}
L'(\chi_{-3},1-2n)=\frac{(-1)^{n+1}(2n-1)!\,3^{2n}}{2^{2n}\sqrt3}\cdot\frac{L(\chi_{-3},2n)}{\pi^{2n-1}}.
\end{equation} 
Plugging in \eqref{eq:zeta'} and \eqref{eq:L'} into equations \eqref{eq:maineven} and \eqref{eq:mainodd}, we obtain for $n\ge1$,
\begin{multline}\label{eq:dereven}
\m(Q_{2n})=\frac{2}{12^n}\Bigg(\sum_{h=1}^n  \frac{(-1)^ha_{n,h-1}}3\,(3^{2h+1}-1)(2^{2h+1}-1)\zeta'(-2h)\\
+\sum_{h=0}^{n-1}\frac{(-1)^hb_{n,h}}{\sqrt3}\,(2^{2h+2}+1)L'(\chi_{-3},-2h-1)\Bigg),
\end{multline}
and for $n\ge0$ we have
\begin{multline}\label{eq:derodd}
\m(Q_{2n+1})=\frac{1}{12^n}\Bigg(\sum_{h=1}^n \frac{(-1)^hc_{n,h-1}}{3\sqrt3}\,(3^{2h+1}-1)(2^{2h+1}-1)\zeta'(-2h)\\
+\sum_{h=0}^{n}\frac{(-1)^hd_{n,h}}{3}\,(2^{2h+2}+1)L'(\chi_{-3},-2h-1)\Bigg),
\end{multline}
where, as a result of Proposition \ref{clm:rational}, $a_{r,s},d_{r,s}\in\Q$ while $b_{r,s},c_{r,s}\in\sqrt3\cdot\Q$. Thus, the Mahler measure $\m(Q_n)$ is always a $\Q$-linear combination of derivatives of the Riemann zeta function evaluated at negative even integers and derivatives of the $L$-function $L(\chi_{-3},s)$ evaluated at negative odd integers. The first few examples of the Mahler measure computation have been listed in Table \ref{table:arbitrary}.

\section{Conclusion}
These results show that the techniques employed in \cite{nvar} can be extended to a more general family of polynomials. It is interesting to note that our formulae have a combination of zeta-values and $L$-values corresponding to the Dirichlet character of conductor 3 in both the odd and even cases.  The appearance of $L(\chi_{-3},s)$ is due to the evaluation of polylogarithms at the third and sixth root of unity, as opposed to the fourth root of unity in \cite{nvar} which yields values of $L(\chi_{-4},s)$. This motivates the question of whether these techniques can be further generalized to higher roots of unity to obtain $L$-values corresponding to higher conductors.  The integrals involving powers of the logarithm as seen in Section \ref{sec:generalintegrals} and the limits of the integrals in the iteration step discussed in \ref{sec:compute} give rise to some complicated calculations. It would be interesting to see if these calculations can be captured in a more general setting. In Section \ref{sec:alternat}, we saw an alternate method to evaluate integrals $f_1(m),f_2(m),g_1(m)$ and $g_2(m)$. This method resulted in polynomials that have explicit expressions, as compared to the recursive polynomials that arose from the first method. While the explicit expressions for the first two polynomials was obtained, it would be worthwhile to use the same techniques and obtain such expressions for all the remaining polynomials. Another intriguing question would be to investigate whether these explicit expressions enable us to write the coefficients  $a_{r,s}\,,b_{r,s}\,,c_{r,s}$ and $d_{r,s}$ in an elegant closed formula, like the coefficients that appear in \cite{nvar}. \\

In Section \ref{sec:generalintegrals}, we used the fact that the polynomial $P_\gamma(y)$ has Mahler measure $\log^+|\gamma|$, and is thus dependent only on the absolute value of $\gamma$. In \cite{nvar}, other families of polynomials were also studied -- those that stemmed from simpler polynomials (like $P_\gamma(y)$) whose Mahler measures only depended on the absolute value of a parameter. For instance, we have the following identities:
\[
\m(1+\gamma x+(1-\gamma)y)=|\mathrm{arg}\gamma|\log|1-\gamma|+|\mathrm{arg}(1-\gamma))|\cdot\log|\gamma|+\left\{\begin{array}{rl}
D(\gamma) & \text{if Im}(\gamma)\ge0,\\
D(\ol{\gamma}) & \text{if Im}(\gamma)<0,
\end{array}\right.
\]
which follow from the Cassaigne-Maillot formula (see  \cite[Proposition 4.3]{BrunaultZudilin}) where $D(z)$ is the Bloch-Wigner dilogarithm (see \cite{Zagier2007}); and 

\[
\m(1+x+\gamma(1+y)z)=\left\{\begin{array}{ll}
\frac{2}{\pi^2}\mathcal{L}_3(|\gamma|) & \text{for }|\gamma|\le1,\\
\log|\gamma|+\frac{2}{\pi^2}\mathcal{L}_3(|\gamma|^{-1}) & \text{for }|\gamma|>1,\\
\end{array}\right.
\]
which was proved by Vandervelde in \cite{Vandervelde} (see also \cite[Theorem 17]{L}), where
\[
\mathcal L_3(\gamma)=-2\int_0^\gamma\frac{\dd s}{s^2-1}\circ\frac{\dd s}s\circ\frac{\dd s}s.
\] 
It would be worth exploring whether similar families can be constructed from such $P_\gamma$ and their Mahler measures studied in context of our results.  

Recently, the method of Lalín was extended to certain families of $n$-variable polynomials with non-linear degree in \cite{LNR}, by evaluating the Mahler measure of rational functions $S_{n,r}$ of the form 
\[
S_{n,r}(x,y,z,x_1,\dots,x_n)=(1+x)z +\left[\left(\frac{1-x_1}{1+x_1}\right)
\cdots \left( \frac{1-x_{n}}{1+x_{n}}\right)\right]^r(1+y),
\] 
for any integer $r\ge1$. Interestingly, considering such a family under the construction mentioned in this article leads to trivial evaluations. Indeed, if we let
\[Q_{n,r}(z_1,\dots, z_n,y):=y+\left[\left(\frac{z_1+\omega}{z_1+1}\right)\cdots\left(\frac{z_n+\omega}{z_n+1}\right)\right]^r,\]
then, we trivially have that 
\begin{align*}\m(Q_{n,r}(z_1,\dots, z_n,y))=&\m(Q_{n,r}(z_1,\dots, z_n,-y^r))=\sum_{j=0}^r\m\left( y-\xi_r^j\left(\frac{z_1+\omega}{z_1+1}\right)\cdots\left(\frac{z_n+\omega}{z_n+1}\right)\right)\\=&r\m(Q_{n,1}(z_1,\dots, z_n,y))=r\m(Q_n(z_1,\dots, z_n,y)),\end{align*}
where $\xi_r$ is an $r^\text{th}$ root of unity. Thus, the Mahler measure of this new family is simply $r$ times that of the family $Q_n$ considered in this article. However, it will be interesting to explore the following family of rational functions
\[T_{n,r}(z_1,\dots, z_n,x,y):=1 + \left[\left(\frac{z_1+\omega}{z_1+1}\right)\cdots\left(\frac{z_n+\omega}{z_n+1} \right)\right]^rx + \left( 1 -\left[\left(\frac{z_1+\omega}{z_1+1}\right)\cdots\left(\frac{z_n+\omega}{z_n+1}\right)\right]^r \right)y.\]

Finally, we note that in \cite{LN2023}, certain transformations are described which when applied on a polynomial, preserve the Mahler measure. In Section 4 of the same article, these transformations are applied on the polynomials appearing in \cite{nvar} to obtain infinitely many new rational functions with the same Mahler measure. We remark that one can apply these transformations to the family of polynomials considered in Theorem \ref{thm:main} as well, and obtain many more polynomials with the same Mahler measure. 

\bibliographystyle{amsplain}
\bibliography{Bibliography}

\providecommand{\bysame}{\leavevmode\hbox to3em{\hrulefill}\thinspace}
\providecommand{\MR}{\relax\ifhmode\unskip\space\fi MR }
\providecommand{\MRhref}[2]{%
  \href{http://www.ams.org/mathscinet-getitem?mr=#1}{#2}
}
\providecommand{\href}[2]{#2}
\begin{thebibliography}{10}

\bibitem{ahlfors}
Lars~V. Ahlfors, \emph{Complex analysis}, third ed., International Series in
  Pure and Applied Mathematics, McGraw-Hill Book Co., New York, 1978, An
  introduction to the theory of analytic functions of one complex variable.
  \MR{510197}

\bibitem{B1}
David~W. Boyd, \emph{Speculations concerning the range of {M}ahler's measure},
  Canad. Math. Bull. \textbf{24} (1981), no.~4, 453--469. \MR{644535}

\bibitem{Bo98}
\bysame, \emph{Mahler's measure and special values of {$L$}-functions},
  Experiment. Math. \textbf{7} (1998), no.~1, 37--82. \MR{1618282}

\bibitem{BrunaultZudilin}
Fran\c{c}ois Brunault and Wadim Zudilin, \emph{Many variations of {M}ahler
  measures---a lasting symphony}, Australian Mathematical Society Lecture
  Series, vol.~28, Cambridge University Press, Cambridge, 2020. \MR{4382435}

\bibitem{Deninger}
Christopher Deninger, \emph{Deligne periods of mixed motives, {$K$}-theory and
  the entropy of certain {${\bf Z}^n$}-actions}, J. Amer. Math. Soc.
  \textbf{10} (1997), no.~2, 259--281. \MR{1415320}

\bibitem{Goncharov}
A.~B. Goncharov, \emph{Geometry of configurations, polylogarithms, and motivic
  cohomology}, Adv. Math. \textbf{114} (1995), no.~2, 197--318. \MR{1348706}

\bibitem{Gon95}
\bysame, \emph{Polylogarithms in arithmetic and geometry}, Proceedings of the
  {I}nternational {C}ongress of {M}athematicians, {V}ol. 1, 2 ({Z}\"{u}rich,
  1994), Birkh\"{a}user, Basel, 1995, pp.~374--387. \MR{1403938}

\bibitem{Gon-arxiv}
\bysame, \emph{{Multiple polylogarithms and mixed Tate motives}}, arXiv
  Mathematics e-prints (2001), math/0103059.

\bibitem{goncharov2022motivic}
A.~B. Goncharov and Daniil {Rudenko}, \emph{{Motivic correlators, cluster
  varieties and Zagier's conjecture on $\zeta(F,4)$}}, arXiv e-prints (2018),
  arXiv:1803.08585.

\bibitem{L}
Matilde~N. Lal\'in, \emph{Some examples of {M}ahler measures as multiple
  polylogarithms}, J. Number Theory \textbf{103} (2003), no.~1, 85--108.
  \MR{2008068}

\bibitem{nvar}
\bysame, \emph{Mahler measure of some {$n$}-variable polynomial families}, J.
  Number Theory \textbf{116} (2006), no.~1, 102--139. \MR{2197862}

\bibitem{LL}
Matilde~N. Lal\'{\i}n and Jean-S\'{e}bastien Lechasseur, \emph{Higher {M}ahler
  measure of an {$n$}-variable family}, Acta Arith. \textbf{174} (2016), no.~1,
  1--30. \MR{3517530}

\bibitem{LN2023}
Matilde~N. Lal\'{\i}n and Siva~Sankar Nair, \emph{An invariant property of
  {M}ahler measure}, Bull. Lond. Math. Soc. \textbf{55} (2023), no.~3,
  1129--1142. \MR{4599103}

\bibitem{LNR}
Matilde~N. Lal\'in, Siva~Sankar Nair, and Subham Roy, \emph{The {M}ahler
  measure of a multivariate polynomial family with non-linear degree}, Acta
  Arith. \textbf{216} (2024), no.~1, 35--61. \MR{4801208}

\bibitem{Le}
D.~H. Lehmer, \emph{Factorization of certain cyclotomic functions}, Ann. of
  Math. (2) \textbf{34} (1933), no.~3, 461--479. \MR{1503118}

\bibitem{Mah}
K.~Mahler, \emph{On some inequalities for polynomials in several variables}, J.
  London Math. Soc. \textbf{37} (1962), 341--344. \MR{0138593}

\bibitem{keckic}
Dragoslav~S. Mitrinovi\'c and Jovan~D. Ke\v{c}ki\'c, \emph{The {C}auchy method
  of residues}, Mathematics and its Applications (East European Series),
  vol.~9, D. Reidel Publishing Co., Dordrecht, 1984, Theory and applications,
  Translated from the Serbian by Ke\v cki\'c. \MR{754560}

\bibitem{montgomery}
Hugh~L. Montgomery and Robert~C. Vaughan, \emph{Multiplicative number theory.
  {I}. {C}lassical theory}, Cambridge Studies in Advanced Mathematics, vol.~97,
  Cambridge University Press, Cambridge, 2007. \MR{2378655}

\bibitem{RZ14}
Mathew Rogers and Wadim Zudilin, \emph{On the {M}ahler measure of
  {$1+X+1/X+Y+1/Y$}}, Int. Math. Res. Not. IMRN (2014), no.~9, 2305--2326.
  \MR{3207368}

\bibitem{lowell}
Lowell Schoenfeld, \emph{The evaluation of certain definite integrals}, SIAM
  Rev. \textbf{5} (1963), 358--369. \MR{159932}

\bibitem{S1}
Chris~J. Smyth, \emph{On measures of polynomials in several variables}, Bull.
  Austral. Math. Soc. \textbf{23} (1981), no.~1, 49--63. \MR{615132}

\bibitem{smyth2003explicit}
\bysame, \emph{Explicit formulas for the {M}ahler measures of families of
  multivariable polynomials}, unpublished notes, 2003.

\bibitem{Vandervelde}
S.~Vandervelde, \emph{A formula for the {M}ahler measure of {$axy+bx+cy+d$}},
  J. Number Theory \textbf{100} (2003), no.~1, 184--202. \MR{1971253}

\bibitem{Zagier1991}
Don Zagier, \emph{Polylogarithms, {D}edekind zeta functions and the algebraic
  {$K$}-theory of fields}, Arithmetic algebraic geometry ({T}exel, 1989),
  Progr. Math., vol.~89, Birkh\"{a}user Boston, Boston, MA, 1991, pp.~391--430.
  \MR{1085270}

\bibitem{Zagier2007}
\bysame, \emph{The dilogarithm function}, Frontiers in number theory, physics,
  and geometry. {II}, Springer, Berlin, 2007, pp.~3--65. \MR{2290758}

\end{thebibliography}
\end{document}